\documentclass[10pt,reqno]{amsart}
\usepackage{amsmath,amsthm,amssymb,mathrsfs,amsfonts,}
\usepackage{enumerate}
\usepackage{graphicx}

\usepackage{color}
\usepackage{xcolor} % A package to add color.
\usepackage[letterpaper, margin=1in]{geometry}
\usepackage[titletoc,title]{appendix}

\usepackage{slashed}
\usepackage{cancel}
\usepackage{cite}
\usepackage[super]{nth}
\usepackage[linktocpage=true]{hyperref}
\usepackage{marginnote}
\usepackage{verbatim}

\numberwithin{equation}{section}
\newtheorem{thm}{Theorem}[section]
\newtheorem{lem}[thm]{Lemma}
\newtheorem{prop}[thm]{Proposition}

\newtheorem{cor}[thm]{Corollary}
\newtheorem{rem}[thm]{Remark}

\theoremstyle{definition}

\newcommand{\norm}[1]{\Vert#1\Vert}
\newcommand{\abs}[1]{\vert#1\vert}
\newcommand{\jb}[1]{\langle #1\rangle}
\newcommand{\angles}[2]{\langle #1,#2\rangle}
\newcommand{\BR}{\mathbb{R}}
\newcommand{\BS}{\mathbb{S}}
\newcommand{\BN}{\mathbb{N}}
\newcommand{\BZ}{\mathbb{Z}}
\newcommand{\BFq}{\mathbf{q}}

\newcommand{\ep}{\varepsilon}

\newcommand{\Romanupper}[1]
{\MakeUppercase{\romannumeral #1}}

\begin{document}

\title{Scattering from infinity of the Maxwell Klein Gordon Equations in Lorenz Gauge}

\author{Lili He}
\address{Department of Mathematics \\ Johns Hopkins University \\ Baltimore, MD 21218, USA}
\email{lhe31@jhu.edu}
\thanks{}

\begin{abstract}
We prove global existence backwards from the scattering data posed at infinity for the Maxwell Klein Gordon equations in Lorenz gauge satisfying the weak null condition. The asymptotics of the solutions to the Maxwell Klein Gordon equations in Lorenz gauge were shown to be wave like at null infinity and homogeneous towards timelike infinity in ~\cite{CKL19} and expressed in terms of radiation fields, and thus our scattering data will be given in the form of radiation fields in the backward problem.
We give a refinement of the asymptotics results in ~\cite{CKL19}, and then making use of this refinement, we find a global solution which attains the prescribed scattering data at infinity. Our result corresponds to ``existence of scattering states" in the scattering theory. Our work starts from the approach in \cite{LV17} and is more delicate since it involves nonlinearities with fewer derivatives. The method of proof relies on a suitable construction of the approximate solution from the scattering data, a weighted conformal Morawetz energy estimate and a spacetime version of Hardy inequality.
\end{abstract} 

\maketitle 
\tableofcontents

\addtocontents{toc}{\protect\setcounter{tocdepth}{1}}

\section{Introduction}

In this paper we study the backward problem for the Maxwell Klein Gordon (MKG) equations on $\BR^{1+3}$. The MKG equations describe the interaction between an electromagnetic field represented by a two form $F_{\alpha\beta}$ and a complex-valued scalar field $\phi$, and read
\begin{equation}\label{eq:mkg}
\begin{split}
D^\alpha D_\alpha \phi &=0\\
\partial^\beta F_{\alpha\beta} &=J_\alpha
\end{split}
\end{equation}
where $D_\alpha=\partial_\alpha+iA_\alpha$ is the covariant derivative corresponding to a real-valued gauge potential $A_\alpha$ which represents the electromagnetic field $F_{\alpha\beta}=\partial_\alpha A_\beta-\partial_\beta A_\alpha$. To complete the coupling of $F$ and $\phi$, the current is defined as $J_\alpha=\Im(\phi \overline{D_\alpha \phi})$ where $\Im(z)$ denotes the imaginary part of $z$, and satisfies the conservation law $\partial^\alpha J_\alpha=0$ .
Here we write $\partial_\alpha=\partial/\partial x^\alpha$, $x^0=t$, $\partial_0=\partial_t$ and use the Einstein summation convention with Greek indices $\mu, \nu,\dots$ summed from $0$ to $3$, and lower case Roman indices $j,k,\dots$ summed from $1$ to $3$. Indices are raised and lowered using the Minkowski metric $m=\text{diag }(-1, 1, 1, 1)$.

%The energy-momentum tensor of these equations is given by
%\[
%Q_{\alpha\beta}=\Re(D_\alpha\phi\overline{D_\beta\phi})-\frac{1}{2}m_{\alpha\beta}D^\gamma\phi\overline{D_\gamma\phi}+\frac{1}{2}F_{\alpha\gamma}F_\beta^{ \,\gamma}+\frac{1}{2}{}^\star F_{\alpha\gamma}{}^\star F_\beta^{\,\gamma}\]
%where $\Re(z)$ is the real part of $z$, ${}^\star F$ denotes the hodge star operator of $F$ which is given by ${}^\star F_{\alpha\beta}=\frac{1}{2}\epsilon_{\alpha\beta\mu\nu}F^{\mu\nu}$ and $\epsilon_{\alpha\beta\mu\nu}$ is the volume form on Minkowski space.

The charge of these equations is given by
\begin{equation}\label{eq:def_of_q_for}
\BFq=\int_{\BR^3}\!J_0 \, dx=\int_{\BR^3} \!\Im(\phi\overline{D_0 \phi}) \, dx.
\end{equation}
It is easy to verify that the charge is conserved.

\medskip

We can rewrite the equations ~\eqref{eq:mkg} in the unknowns $\phi$ and $A_\alpha$ as
\begin{equation}\label{eq:mkg_in_A}
\begin{split}
\Box \phi &=-2iA^\alpha\partial_\alpha\phi+A^\alpha A_\alpha\phi-i\lambda\phi\\
\Box A_\alpha &=-J_\alpha+\partial_\alpha\lambda
\end{split}
\end{equation}
where $\Box=\partial^\alpha\partial_\alpha=-\partial_t^2+\Delta$ and $\lambda=\partial^\alpha A_\alpha$.
The system of equations ~\eqref{eq:mkg_in_A} does not uniquely determine $\phi$ and $A$. In fact, MKG system is invariant under the gauge transformation
\[
\phi\to\tilde{\phi}=e^{i\psi}\phi,\quad A_\alpha\to\tilde{A}_\alpha=A_\alpha-\partial_\alpha\psi
\]
for any real-valued function $\psi$. In other words, if $(\phi, A_\alpha)$ satisfies ~\eqref{eq:mkg_in_A}, then so does $(\tilde{\phi}, \tilde{A}_\alpha)$. Thus, we have gauge freedom, that is, we may stipulate a condition that the gauge potential  should satisfy. In this work we fix the gauge by imposing the Lorenz gauge condition
\begin{equation}\label{eq:lorenz}
\lambda=\partial^\alpha A_\alpha=0.
\end{equation}
Now we define the reduced MKG equations to be 
\begin{equation}\label{eq:rmkg}
\begin{split}
\Box \phi &=-2iA^\alpha\partial_\alpha\phi+A^\alpha A_\alpha\phi\\
\Box A_\alpha &=-J_\alpha.
\end{split}
\end{equation}
Then we note that if the Lorenz gauge condition ~\eqref{eq:lorenz} holds, the MKG equations ~\eqref{eq:mkg_in_A} become the reduced MKG equations ~\eqref{eq:rmkg} which is a system of nonlinear wave equations. A key observation is that if $(\phi, A)$ solves ~\eqref{eq:rmkg}, then $\lambda$ satisfies a wave equation
\begin{equation}\label{eq:eq_for_lorenz}
\Box \lambda=\abs{\phi}^2\lambda.
\end{equation}
In view of ~\eqref{eq:eq_for_lorenz}, we see that if the solution to ~\eqref{eq:rmkg} satisfies the Lorenz gauge condition ~\eqref{eq:lorenz} at $t=0$, then ~\eqref{eq:lorenz} holds for all times.

For the study of global existence for MKG equations, first the pioneering works ~\cite{EM82, EM82_2} of Eardly-Moncrief established the global existence result for the Yang-Mills-Higgs equations with sufficiently smooth initial data. Later the result was refined by Klainerman-Machedon in ~\cite{KM94} for the MKG equations with finite energy. 
 
In this paper, we focus on the scattering theory for MKG equations in Lorenz gauge \eqref{eq:rmkg}. In scattering theory for ~\eqref{eq:rmkg}, once we formulate a suitable notion of scattering data, we shall ask the following questions
\begin{itemize}
	\item[(\romannumeral1)] \emph{Existence of scattering states}: for every scattering data, does there exist a global solution to \eqref{eq:rmkg} which has the given scattering data?
	\item[(\romannumeral2)] \emph{Uniqueness of scattering states}: must two global solutions corresponding to the same scattering data be the same?
	\item[(\romannumeral3)]\emph{Asymptotic completeness}: do the above solutions determined by scattering data include all the global solutions to ~\eqref{eq:rmkg}?
\end{itemize}  
We refer to ~\cite{RS3, Yafaev1, Yafaev2} for a general introduction to scattering theory.

In ~\cite{CKL19} Candy-Kauffman-Lindblad formulated a suitable notion of scattering data in terms of radiation field and gave an affirmative answer to question (\romannumeral3) for small initial data solution. They showed that ~\eqref{eq:rmkg} satisfies the weak null condition of Lindblad-Rodnianski ~\cite{LR03, LR05, LR10} and established the detailed asymptotics of the field and the gauge potential. These asymptotics have two parts and can be expressed by the scattering data, one wave like along outgoing light cones at null infinity, and one homogeneous inside the light cone at timelike infinity, and the charge imposes an oscillating factor in the asymptotic behavior for the scalar field.

In this work, we first establish a refinement of the asymptotics obtained in ~\cite{CKL19}. Then making use of the asymptotics of the new form, we answer question (\romannumeral1) in the affirmative for small scattering data. One inspiration for our work is the work of Lindblad-Schlue ~\cite{LV17} where they studied the backward problem for models of Einstein equations by finding a solution that has the asymptotics of the form given in ~\cite{Lind17}.

\subsection{Statement of the main results}

To state our main theorem, we first define some necessary notations and review the asymptotics results in ~\cite{CKL19}. We use the Cartesian coordinates $\{x^0=t, x^1, x^2, x^3\}$ of Minkowski space as well as the null coordinates \[
q=r-t, \quad p=r+t.
\]
We define a null frame $\{L, \underline{L}, e_1, e_2\}$ as
\begin{equation}\label{eq:nullframe}
L=\partial_t+\partial_r=\partial_t+\omega^i\partial_i=L^\mu\partial_\mu, \quad \underline{L}=\partial_t-\partial_r=\partial_t-\omega^i\partial_i=\underline{L}^\mu\partial_\mu, \quad e_B=\omega_B^\mu\partial_\mu
\end{equation}
where $r=\abs{x}$, $w^i=x^i/r$ and $\{e_B\}_{B=1, 2}$ is an orthonormal basis on each sphere in $\BR_x^3$ for constant $t$. The null decomposition of the  gauge potential $A_\alpha$ can be expressed as follows
\[
A_L=L^\mu A_\mu, \quad A_{\underline{L}}=\underline{L}^\mu A_\mu, \quad A_{e_B}=\omega_B^\mu A_\mu.
\]

\subsubsection{Asymptotics Results}

In ~\cite{CKL19} the asymptotic system, which was introduced by H\"{o}rmander in ~\cite{Horm87} to study finite time blow-up and later used by Lindblad in ~\cite{Lind92} to study the global solutions of nonlinear wave equations, was obtained by plugging the expansion 
\[
\phi(t, r\omega)\sim\frac{\ep\Phi_0(q, s, \omega)}{r}, \quad A_\alpha(t, r\omega)\sim\frac{\ep\mathcal{A}_\alpha(q, s, \omega)}{r}\quad\text{where}\quad s=\ep\ln r, \quad\omega=x/r.
\]
into ~\eqref{eq:rmkg} and equating the powers of order $\ep^2/r^2$
\begin{equation}\label{eq:asymtotic_system}
\begin{split}
\partial_s\partial_q \mathcal{A}_\alpha&=-\frac{L_\alpha}{2}\Im(\Phi_0\overline{\partial_q\Phi_0})\\
\partial_s\partial_q\Phi_0&=-i\mathcal{A}_L\partial_q\Phi_0.
\end{split}
\end{equation}
It was shown in ~\cite{CKL19} that the equations ~\eqref{eq:rmkg} satisfy the weak null condition since the corresponding asymptotic system given above has a global solution and the solution together with its derivatives grows at most exponentially in $s$.
It was then proved that in the Lorenz gauge if the initial data $(\phi, D_t\phi, A_j, \partial_tA_j)|_{t=0}$ lies in a weighted Sobolev space, then for each $q\in\BR, \omega=x/r\in \BS^2$, the scalar field is asymptotically 
\begin{equation}\label{eq:asym_of_phi}
\phi(t, r\omega)\sim e^{-i\frac{\BFq}{4\pi}\ln(1+r)}\frac{\Phi(q, \omega)}{r}\quad\text{where}\quad\abs{\Phi}\leq\ep\jb{q}^{-\gamma}
\end{equation}
for some $1/2<\gamma<1$, $0<\ep\ll1$ and is concentrated close to the light cone $t-r$ constant. And asymptotically the null decomposition of the gauge potential is
\begin{gather*}
A_L(t, r\omega)\sim\frac{\BFq}{4\pi r}, \quad A_{e_B}(t, r\omega)\sim\frac{\mathcal{A}_{e_B}(q, \omega)}{r}, \\ A_{\underline{L}}(t, r\omega)\sim\frac{1}{2r}\ln\frac{\jb{t+r}}{\jb{t-r}}\int_{q}^\infty\!\mathcal{J}_{\underline{L}}(\eta, \omega)\, d\eta+\frac{\mathcal{A}_{\underline{L}}(q, \omega)}{r}.
\end{gather*}
where \[
\abs{\mathcal{A}_{e_B}}, \abs{\mathcal{A}_{\underline{L}}}\leq\ep\jb{q_+}^{-\gamma},\quad \mathcal{J}_{\underline{L}}=-2\Im\bigl(\Phi(q, \omega)\overline{\partial_q\Phi(q, \omega)}\bigr)
\] 
with $\jb{q_+}=(1+q_+^2)^{1/2},\, q_+=\max\{q, 0\}$. 
Here we notice that the radiation field for the gauge potential does not decay in the interior of the light cone. To obtain more precise asymptotic behavior, Candy-Kauffman-Lindblad in ~\cite{CKL19} subtracted a better approximation using the formulas from ~\cite{Lind90} and thus concluded that asymptotically the gauge potential is 
\begin{equation}\label{eq:asym_of_A}
A_\alpha(t, r\omega)\sim\frac{\mathcal{A}_\alpha(q, \omega)}{\jb{t+r}}+\frac{\mathcal{K}_\alpha(q, \omega, \frac{\jb{t+r}}{\jb{t-r}})}{\jb{t+r}}.
\end{equation} 
Here, $\mathcal{A}_\alpha$ is concentrated close to the light cones $r-t$ constant, and $\mathcal{K}_\alpha$ is homogeneous of degree $0$ with a logarithmic loss close to the light cone. More precisely, we have
\begin{equation}\label{eq:decay_of_A}
	\begin{split}
\abs{\mathcal{A}_\alpha(q, \omega)} &\leq\ep\jb{q}^{-\gamma},\\
\abs{\mathcal{K}_\alpha(q, \omega, s)} &\leq\ep^2\ln\abs{s}.
	\end{split}
\end{equation}
$\mathcal{A}_\alpha$ is the radiation field for the free wave operator. $\mathcal{K}_\alpha$ has two parts, the backscattering of the wave operator with the quadratic source term $\mathcal{J}_\alpha$, and a term coming from the long range effect of the charge $\BFq$.  
 In particular, in the wave zone $\abs{t-r}\ll t+r$
\begin{equation}\label{eq:K_in_exterior}
A_{\alpha}(t, r\omega)\sim\frac{\mathcal{A}_\alpha(q, \omega)}{\jb{t+r}}+\frac{1}{2r}\ln\frac{\jb{t+r}}{\jb{t-r}}\int_{q}^\infty\!\mathcal{J}_{\alpha}(\eta, \omega)\, d\eta+\frac{\BFq}{4\pi r}\chi_{ex}(q)\delta_{\alpha 0}
\end{equation}
where $\mathcal{J}_\alpha(q, \omega)=L_\alpha(\omega)\Im\bigl(\Phi(q, \omega)\overline{\partial_q\Phi(q, \omega)}\bigr)$, $\chi_{ex}$ is a smooth cutoff such that $\chi_{ex}(q)=1$ if $q\geq 1$, and $\chi_{ex}(q)=0$ if $q\leq 1/2$. We notice that the integral in the second term of the right-hand side in ~\eqref{eq:K_in_exterior} does not decay in the interior and we cannot use the weighted conformal Morawetz energy estimate to estimate it. So we have to subtract off a leading term in the interior.

In ~\cite{CKL19}, it was also shown that at timelike infinity in the interior, say $\abs{x}<t$, we have given $y\in\BR^3$ with $\abs{y}<1$
\begin{equation}\label{eq:timelike_asym}
A_\alpha(t, ty)\sim
\frac{1}{4\pi t}\int_{-\infty}^\infty\!\int_{\BS^2}\!\frac{\mathcal{J}_{\alpha}(\eta, \sigma)}{1-\angles{y}{\sigma}}\, dS(\sigma)\,d\eta.
\end{equation}
Then by re-examining the proof of asymptotics of $A_\alpha$ ~\eqref{eq:timelike_asym} at timelike infinity in ~\cite{CKL19}, we refine the asymptotics for $A_\alpha$ in the interior wave zone, see Proposition ~\ref{prop:refined_asymptotics}. In fact, combining ~\eqref{eq:timelike_asym} and our refinement, we conclude that the asymptotics of $A_\alpha$ take the form
\begin{align}\label{eq:K_ex}
A_{\alpha}(t, r\omega)&\sim\frac{\mathcal{A}_\alpha}{r}+\frac{1}{2r}\ln\frac{\jb{t+r}}{\jb{t-r}}\int_{q}^\infty\!\mathcal{J}_{\alpha}(\eta, \omega)\, d\eta+\frac{\BFq}{4\pi r}\chi_{ex}(q)\delta_{\alpha 0}\quad\text{at null infinity with } q\geq -1,\\
\label{eq:K_in}
\begin{split}
A_\alpha(t, r\omega)
&-\frac{1}{4\pi}\int_{-\infty}^\infty\!\int_{\BS^2}\!\frac{\mathcal{J}_{\alpha}(\eta, \sigma)}{t-r\angles{\omega}{\sigma}}\, dS(\sigma)\,d\eta\\
&\sim\frac{\mathcal{A}_\alpha}{r}-\frac{1}{2r}\ln\frac{\jb{t+r}}{\jb{t-r}}\int_{-\infty}^q\!\mathcal{J}_{\alpha}(\eta, \omega)\, d\eta\quad\text{at interior null infinity }q\leq-1,
\end{split}\\
\label{eq:K_timelike}
A_\alpha(t, ty)&\sim
\frac{1}{4\pi t}\int_{-\infty}^\infty\!\int_{\BS^2}\!\frac{\mathcal{J}_{\alpha}(\eta, \sigma)}{1-\angles{y}{\sigma}}\, dS(\sigma)\,d\eta\quad \text{if} \quad t\gg r.
\end{align}

Combining ~~\eqref{eq:K_ex} and ~\eqref{eq:K_in}, we find that after subtracting a leading term in the interior, $A_\alpha$ has asymptotics at null infinity in the sense that the radiation field and the logarithm term have decay in $q$ both in the exterior and interior. Then we may express the asymptotics of $A_\alpha$ after subtracting a leading term in the interior in a smoothed out version in the wave zone
\begin{equation}\label{eq:unified_K}
\begin{split}
A_\alpha(t, r\omega)&-
\frac{1}{4\pi}\int_{-\infty}^\infty\!\int_{\BS^2}\!\frac{\mathcal{J}_{\alpha}(\eta, \sigma)}{t-r\angles{\omega}{\sigma}}\, dS(\sigma)\,dq\chi_{in}(q)\\
&\sim\frac{\mathcal{A}_\alpha(q, \omega)}{r}-\frac{1}{2r}\ln\frac{\jb{t+r}}{\jb{t-r}}\int_{-\infty}^q\!\mathcal{J}_{\alpha}(\eta, \omega)\, d\eta\chi_{in}(q)\\
&\quad+\Bigl(\frac{1}{2r}\ln\frac{\jb{t+r}}{\jb{t-r}}\int_{q}^\infty\!\mathcal{J}_{\alpha}(\eta, \omega)\, d\eta\Bigr)\bigl(1-\chi_{in}(q)\bigr)+\frac{\BFq}{4\pi r }\chi_{ex}(q)\delta_{\alpha 0}.
\end{split}
\end{equation}
where $\chi_{in}$ is a smooth cutoff such that $\chi_{in}(q)=1$ if $q\leq -1$, and $\chi_{in}(q)=0$ if $q\geq -1/2$. Here the smoothed version of the asymptotics~\eqref{eq:unified_K} at null infinity depends on the cutoff $\chi_{in}$. This is not a big deal as this ambiguity just amounts to a radiation field with sufficient decay in $q$.

\subsubsection{Asymptotic Lorenz gauge condition}

Another key conclusion in ~\cite{CKL19} is that asymptotically the Lorenz gauge condition can be written as
\[
\partial_q(rA_L)\sim 0
\]
at null infinity. More rigorously, in view of ~\eqref{eq:asym_of_phi} and ~\eqref{eq:unified_K} we see that in the wave zone every set of asymptotics $(\phi_{asy}, A_{\alpha, asy})$ of the solution to the MKG equations ~\eqref{eq:mkg_in_A} is expressed in terms of a radiation filed set $(\Phi, \mathcal{A}_\alpha)$, which is exactly the scattering data we impose in the backward problem, as follows
\begin{equation}\label{eq:asym_of_two}
\begin{split}
\phi_{asy}&=e^{-i\frac{\BFq}{4\pi}\ln r}\frac{\Phi(q,\omega)}{r},\\
A_{\alpha, asy}&=\frac{1}{4\pi}\int_{-\infty}^\infty\!\int_{\BS^2}\!\frac{\mathcal{J}_{\alpha}(\eta, \sigma)}{t-r\angles{\omega}{\sigma}}\, dS(\sigma)\,dq\chi_{in}(q)+\text{the right-hand side of ~\eqref{eq:unified_K}}.
%\frac{\mathcal{A}(q, \omega)}{r}-\frac{1}{2r}\ln\frac{\jb{t+r}}{\jb{t-r}}\int_{-\infty}^q\!\mathcal{J}_{\alpha}(\eta, \omega)\, d\eta\chi_{in}(q)\\
%&\quad+\Bigl(\frac{1}{2r}\ln\frac{\jb{t+r}}{\jb{t-r}}\int_{q}^\infty\!\mathcal{J}_{\alpha}(\eta, \omega)\, d\eta\Bigr)\bigl(1-\chi_{in}(q)\bigr)+\frac{\BFq}{4\pi r }\chi_{ex}(q)\delta_{\alpha 0}.
\end{split}
\end{equation}
 Since we discuss the asymptotic form of the Lorenz condition, we focus on the region where $r$ is large and do not need to worry about the singularity of $1/r$ at $r=0$. In fact, we say that a radiation filed set $(\Phi, \mathcal{A}_\alpha)$ satisfies the \emph{asymptotic Lorenz gauge condition} if modulo the terms of order $\mathcal{O}(r^{-2}\ln r)$, it holds true for the corresponding asymptotics $(\phi_{asy}, A_{\alpha, asy})$ that
\begin{equation}\label{eq:asym_Lorenz}
\partial^\alpha A_{\alpha, asy}(r-q, r\omega)=0.
\end{equation}
 for each $q$ and sufficient large $r$. 
Thus the radiation fields cannot be specified freely and they must satisfy the \emph{asymptotic Lorenz gauge condition}.

\medskip

To capture the decay of the scattering data which is given in terms of radiation fields, we define the following norm for any radiation field $F(q, \omega)$ 
\begin{equation}\label{eq:norm_of_radiation}
\norm{F}_{N, \infty, \gamma}:=\sum_{\abs{\alpha}+k\leq N}\sup_{q\in\BR}\sup_{\omega\in\BS^2}\abs{(\jb{q}\partial_q)^k\partial_\omega^\alpha F(q, \omega)}\jb{q}^\gamma
\end{equation}
for some positive integer $N$ and constant $1/2<\gamma<1$.

\subsubsection{Statement of the main theorem}

For the backward problem we simply give a scattering data in terms of a radiation field set $(\Phi, \mathcal{A}_\alpha)$ satisfying the \emph{asymptotic Lorenz gauge condition} and a smallness condition in the sense of norm defined in  ~\eqref{eq:norm_of_radiation}. Then we construct an approximate solution $(\phi_{app}, A_{\alpha,app})$ (see Subsubsection ~\ref{subsubsec:app_soln}) associated to $(\Phi, \mathcal{A}_\alpha)$ such that $(\phi_{app}, A_{\alpha,app})$ coincides with the given asymptotics ~\eqref{eq:asym_of_two} in the wave zone and with $\eqref{eq:K_timelike}$ in the far interior for sufficiently large time $t$. We then show that there is an exact solution $(\phi, A_\alpha)$ to the equations ~\eqref{eq:mkg_in_A} which is asymptotically the same as $(\phi_{app}, A_{\alpha,app})$ when $t\to\infty$. In our analysis, we apply the energy estimates to the remainders 
\[u=\phi-\phi_{app}, \quad v_\alpha=A_\alpha-A_{\alpha, app}.
\]
We are ready to state the main theorem of this paper. 
\begin{thm}[Main theorem]\label{thm:mainthm}
Let $N\geq 7$ with $N\in \BN$, and let $\gamma, \mu$ be given such that $1/2<\gamma<1$, $\mu>0$ and $\mu<\gamma-1/2$. Let $(\Phi(q, \omega), \mathcal{A}_0(q, \omega), \mathcal{A}_1(q, \omega),  \mathcal{A}_2(q, \omega),  \mathcal{A}_3(q, \omega))$ be a radiation field set (scattering data) and define the charge to be the value
\begin{equation}\label{eq:def_of_q}
\BFq=\int_{-\infty}^\infty\!\int_{\BS^2}\!\mathcal{J}_0(q, \omega)\, dS(\omega)\,dq=\int_{-\infty}^\infty\!\int_{\BS^2}\!-\Im\bigl(\Phi(q, \omega)\overline{\partial_q\Phi(q, \omega)}\bigr)\, dS(\omega)\,dq.
\end{equation}
Then there exists a small absolute constant $\ep_0$, which depends on $N, \gamma, \mu$, such that for any radiation field set $(\Phi, \mathcal{A}_\alpha)$ given above which satisfies the asymptotic Lorenz gauge condition (see the detail in Proposition~\ref{prop:asy_Lorenz}) and smallness condition
\[
\ep:=\norm{\Phi}_{N, \infty, \gamma}+\sum_{\alpha=0}^3\norm{\mathcal{A}_\alpha}_{N, \infty, \gamma}\leq\ep_0,
\]
there exists a global solution $(\phi, A_\alpha)$  to the equations ~\eqref{eq:mkg_in_A} with asymptotics ~\eqref{eq:asym_of_two} at null infinity and ~\eqref{eq:K_timelike} at far interior timelike infinity and satisfying the Lorenz gauge condition, and the following properties of the solution hold:
\begin{itemize}
	\item [(i)] The remainders satisfy the energy bounds
	\begin{equation}\label{eq:energy_in_thm}
	\norm{w^{\frac{1}{2}}Z^Iu}_{L^2(\Sigma_t)}+\norm{w^{\frac{1}{2}}Z^Iv_\alpha}_{L^2(\Sigma_t)}\leq C\ep\jb{t}^{-\frac{\gamma}{2}+\frac{\mu}{2}+\frac{1}{4}} \quad \text{for} \quad \abs{I}\leq N-4
	\end{equation}
	where
	\[
		w(q)=\begin{cases}
	1+(1+\abs{q})^\mu, \quad &q=r-t<0\\
	1+(1+\abs{q})^{-\gamma+\frac{1}{2}}, \quad &q=r-t\geq 0
	\end{cases}
	\]
	\item[(ii)] The remainders satisfy the pointwise estimates
	\begin{equation}\label{eq:decay_in_thm}
	 \abs{w^{\frac{1}{2}}Z^Iu}+\abs{w^{\frac{1}{2}}Z^Iv_\alpha}\leq
	 C\ep\jb{t+r}^{-1}\jb{t-r}^{-\frac{1}{2}}\jb{t}^{-\frac{\gamma}{2}+\frac{\mu}{2}+\frac{1}{4}} \quad \text{for}  \quad \abs{I}\leq N-6
	 \end{equation}
\end{itemize}
where $Z\in\{\partial_\mu, \, x_\mu\partial_\nu-x_\nu\partial_\mu, \, t\partial_t+r\partial_r\}$ and the constant $C$ depends only on $N, \gamma, \mu$.
\end{thm}

\begin{rem}
	We note that the asymptotics of $A_\alpha$ at timelike infinity and $\phi$ at null infinity were established in ~\cite{CKL19}. In this work, we shall give a refinement of the asymptotics of $A_\alpha$ at null infinity ~\eqref{eq:unified_K}, especially at interior null infinity.(see Proposition ~\ref{prop:refined_asymptotics} and Remark \ref{rem:refined_asym})
	% $A_\alpha$ has the asymptotics ~\eqref{eq:K_timelike} in far interior timelike infinity, and $(\phi, A_\alpha)$ has the asymptotics ~\eqref{eq:asym_of_two} at null infinity. The asymptotics of $A_\alpha$ at null infinity, especially at interior null infinity, is a direct result of Proposition ~\ref{prop:refined_asymptotics} which is a refinement of  ~\cite[Proposition 19]{Lind17}(see Remark \ref{rem:refined_asym}).
\end{rem}

\begin{rem}\label{rem:charge}
We define $\BFq$ to be $\BFq=\int_{\BR^3}\Im(\phi_0\overline{\phi_1})\, dx$ with $(\phi_0, \phi_1)=(\phi, D_t\phi)|_{t=0}$ in the forward problem. However, here $\BFq$ is defined as in ~\eqref{eq:def_of_q} in terms of the radiation field instead of the scalar field itself. See the rationale for the definition ~\eqref{eq:def_of_q} in Proposition ~\ref{prop:definition_of_charge}.
\end{rem}

\subsection{Related works}

The investigation of the long time dynamics of the solutions to the MKG equations originates in the work of Choquet-Bruhat and Christodoulou in ~\cite{C-BC81} where they applied the conformal compactification method under the assumption of zero charge. Decay estimates for the linear Maxwell field and Yang-Mills equations were obtained in the work of ~\cite{CK90} and ~\cite{Shu91} respectively. The pointwise decay property of global solutions to ~\eqref{eq:mkg} for small initial data with nonzero charge was proved in Lindblad-Sterbenz ~\cite{LS06} after the outline Shu provided in ~\cite{Shu92}. A simpler proof was later given by Bieri-Miao-Shahshahani in ~\cite{BMS17}. Recently, the decay of the solutions to ~\eqref{eq:mkg} with large initial data was studied in ~\cite{Yang16, Yang18, YangYu19}. Regarding the long time dynamics of solutions to the  massive MKG equations, we refer the reader to ~\cite{Psa99, Psa99_2, KWY20, FWY19} and references therein.

The scattering problem for MKG equations was recently studied by Taujanskas. He showed in ~\cite{Tau19} the existence of a bounded invertible scattering operator for the MKG equations in de Sitter spacetime by exploiting the conformal compactification method. However, this work cannot be extended to the Minkowski space as the charge will cause a long rang effect on the asymptotic behavior of the solutions.

There are some previous results on the scattering problem for some other nonlinear dispersive and hyperbolic equations. In \cite{FST87}, Flato-Simon-Taflin proved the existence of the modified wave operators for Maxwell-Dirac equations in $1+3$ dimensions. In ~\cite{LSo05, LSo06}, Lindblad-Soffer studied the scattering for one dimensional nonlinear Klein Gordon equation and one dimensional nonlinear Schr\"odinger equation respectively by constructing approximate solutions at infinity. In ~\cite{Wang10, Wang13}, Wang considered the radiation field for Einstein vacuum equation in harmonic gauge in higher dimensions using conformal compactification method. However, the three spatial dimension case is more delicate since the mass has a long range effect and there is an additional logarithmic leading term for some metric component. Lindblad-Schlue ~\cite{LV17} studied the backward problem for models of Einstein equations finding a solution that has the asymptotics of the form given in ~\cite{Lind17}. A scattering theory construction on dynamical vacuum black holes was given in ~\cite{DHR}. A scattering theory for the wave equation on Kerr black hole exteriors and extremal {R}eissner-{N}ordstr\"{o}m was established in \cite{DRS} and ~\cite{AAG} respectively. In ~\cite{Yu20}, Yu proved the existence of the modified wave operators for a scalar quasilinear wave equation satisfying the weak null condition.

%The goal in this paper is to obtain the global existence of the solution to the backward problem for the MKG equations in the Lorenz gauge. More precisely, we find a global solution that has the asymptotics of the form obtained in ~\cite{CKL19}. 
%The results of this paper rely on a spacetime version of Hardy inequality which allows us to employ the spacetime integral with a special derivative in the weighted conformal Morawetz energy inequality.

\subsection{Sketch of the proof}

\subsubsection{Discussion of the difficulties}\label{subsubsec:difficulties}

First we discuss the main difficulty in our analysis. In view of ~\eqref{eq:asymtotic_system}, we consider the following simplified model 
\begin{equation}\label{eq:simplified_semilinear_model}
\begin{split}
\Box \psi&=0,\\
\Box \phi&=2i\psi\partial_t\phi,\\
\Box \varphi&=-2\Im(\phi\overline{\partial_t\phi}).
\end{split}
\end{equation}
which has a similar weak null structure as MKG equations in Lorenz gauge. 

This model is different from that studied in ~\cite{LV17} whose form is $\Box \psi =Q(\partial\psi, \partial\phi), \Box \phi=\tilde{Q}(\partial\phi, \partial\varphi),\Box \varphi=(\partial_t\psi)^2$, where $Q$ and $\tilde{Q}$ are two null forms. First, since the quadratic nonlinearities in the model considered in ~\cite{LV17} only depend on the derivatives, they can be put in the standard energy norm $\int\!\abs{\partial\psi}^2+\abs{\partial\phi}^2+\abs{\partial\varphi}^2$. However, our model ~\eqref{eq:simplified_semilinear_model} contains the dangerous terms $2i\psi\partial_t\phi$ and $-2\Im(\phi\overline{\partial_t\phi})$ where $\psi$ and $\phi$ cannot be placed in the standard energy norm due to the lack a derivative $\partial$. Second, as we observe in the forward problem, we have $\abs{\psi}, \abs{\phi}, \abs{\varphi}\sim\jb{t}^{-1}$ in the far interior region $t>2r$. Then when we consider these two models in the region $t>2r$, for our model ~\eqref{eq:simplified_semilinear_model}, both the left-hand side and the right-hand side decay at the rate of  $\jb{t}^{-3}$ because the derivative decreases the order of decay rate. However, for the model in ~\cite{LV17}, the left-hand side decays like $\jb{t}^{-3}$ while the right-hand side has the decay rate $\jb{t}^{-4}$. Therefore, the nonlinearities in the model in ~\cite{LV17} give us more room as far as decay in the interior.

\subsubsection{The charge $\BFq$}

In the forward problem, it was observed that the asymptotics of $rA_L$  close to the light cone is $\frac{\BFq}{4\pi}$. On the other hand it can also be deduced from the asymptotics for the wave equation with a quadratic source term given by $\Im(\Phi\overline{\partial_q\Phi})$. Therefore we expect that there is a relation between $\BFq$ and $\Im(\Phi\overline{\partial_q\Phi})$. In fact, this relation is given by ~\eqref{eq:def_of_q} which we will prove rigorously in Proposition ~\ref{prop:definition_of_charge}.

\subsubsection{Construction of the approximate solution}\label{subsubsec:app_soln}

First, we need to take into account the \emph{asymptotic Lorenz gauge condition}. In fact, we derive an ODE that the radiation field set $(\Phi, \mathcal{A}_\alpha)$ should satisfy
\begin{equation}\label{eq:condition_from_asymLorenz}
\partial_q\Bigl(L^\alpha(\omega)\mathcal{A}_\alpha(q, \omega)+\frac{\BFq}{4\pi}\chi_{in}(q)+\frac{\BFq}{4\pi}\chi_{ex}(q)\Bigr)=0.
\end{equation}
See Proposition ~\ref{prop:asy_Lorenz} for the derivation of this ODE.

Next, in view of ~\eqref{eq:asym_of_two}, and ~\eqref{eq:K_timelike}, we construct $(\phi_{app}, A_{\alpha, app})$ as
\begin{equation}\label{eq:app_in_proofstrategy}
	\begin{split}
		\phi_{app}(t, r\omega)&=e^{-i\frac{\BFq}{4\pi}\ln r}\frac{\Phi(q, \omega)}{r}\chi\Bigl(\frac{\jb{q}}{r}\Bigr)\\
		A_{\alpha, app}(t, r\omega)&=\frac{1}{4\pi}\int_{-\infty}^{\infty}\!\int_{\BS^2}\!\frac{\mathcal{J}_{\alpha}(\eta, \sigma)}{t-r\angles{\omega}{\sigma}}\, dS(\sigma)\,d\eta\chi_{ex}(\frac{t}{12})\chi_{in}(q)\\
		&\qquad-\frac{1}{2r}\ln\frac{2r}{\jb{t-r}}\int_{-\infty}^q\!\mathcal{J}_{\alpha}(\eta, \omega)\, d\eta\chi\Bigl(\frac{\jb{q}}{r}\Bigr)\chi_{in}(q)\\
		&\qquad+\frac{1}{2r}\ln\frac{2r}{\jb{t-r}}\int_{q}^\infty\!\mathcal{J}_{\alpha}(\eta, \omega)\, d\eta\chi\Bigl(\frac{\jb{q}}{r}\Bigr)\Bigl(1-\chi_{in}(q)\Bigr)\\
		&\qquad+\frac{\BFq}{4\pi r}\delta_{\alpha 0}\chi_{ex}(q)+\frac{\mathcal{A}_{\alpha}(q, \omega)}{r}\chi\Bigl(\frac{\jb{q}}{r}\Bigr).
		\end{split}
	\end{equation}
Here $\chi(s)=1$ when $s\leq 1/2$ and $\chi(s)=0$ when $s\geq 3/4$, which is used to localize in the wave zone, away from the origin. Indeed this is a ``good approximate solution" in the sense that module terms of order $\mathcal{O}(r^{-3}\ln r)$, $(\phi_{app}, A_{\alpha, app})$ solves the reduced MKG equations ~\eqref{eq:rmkg}, i.e.
\begin{equation}\label{eq:goodapp}
\begin{split}
\Box \phi_{app} &\sim-2iA^{\alpha, app}\partial_\alpha\phi_{app}+A^{\alpha, app} A_{\alpha, app}\phi_{app}\\
\Box A_{\alpha, app} &\sim-\Im(\phi_{app}\overline{\partial_\alpha\phi_{app}})+\abs{\phi_{app}}^2A_{\alpha, app}.
\end{split}
\end{equation}
\begin{rem}
	The equation ~\eqref{eq:condition_from_asymLorenz} for $\mathcal{A}_\alpha$ implies that $A_{L, app}=\frac{\BFq}{4\pi r}+\mathcal{O}(r^{-2}\ln r)$, which coincides with the asymptotics of $A_L$ in ~\cite{CKL19}.
\end{rem}
	\begin{rem}
		Observing the first term in the construction of $A_{\alpha, app}$, we find that there is no singularity at $r=0$ when the derivative falls on the integral term, however, if the derivative falls on the cutoff $\chi_{in}(q)$, $\frac{1}{r}$ is generated which leads to a singularity at $r=0$. Therefore, we use the cutoff $\chi_{ex}(\frac{t}{12})$ to remove the singularity at $r=0$. 
	\end{rem}

\subsubsection{Reduced MKG equations}

 Next, we expect to find a sequence of  exact solutions $(\phi_T, A_{\alpha T})$ that take the form $(\phi_T, A_{\alpha T})=(\phi_{app}+u_T, A_{\alpha, app}+v_{\alpha T})$ to the reduced MKG equations ~\eqref{eq:rmkg}. To this end we consider the following equations for $(u_T, v_{\alpha T})$ with trivial data at $t=2T$ 
 \begin{equation}\label{eq:eqn_for_remainders}
 \begin{split}
 \Box u &=\widetilde{\chi}\bigl(\frac{t}{T}\bigr)\Bigl(-2iA^\alpha\partial_\alpha\phi+A^\alpha A_\alpha \phi-\Box\phi_{app}\Bigr)\\
 \Box v_{\alpha} &=\widetilde{\chi}\big(\frac{t}{T}\big)\Bigl(-J_\alpha-\Box A_{\alpha, app}\Bigr), \quad J_\alpha=\Im(\phi\overline{(\partial_\alpha+iA_\alpha)\phi})
 \end{split}
\end{equation}
 where $\widetilde{\chi}$ is a smooth cutoff such that $\widetilde{\chi}(s)=1$ when $s\leq 1$, and $\widetilde{\chi}(s)=0$ when $s\geq 2$. Clearly, $(\phi_T, A_{\alpha T})=(\phi_{app}+u_T, A_{\alpha, app}+v_{\alpha T})$ is an exact solution to the reduced MKG equations ~\eqref{eq:rmkg} when $t\leq T$. In Section ~\ref{sec:4}, ~\ref{sec:5}, we show that ~\eqref{eq:eqn_for_remainders} has a solution $(u_T, v_{\alpha T})$ for $t\leq 2T$ and that the limit $(u, v_\alpha)=\lim_{T\to\infty}(u_T, v_{\alpha T})$ exists in the energy norm used in this work. In this way we obtain a global solution $(\phi, A_{\alpha})=(\phi_{app}+u, A_{\alpha, app}+v_{\alpha})$ to the reduced MKG equations ~\eqref{eq:rmkg} which is asymptotically the same as $(\phi_{app}, A_{\alpha, app})$ when $t\to\infty$, in the sense that the energy norm of $(u, v_\alpha)$ approaches 0 as $t\to\infty$.
 
 \begin{rem}
 If we solve the backward problem for the reduced MKG equations ~\eqref{eq:rmkg} directly, we face the obstacle that while $(u_T, v_{\alpha T})$ has trivial data at $t=2T$, this does not hold for $(Z^Iu_T, Z^Iv_{\alpha T})$ for $\abs{I}\geq 1$. Here we introduce the cutoff $\widetilde{\chi}$ to overcome this problem.
 \end{rem}

 \subsubsection{Weighted conformal Morawetz energy estimate in backward direction}
 
 We establish a backward version of weighted conformal Morawetz energy estimate in Proposition ~\ref{prop:weightedenergyestimate}
 \begin{equation}\label{eq:energyineq_in_intro}
 E^w[\varphi](t_1)+\int_{t_1}^{t_2}\!\int_{\Sigma_t}\!\jb{t+r}^2\abs{\frac{1}{r}L(r\varphi)}^2 \abs{w'}\,dx\,dt
 \leq E^w[\varphi](t_2)+\int_{t_1}^{t_2}\!\int_{\Sigma_t}\!\Re{\Bigl(\frac{2}{r}\overline{K}_0(r\varphi)\overline{\Box \varphi}\Bigr)}w\,dx\,dt
 \end{equation}
 where $\overline{K}_0=\bigl(\jb{t+r}^2L+\jb{t-r}^2\underline{L}\bigr)/2$, 
 \[E^w[\varphi](t)=\frac{1}{2}\int_{\Sigma_t}\!\Bigl(\jb{t+r}^2\abs{\frac{1}{r}L(r\varphi)}^2+\left(\jb{t+r}^2+\jb{t-r}^2\right)\abs{\slashed{\partial}\varphi}^2+\jb{t-r}^2\abs{\frac{1}{r}\underline{L}(r\varphi)}^2\Bigr)w\,dx\]
 with the weight $w=w(q)$ defined such that $w'\leq 0$. This energy estimate ~\eqref{eq:energyineq_in_intro} has two advantages: first since $\sum_{\abs{I}\leq 1}\int\!\abs{Z^I\varphi}w\,dx\lesssim E^w[\varphi](t)$(see Proposition ~\ref{prop:control_of_norm} ), $\varphi$ can be put in $E^w[\varphi](t)$ ; second, due to the weight, an additional spacetime integral is generated in the left-hand side in ~\eqref{eq:energyineq_in_intro}, which plays a crucial role in our analysis.
 
 The energy estimate ~\eqref{eq:energyineq_in_intro} is a generalization of the classical conformal Morawetz estimate in ~\cite{M62}. We remark that there were other variants of the classical conformal Morawetz estimate. Lindblad-Sterbenz introduced in ~\cite{LS06} a weighted fractional Morawetz estimate in the forward direction. In ~\cite{LV17}, Lindblad-Schlue established  the fractional Morawetz estimate in the backward direction.

\subsubsection{Estimates for the main error terms}

 Now we explain how we deal with the difficulty we discussed in Subsubsection ~\ref{subsubsec:difficulties}. Here we only display the treatment of the most delicate  error terms in the energy estimate. For simplicity, we consider the equation for $v_\alpha$ as an example and $u$ can be handled in an analogous way. In view of ~\eqref{eq:goodapp} and ~\eqref{eq:eqn_for_remainders}, we think of the equation for $v_\alpha$ as
\begin{align}\label{eq:simplifiedmodel1}
\Box v_\alpha&=-\Im(u\overline{\partial_\alpha \phi_{app}})-\Im(\phi_{app}\overline{\partial_\alpha u})+\cdots%\\\label{eq:simplifiedmodel2}
%\Box u &=-\underbrace{2iv_L\frac{\partial_q\Phi}{r}\chi(\frac{\jb{t-r}}{r})}_{\text{type \Romanupper{1}}}-\underbrace{2iA_{L, app}\partial_q u}_{\text{type \Romanupper{2}}}+\{\text{better terms}\}
\end{align}
Then we commutate the vector fields $Z$ through the equation ~\eqref{eq:simplifiedmodel1} and apply ~\eqref{eq:energyineq_in_intro} to the following equations
\begin{equation}\label{eq:higherordereqn}
\Box Z^I v_\alpha=-\sum_{\abs{J}\leq\abs{I}}c_JZ^J\bigl(\Im(u\overline{\partial_\alpha\phi_{app}})\bigr)-\sum_{\abs{J}\leq\abs{I}}c_JZ^J\bigl(\Im(\phi_{app}\overline{\partial_\alpha u})\bigr)+\cdots=B+C+\cdots
\end{equation}
where $c_J$ are constants for all $\abs{I}\leq N-3$. Owing to  $\abs{\partial\varphi}\lesssim\jb{t-r}^{-1}\sum_{\abs{I}\leq 1}\abs{Z^I\varphi}$ and $\abs{Z^J\phi_{app}}\lesssim\ep\jb{t+r}^{-1}\jb{t-r}^{-\gamma}$ for all $J$ which is from ~\eqref{eq:app_in_proofstrategy}
%$r\sim\jb{t+r}$ in the support of $\phi_{app}$
, we see that 
\[
\abs{B}+\abs{C}\lesssim\frac{1}{\jb{t-r}}\sum_{\abs{J}+\abs{K}\leq\abs{I}+1}\abs{Z^J \phi_{app}}\abs{Z^Ku}\lesssim\frac{\ep}{\jb{t+r}\jb{t-r}^{1+\gamma}}\sum_{\abs{K}\leq\abs{I}+1}\abs{Z^K u}.
\]

First we deal with the component $(\jb{t-r}^2\underline{L})/2$ in the multiplier $\overline{K}_0$. We have
\begin{align*}
&\int_{t_1}^{t_2}\!\int_{\Sigma_t}\!\Re{\Bigl(\frac{\jb{t-r}^2}{r}\underline{L}(rZ^Iv_\alpha)\overline{B+C}\Bigr)}w\,dx\,dt\\
&\quad\lesssim\ep\int_{t_1}^{t_2}\!\int_{\Sigma_t}\!\frac{\jb{t-r}}{r}\abs{\underline{L}(rZ^Iv_\alpha)}\sum_{\abs{K}\leq\abs{I}+1}\frac{\abs{Z^Ku}}{\jb{t+r}}w\,dx\,dt\\
&\quad\lesssim\sum_{\abs{K}\leq\abs{I}+1}\int_{t_1}^{t_2}\!\int_{\Sigma_t}\!\Bigl(\frac{\jb{t-r}^2}{r^2}\abs{\underline{L}(rZ^Iv_\alpha)}^2+\abs{Z^Ku}^2\Bigr)\frac{\ep}{\jb{t+r}}w\,dx\,dt\\
&\quad\lesssim\int_{t_1}^{t_2}\!\frac{\ep \Bigl(E^w[Z^Iv_\alpha](t)+E^w[Z^Iu](t)\Bigr)}{\jb{t}}\,dt
\end{align*}
where we use $\sum_{\abs{I}\leq 1}\int\!\abs{Z^I\varphi}w\,dx\lesssim E^w[\varphi](t)$(see Proposition ~\ref{prop:control_of_norm} ) in the last step.

The treatment of $(\jb{t+r}^2L)/2$ component needs more delicate analysis. We rewrite the term $C$ in the right-hand side of ~\eqref{eq:higherordereqn} as
\[
C=-\Im(\phi_{app}\overline{\partial_\alpha Z^Iu})+\Bigl(\Im(\phi_{app}\overline{\partial_\alpha Z^Iu})-\sum_{\abs{J}\leq\abs{I}}c_JZ^J\bigl(\Im(\phi_{app}\overline{\partial_\alpha u})\bigr)\Bigr)=C_1+C_2.
\] 
Therefore, the highest order term $\Im(\phi_{app}\overline{\partial_\alpha Z^Iu})$ only occurs in $C_1$. For the terms $B$ and $C_2$, we have the following estimate
\begin{equation}
\begin{split}
\abs{B}+\abs{C_2}&\lesssim\frac{1}{\jb{t-r}}\sum_{ \abs{K}\leq\abs{I}, \abs{J}+\abs{K}\leq\abs{I}+1}\abs{Z^J\phi_{app}}\abs{Z^Ku}\\
&\lesssim\frac{\ep}{\jb{t+r}\jb{t-r}^{1+\gamma}}\sum_{\abs{K}\leq\abs{I}}\abs{Z^K u}.
\end{split}
\end{equation}
Since $Z^Ju=\partial_tZ^Ju=0$ for all $J$, we are able to use a spacetime version of Hardy inequality  (see Lemma ~\ref{lem:poincareineq}): if $\varphi=\partial_t\varphi=0$,
\begin{equation}\label{eq:poincare_in_intro}
\int_{t_1}^{t_2}\!\int_{\Sigma_t}\!\abs{\varphi}^2 \abs{w'}\,dx\,dt\lesssim \int_{t_1}^{t_2}\!\int_{\Sigma_t}\!\jb{t+r}^2\abs{\frac{1}{r}L(r\varphi)}^2 \abs{w'}\,dx\,dt
\end{equation}
to obtain
\begin{align*}
&\int_{t_1}^{t_2}\!\int_{\Sigma_t}\!\Re{\Bigl(\frac{\jb{t+r}^2}{r}L(rZ^Iv_\alpha)\overline{B+C_2}\Bigr)}w\,dx\,dt\\
&\quad\lesssim\ep\int_{t_1}^{t_2}\!\int_{\Sigma_t}\!\frac{\jb{t+r}}{r}\abs{L(rZ^Iv_\alpha)}\sum_{\abs{K}\leq\abs{I}}\frac{\abs{Z^K u}}{\jb{t-r}^{1+\gamma}}w\,dx\,dt\\
&\quad\lesssim\ep\int_{t_1}^{t_2}\!\int_{\Sigma_t}\!\sum_{\abs{K}\leq\abs{I}}\Bigl(\frac{\jb{t+r}^2}{r^2}\abs{L(rZ^Iv_\alpha)}^2+\abs{Z^K u}^2\Bigr)\abs{w'}\,dx\,dt\\
&\quad\lesssim \ep\int_{t_1}^{t_2}\!\int_{\Sigma_t}\!\sum_{\abs{K}\leq\abs{I}}\frac{\jb{t+r}^2}{r^2}\Bigl(\abs{L(rZ^Iv_\alpha)}^2+\abs{L(rZ^K u)}^2\Bigr)\abs{w'}\,dx\,dt.
\end{align*}
where we use $\abs{w\jb{t-r}^{-1-\gamma}}\lesssim\abs{w'}$ in the second step and ~\eqref{eq:poincare_in_intro} in the last step. 
 
 It remains to consider the term $C_1$. We rewrite \[
 \partial_\alpha Z^Iu= L_\alpha(\omega)\partial_qZ^I u -\frac{\underline{L}_\alpha(\omega)}{2}LZ^I u+\omega_\alpha^Be_B(Z^I u).
 \]
 We are only concerned about the principal term $\partial_qZ^I\varphi$ because the other terms have good derivatives $L, e_B$ and we have $\abs{L\varphi}+\abs{e_B(\varphi)}\lesssim\jb{t+r}^{-1}\sum_{\abs{I}\leq 1}\abs{Z^I\varphi}$. Then we are left with estimating 
\[
\int_{t_1}^{t_2}\!\int_{\Sigma_t}\!-\frac{\jb{t+r}^2}{r}L(rZ^Iv_{\alpha})L_\alpha(\omega)\Im(\phi_{app}\overline{\partial_q Z^Iu})\,dx\,dt.
\]
In order to handle this, we integrate by parts with respect to the $q$ direction to transfer $\partial_q$ from $\partial_q Z^I u$ to $L(rZ^I v_\alpha)$. Since $\Box \varphi=2r^{-1}\partial_qL(r\varphi)+r^{-2}\Delta_{\omega}\varphi=2r^{-1}\partial_qL(r\varphi)+(r^{-2}\sum_{i, j=1}^3\Omega_{ij}^2\varphi)/2$, we find that
 \begin{align*}
&\int_{t_1}^{t_2}\!\int_{\Sigma_t}\!-\frac{\jb{t+r}^2}{r}L(rZ^Iv_{\alpha})L_\alpha(\omega)\Im(\phi_{app}\overline{\partial_q 	Z^Iu})w\,dx\,dt\\
&\qquad=\frac{1}{2}\int_{t_1}^{t_2}\!\int_{\Sigma_t}\!\jb{t+r}^2\Box (Z^Iv_{\alpha})L_\alpha(\omega)\Im(\phi_{app}\overline{Z^Iu})w\,dx\,dt\\
&\qquad\qquad-\int_{t_1}^{t_2}\!\int_{\Sigma_t}\!\jb{t+r}^2\frac{1}{4r^2}\Bigl(\sum_{i, j=1}^3\Omega_{ij}^2Z^Iv_\alpha \Bigr)L_\alpha(\omega)\Im(\phi_{app}\overline{Z^Iu})w\,dx\,dt+\cdots
\end{align*}
For the second integral in the right-hand side above, we do integration by parts with respect to the spatial variables to move one angular rotation vector field $\Omega_{ij}$ from $\Omega_{ij}^2Z^I v_\alpha$ to the other terms and then we are able to control it by $
\ep\int_{t_1}^{t_2}\!\jb{t}^{-1}(E^w[Z^Iv_\alpha](t)+E^w[Z^Iu](t))\,dt
$. It remains to estimate the first integral in the right-hand side above. We insert the equation for $\Box Z^I v_\alpha$ into the integral and again we are only concerned about the principal term  
$-L_\alpha(\omega)\Im(\phi_{app}\overline{\partial_q Z^Iu})$ in  $\Box Z^I v_\alpha$ because the other terms can be controlled by using the spacetime version of Hardy inequality ~\eqref{eq:poincare_in_intro} as before. In this way we obtain 
\begin{align*}
&-\frac{1}{2}\int_{t_1}^{t_2}\!\int_{\Sigma_t}\!\jb{t+r}^2L^2_\alpha(\omega)\Im(\phi_{app}\overline{\partial_q Z^Iu})\Im(\phi_{app}\overline{Z^Iu})w\,dx\,dt\\
&\qquad=\frac{1}{8}\int_{t_1}^{t_2}\!\int_{\Sigma_t}\!\jb{t+r}^2L^2_\alpha(\omega)\Bigl(\frac{\phi_{app}^2}{2}\partial_q(\overline{Z^I u})^2-\abs{\phi_{app}}^2\partial_q\abs{Z^I u}^2+\frac{\overline{\phi_{app}}^2}{2}\partial_q(Z^I u)^2\Bigr)w\,dx\,dt.
\end{align*}
Integrating by parts in $q$ direction again yields
\begin{equation}\label{eq:type2}
-\frac{1}{8}\int_{t_1}^{t_2}\!\int_{\Sigma_t}\!\jb{t+r}^2L^2_\alpha(\omega)\Bigl((\overline{Z^I u})^2\partial_q(\frac{\phi_{app}^2}{2})-\abs{Z^I u}^2\partial_q\abs{\phi_{app}}^2+(Z^I u)^2\partial_q(\frac{\overline{\phi_{app}}^2}{2})\Bigr)w\,dx\,dt+\cdots.
\end{equation}
Then we are in the position to use the spacetime version of Hardy inequality ~\eqref{eq:poincare_in_intro} again to estimate the leading term in ~\eqref{eq:type2}.
Our conclusion then follows by using a standard bootstrap argument.

\subsubsection{MKG equations}

Finally in Section ~\ref{sec:5} we prove that the global solution to the backward problem for the reduced MKG equations ~\eqref{eq:rmkg} is indeed a solution to MKG equations ~\eqref{eq:mkg_in_A} provided that the radiation set $(\Phi, \mathcal{A}_\alpha)$ satisfies the \emph{asymptotic Lorenz gauge condition}. Thanks to the key observation that $\lambda=\partial^\alpha A_\alpha$ satisfies the wave equation ~\eqref{eq:eq_for_lorenz}, given the solution $(\phi_T, A_{\alpha T})$ to the reduced MKG equations ~\eqref{eq:rmkg} for $t\leq T$ we conclude for $Z^I\lambda_T$ with $\abs{I}\leq 1$ that uniformly in $(t, x)$
\[
\sum_{\abs{I}\leq 1}\abs{Z^I\lambda_T(t, x)}\to 0\quad \text{as}\quad T\to \infty,
\]
from which our conclusion follows.

\subsection*{Acknowledgements}

The author would like to thank her advisor, Hans Lindblad, for suggesting this problem, many helpful discussions and his constructive criticism of the manuscript.

\section{Preliminaries}\label{sec:2}

In this section we first introduce the basic notations and identities in Subsection ~\ref{subsec:notations}. Then we derive a backward version of weighted conformal Morawetz energy estimate in Subsection ~\ref{subsec:energyestimate}. Finally we establish a spacetime version of Hardy inequality and a weighted Klainerman-Sobolev inequality in Subsection ~\ref{subsec:poincareineq}.

\subsection{Notation and conventions}\label{subsec:notations}

For nonnegative $A$, $B$, we write $A\lesssim B$ or $A=\mathcal{O}(B)$ if $A\leq CB$ for some universal positive constant $C$. We write $A\sim B$ if $A\lesssim B$ and $B\lesssim A$. We use the notation $\lesssim_\nu$ or $C_\nu$ to indicate that the constant depends on a parameter $\nu$. The exact values of all constants in this paper may vary from line to line. Moreover, we use the Japanese brackets notation $\jb{a}=(1+a^2)^{1/2}$ and we denote by $\slashed{\partial}_i$ the spatial angular components of $\partial_i$.

We can write the usual translation invariant frame $\{\partial_\alpha\}$ in terms of the null frame defined in ~\eqref{eq:nullframe} as follows
\begin{equation}\label{eq:Euframe}
\partial_\alpha=-\frac{\underline{L}_\alpha}{2}L-\frac{L_\alpha}{2}\underline{L}+\omega_\alpha^Be_B.
\end{equation}
where $\omega_\alpha^B=e_B(x^\alpha)=\omega^\alpha_B$, which follows from the formula
\[
\omega_\alpha^B=\angles{\partial_\alpha}{e_B}=dx^\alpha(e_B)=e_B(x^\alpha).
\]
Also we note that 
\begin{equation}
\omega^i\omega_i^B=\omega_i\omega^i_B=0, \quad \omega_B^i\omega_i^D=\delta_B^D, \quad \omega_i^B\omega_B^j=\delta_i^j-\omega_i\omega^j.
\end{equation}
Similarly, we can write the gauge potential $A_\alpha$ in terms of its null decomposition
\begin{equation}\label{eq:A_in_nullframe}
A_\alpha=-\frac{\underline{L}_\alpha}{2}A_L-\frac{L_\alpha}{2}A_{\underline{L}}+\omega_\alpha^BA_{e_B}.
\end{equation}

\medskip

We use the following vector fields in this work
\begin{equation}\label{eq:vectorfields}
\begin{split}
\partial_\alpha, \quad \Omega_{ij}=x_i\partial_j&-x_j\partial_i, \quad \Omega_{0i}=t\partial_i+x_i\partial_t, \quad S=t\partial_t+r\partial_r, \\
 \overline{K}_0&=\frac{1+(t+r)^2}{2}L+\frac{1+(t-r)^2}{2}\underline{L}.
\end{split}
\end{equation}
$\partial_\alpha$, $\Omega_{ij}$, $\Omega_{0i}$ and $S$ will be used as commutators, while $\overline{K}_0$ will be served as our multiplier. We let $Z$ be any of the commutator vector fields here and in the sequel. For any multiindex $I$ with length $\abs{I}$, let $Z^I$ denote the product of $\abs{I}$ such commutator vector fields. Moreover, by a sum over $I=J+K$ we mean a sum over all possible order preserving partitions of the ordered multiindex into two ordered multiindex $J$ and $K$, that is, if $I=(\iota_1,\ldots, \iota_k)$, then $J=(\iota_{i_1},\ldots, \iota_{i_n})$ and $K=(\iota_{i_{n+1}}, \ldots, \iota_{i_k})$, where $i_1,\ldots, i_k$ is any reordering of the integers $1,\ldots, k$ such that $i_1<\ldots<i_n$ and $i_{n+1}<\ldots<i_k$. Then we have Leibniz's rule
\begin{equation}\label{eq:leibniz}
Z^I(fg)=\sum_{J+K=I}Z^JfZ^Kg.
\end{equation}
We define the tangential derivatives $\overline{\partial}$ to be either the $\partial_p$ or the angular derivatives $\slashed{\partial}_i$, and we have the following expressions for the tangential derivatives in terms of the commutator vector fields
\begin{equation}\label{eq:tanderi}
\partial_p=\frac{1}{2}L=\frac{S+\omega^i\Omega_{0i}}{2(t+r)}, \quad \slashed{\partial}_i=\partial_i-\omega_i\partial_r=\frac{\omega^j\Omega_{ji}}{r}=\frac{-\omega_i\omega^j\Omega_{0j}+\Omega_{0i}}{t}.
\end{equation}
We have the following identities
\begin{equation}\label{eq:vectorfield_identity}
\begin{split}
\partial_\alpha&=-\frac{\underline{L}_\alpha}{2}L-\frac{L_\alpha}{2}\underline{L}+\omega_\alpha^Be_B,
\quad \Omega_{ij}=(x_i\omega_j^B-x_j\omega_i^B)e_B,\\
\Omega_{0i}&=\frac{\omega_i}{2}((t+r)L-(t-r)\underline{L})+t\omega_i^	Be_B, \quad S=\frac{1}{2}((t+r)L+(t-r)\underline{L}).
\end{split}
\end{equation}
In this paper, we repeatedly use the following commutation identities
\begin{equation}
\begin{split}\label{eq:com_iden}
[\Box, Z]&=0, \quad \text{if} \quad Z=\partial_\alpha, \Omega_{ij},\Omega_{0i},\quad [\Box, S]=2\Box,\\
[\partial_\alpha, \partial_\beta]&=0, \quad[S, \partial_\alpha]=-\partial_\alpha,\quad[L, \partial_0]=[\underline{L}, \partial_0]=[e_B, \partial_0]=0,\\
[\Omega_{0i}, \partial_\alpha]&=-\delta_{0\alpha}\partial_i-\delta_{i\alpha}\partial_0,\quad[\Omega_{ij}, \partial_\alpha]=\delta_{j\alpha}\partial_i-\delta_{i\alpha}\partial_j,\\
[L, \partial_i]&=\omega_i^B[L, e_B]=-\frac{\omega_i^B}{r}e_B,\\
[\underline{L}, \partial_i]&=\omega_i^B[\underline{L}, e_B]=\frac{\omega_i^B}{r}e_B,\\
[e_B, \partial_i]&=\frac{\omega_{Bi}}{2r}(L-\underline{L})-\omega_i^E\overline{\Gamma}_{EB}^De_D,\\
[L, \Omega_{ij}]&=-[\underline{L}, \Omega_{ij}]=0,\\
[e_B, \Omega_{ij}]&=e_B(\Omega_{ij}^D)e_D+\Omega_{ij}^E[e_B, e_E]^De_D,\\
[L, \Omega_{0i}]&=\omega_iL+\frac{r-t}{r}\omega_i^De_D,\\
[\underline{L}, \Omega_{0i}]&=-\omega_i\underline{L}+\frac{t+r}{r}\omega_i^De_D,\\
[e_B, \Omega_{0i}]&=\frac{\omega_i^B}{2r}((t+r)L+(r-t)\underline{L})-t\omega_i^E\overline{\Gamma}_{EB}^De_D,\\
[L, S]&=L,\quad [\underline{L}, S]=\underline{L},\quad [e_A, S]=0.
\end{split}
\end{equation}
Here, $\overline{\Gamma}_{BE}^D$ denotes the Christoffel symbol of the sphere.

Finally, we record a pointwise estimate on partial derivatives
\begin{lem}\label{lem:pointwise_estimate_of_der}
	For any function $f$, we have the estimate
	\begin{equation}
	(1+t+r)\abs{\overline{\partial}f}+(1+\abs{t-r})\abs{\partial f}\leq C\sum_{\abs{I}=1}\abs{Z^I f}.
	\end{equation}
	\end{lem}

\subsection{Energy estimates}\label{subsec:energyestimate}

In this subsection we derive a weighted conformal Morawetz energy estimate for the backward problem for the equation
\begin{equation}\label{eq:waveeqn}
\Box \varphi=F.
\end{equation}
where $\varphi$ and $F$ are complex-valued functions.

We define the weighted conformal energy at time slice $\Sigma_t=\{(x^0, x^1, x^2,x^3)\mid x^0=t\}$
\begin{equation}\label{eq:weightedenergy_at_t}
E^{w}[\varphi](t):=\frac{1}{2}\int_{\Sigma_t}\!\Bigl(\jb{t+r}^2\abs{\frac{1}{r}L(r\varphi)}^2+\left(\jb{t+r}^2+\jb{t-r}^2\right)\abs{\slashed{\partial}\varphi}^2+\jb{t-r}^2\abs{\frac{1}{r}\underline{L}(r\varphi)}^2\Bigr)w\,dx.
\end{equation}
Now we state the weighted conformal Morawetz energy estimate we will use in this work
\begin{prop}\label{prop:weightedenergyestimate}
Let $\varphi$ be a solution to the equation ~\eqref{eq:waveeqn} with $\lim_{\abs{x}\to\infty}\abs{x}^{3/2}\varphi(t, x)w(q)^{1/2}=0$, and let $E^{w}[\varphi](t)$ be defined by ~\eqref{eq:weightedenergy_at_t}. Then for $t_1\leq t_2$,
\begin{equation}\label{eq:weightedenergyestimate}
\begin{split}
E^{w}[\varphi](t_1)&=E^{w}[\varphi](t_2)+\int_{t_1}^{t_2}\!\int_{\Sigma_t}\!\Re\Bigl(\frac{2}{r}\overline{K}_0(r\varphi)\overline{F}\Bigr)w\,dxdt\\
&\qquad+\int_{t_1}^{t_2}\!\int_{\Sigma_t}\!\Bigl(\jb{t+r}^2\abs{\frac{1}{r}L(r\varphi)}^2+\jb{t-r}^2\abs{\slashed{\partial}\varphi}^2\Bigr)w'\,dxdt
\end{split}
\end{equation}
where $w=w(q)$ is an arbitrary function of $q=r-t$.
\end{prop} 
 
 \begin{proof}
 Let $Q_{\alpha\beta}$ be the energy momentum tensor for the linear wave  equation ~\eqref{eq:waveeqn}
 \[
 Q_{\alpha\beta}=\Re(\partial_\alpha\varphi\overline{\partial_\beta\varphi})-\frac{1}{2}m_{\alpha\beta}\partial^\gamma\varphi\overline{\partial_\gamma\varphi}.
 \]
 We use the conformal killing vector field $\overline{K}_0$ as defined in ~\eqref{eq:vectorfields} whose deformation tensor is given by
 \begin{equation}\label{eq:defortensor}
 {}^{\overline{K}_0}\pi^{\alpha\beta}=\partial^\alpha\overline{K}_0^\beta+\partial^\beta\overline{K}_0^\alpha=4tm^{\alpha\beta}.
 \end{equation}
 We now form a momentum density associated with the vector field $\overline{K}_0$ and the extra weight $w$
 \begin{equation}\label{eq:momentumdensity}
 P_\alpha=\bigl(Q_{\alpha\beta}\overline{K}_0^\beta+2t\Re(\varphi\overline{\partial_\alpha \varphi})-\frac{1}{4}\partial_\alpha(4t)\abs{\varphi}^2\bigr)w(q).
 \end{equation}
 Now, calculating the spacetime divergence of the quantity ~\eqref{eq:momentumdensity} we have that 
 \begin{equation}\label{eq:divergence}
 \begin{split}
 \partial^\alpha P_\alpha &=\bigl(\partial^\alpha Q_{\alpha\beta}\overline{K}_0^\beta+\frac{1}{2}Q_{\alpha\beta} {}^{\overline{K}_0}\pi^{\alpha\beta}+2t\partial^\alpha \varphi \overline{\partial_\alpha \varphi}+2t\Re(\varphi \overline{F})\bigr)w(q)\\
 &\qquad+\bigl(Q(\overline{K}_0, L)+2t\Re(\varphi\overline{L \varphi})-\abs{\varphi}^2\bigr)w'(q)\\
 &=\bigl(\Re(\overline{K}_0(\varphi)\overline{F}) +2t\Re(\varphi \overline{F})\bigr)w(q)+\bigl(Q(\overline{K}_0, L)+2t\Re(\varphi\overline{L \varphi})-\abs{\varphi}^2\bigr)w'(q)\\
 &=\Re(\frac{1}{r}\overline{K}_0(r\varphi)\overline{F})w(q)+\bigl(Q(\overline{K}_0, L)+2t\Re(\varphi\overline{L \varphi})-\abs{\varphi}^2\bigr)w'(q).
 \end{split}
 \end{equation}
 Integrating both sides of ~\eqref{eq:divergence} over the time slabs $\{t_1\leq t\leq t_2\}$ and applying the Stokes theorem we arrive at
 \begin{equation}\label{eq:energyidentity}
 \int_{\Sigma_{t_1}}\!P_0\,dx= \int_{\Sigma_{t_2}}\!P_0\,dx+\int_{t_1}^{t_2}\!\int_{\Sigma_t}\!\Bigl(\Re(\frac{1}{r}\overline{K}_0(r\varphi)\overline{F})w(q)+\bigl(Q(\overline{K}_0, L)+2t\Re(\varphi\overline{L \varphi})-\abs{\varphi}^2\bigr)w'(q)\Bigr)\,dx\,dt.
 \end{equation}
 In order to proceed, we calculate each term in ~\eqref{eq:energyidentity} individually.
 
 \medskip
 
For the the flux term $P_0=\bigl(Q(\overline{K}_0, \partial_t)+2t\Re(\varphi\overline{\partial_t\varphi})-\abs{\varphi}^2\bigr)w(q)$, we have 
\[
Q(\overline{K}_0, \partial_t)=\frac{1}{4}\jb{t+r}^2\abs{L\varphi}^2+\frac{1}{4}(\jb{t+r}^2+\jb{t-r}^2)\abs{\slashed{\partial}\varphi}^2+\frac{1}{4}\jb{t-r}^2\abs{\underline{L}\varphi}^2.
\]
We also note that 
\begin{align*}
2t\Re(\varphi\partial_t\varphi)&=\frac{\jb{t+r}^2-\jb{t-r}^2}{2r}\Re(\varphi\overline{\partial_t\varphi})\\
&=\frac{\jb{t+r}^2}{2r}\Re(\varphi\overline{L\varphi})-\frac{\jb{t-r}^2}{2r}\Re(\varphi\overline{\underline{L}\varphi})-\frac{\jb{t+r}^2+\jb{t-r}^2}{2r}\Re(\varphi\overline{\partial_r\varphi}).
\end{align*}
Integrating by parts in $r$ we find that
\begin{align*}
&\int_{\Sigma_t}\!-\frac{\jb{t+r}^2+\jb{t-r}^2}{2r}\Re(\varphi\overline{\partial_r\varphi})w(q)\,dx\\
&\qquad=-\frac{1}{4}\int_{\Sigma_t}\!(\jb{t+r}^2+\jb{t-r}^2)r\partial_r(\varphi\overline{\varphi})w(q)\,dS(\omega)\,dr\\
&\qquad=\frac{1}{4}\int_{\Sigma_t}\!\frac{\jb{t+r}^2+\jb{t-r}^2}{r^2}\abs{\varphi}^2w(q)\,dx+\int_{\Sigma_t}\abs{\varphi}^2w(q)\,dx+\frac{1}{4}\int_{\Sigma_t}\!\frac{\jb{t+r}^2+\jb{t-r}^2}{r}\abs{\varphi}^2w'(q)\,dx.
\end{align*}
Therefore we obtain
\begin{equation}\label{eq:flux}
\begin{split}
\int_{\Sigma_{t_1}}\!P_0\,dx&=\int_{\Sigma_{t_1}}\!\Bigl(\frac{1}{4}\jb{t+r}^2\abs{L\varphi}^2+\frac{1}{4}(\jb{t+r}^2+\jb{t-r}^2)\abs{\slashed{\partial}\varphi}^2+\frac{1}{4}\jb{t-r}^2\abs{\underline{L}\varphi}^2\Bigr)w(q)\,dx\\
&\qquad+\int_{\Sigma_{t_1}}\!\Bigl(\frac{\jb{t+r}^2}{2r}\Re(\varphi\overline{L\varphi})-\frac{\jb{t-r}^2}{2r}\Re(\varphi\overline{\underline{L}\varphi})+\frac{\jb{t+r}^2+\jb{t-r}^2}{4r^2}\abs{\varphi}^2\Bigr)w(q)\,dx\\
&\qquad+\frac{1}{4}\int_{\Sigma_t}\!\frac{\jb{t+r}^2+\jb{t-r}^2}{r}\abs{\varphi}^2w'(q)\,dx\\
&=\frac{1}{4}\int_{\Sigma_{t_1}}\!\Bigl(\jb{t+r}^2\abs{\frac{1}{r}L(r\varphi)}^2+\left(\jb{t+r}^2+\jb{t-r}^2\right)\abs{\slashed{\partial}\varphi}^2+\jb{t-r}^2\abs{\frac{1}{r}\underline{L}(r\varphi)}^2\Bigr)w(q)\,dx\\
&\qquad+\frac{1}{4}\int_{\Sigma_{t_1}}\!\frac{\jb{t+r}^2+\jb{t-r}^2}{r}\abs{\varphi}^2w'(q)\,dx.
\end{split}
\end{equation}
It remains to calculate the last term on the right-hand side of ~\eqref{eq:energyidentity}. First we have
\[Q(\overline{K}_0, L)=\frac{1}{2}\jb{t+r}^2\abs{L\varphi}^2+\frac{1}{2}\jb{t-r}^2\abs{\slashed{\partial}\varphi}^2.\]
Then we write 
\[
2t\Re(\varphi\overline{L\varphi})=\frac{\jb{t+r}^2-\jb{t-r}^2}{2r}\Re(\varphi\overline{L\varphi})=\frac{\jb{t+r}^2}{r}\Re(\varphi\overline{L\varphi})-\frac{\jb{t+r}^2+\jb{t-r}^2}{2r}\Re(\varphi\overline{L\varphi}).
\]
Integrating by parts in $p=t+r$ direction yields
\begin{align*}
&\int_{t_1}^{t_2}\!\int_{\Sigma_t}\!-\frac{\jb{t+r}^2+\jb{t-r}^2}{2r}\Re(\varphi\overline{L\varphi})w'(q)\,dx\,dt\\
&\qquad=-\frac{1}{4}\int_{\BS^2}\!\int_{-t_2}^\infty\!\int_{p_1}^{p_2}\,(\jb{p}^2+\jb{q}^2)r\partial_p(\varphi\overline{\varphi})w'(q)dp\,dq\,dS(\omega)\\
&\qquad=-\frac{1}{4}\int_{\Sigma_{t_2}}\!\frac{\jb{t+r}^2+\jb{t-r}^2}{r}\abs{\varphi}^2w'(q)\,dx+\frac{1}{4}\int_{\Sigma_{t_1}}\!\frac{\jb{t+r}^2+\jb{t-r}^2}{r}\abs{\varphi}^2w'(q)\,dx\\
&\qquad\qquad+\frac{1}{4}\int_{t_1}^{t_2}\!\int_{\Sigma_t}\!\frac{\jb{t+r}^2+\jb{t-r}^2}{r^2}\abs{\varphi}^2w'(q)\,dx\,dt+\int_{t_1}^{t_2}\!\int_{\Sigma_t}\!\frac{(t+r)}{r}\abs{\varphi}^2w'(q)\,dx\,dt.
\end{align*}
Thus we find that 
\begin{equation}\label{eq:spacetimeterm}
\begin{split}
&\int_{t_1}^{t_2}\!\int_{\Sigma_t}\!\bigl(Q(\overline{K}_0, L)+2t\Re(\varphi\overline{L \varphi})-\abs{\varphi}^2\bigr)w'(q)\,dx\,dt\\
&\qquad=\int_{t_1}^{t_2}\!\int_{\Sigma_t}\!\Bigl(\frac{1}{2}\jb{t+r}^2\abs{L\varphi}^2+\frac{1}{2}\jb{t-r}^2\abs{\slashed{\partial}\varphi}^2+\frac{\jb{t+r}^2}{r}\Re(\varphi\overline{L\varphi})\Bigr)w'(q)\,dx\,dt\\
&\qquad\qquad+\frac{1}{4}\int_{t_1}^{t_2}\!\int_{\Sigma_t}\!\Bigl(\frac{\jb{t+r}^2+\jb{t-r}^2}{r^2}+\frac{4(t+r)}{r}-4\Bigr)\abs{\varphi}^2w'(q)\,dx\,dt\\
&\qquad\qquad-\frac{1}{4}\int_{\Sigma_{t_2}}\!\frac{\jb{t+r}^2+\jb{t-r}^2}{r}\abs{\varphi}^2w'(q)\,dx+\frac{1}{4}\int_{\Sigma_{t_1}}\!\frac{\jb{t+r}^2+\jb{t-r}^2}{r}\abs{\varphi}^2w'(q)\,dx\\
&\qquad=\frac{1}{2}\int_{t_1}^{t_2}\!\int_{\Sigma_t}\!\Bigl(\jb{t+r}^2\abs{\frac{1}{r}L(r\varphi)}^2+\jb{t-r}^2\abs{\slashed{\partial}\varphi}^2\Bigr)w'(q)\,dx\,dt\\
&\qquad\qquad-\frac{1}{4}\int_{\Sigma_{t_2}}\!\frac{\jb{t+r}^2+\jb{t-r}^2}{r}\abs{\varphi}^2w'(q)\,dx+\frac{1}{4}\int_{\Sigma_{t_1}}\!\frac{\jb{t+r}^2+\jb{t-r}^2}{r}\abs{\varphi}^2w'(q)\,dx.
\end{split}
\end{equation}
Putting ~\eqref{eq:energyidentity}--\eqref{eq:spacetimeterm} together we see that the flux terms containing $w'(q)$ exactly cancel, and this finishes the proof of Proposition ~\ref{prop:weightedenergyestimate}.
 \end{proof}

\medskip

The following corollary corresponds to a specific choice of the weight function $w(q)$. We distinguish between the interior and exterior of the light cone and introduce the notation
\[
\Sigma_t^i=\{(t, x)\mid \abs{x}\leq t\}, \quad \Sigma_t^e=\{(t, x)\mid \abs{x}>t)\}.
\]
We define a weighted conformal Morawetz energy as follows
\begin{equation}\label{eq:specificenergy_at_t}
\begin{split}
E_\nu[\varphi](t)&:=\frac{1}{2}\int_{\Sigma_t^i}\!\Bigl(\jb{t+r}^2\abs{\frac{1}{r}L(r\varphi)}^2+\left(\jb{t+r}^2+\jb{t-r}^2\right)\abs{\slashed{\partial}\varphi}^2+\jb{t-r}^2\abs{\frac{1}{r}\underline{L}(r\varphi)}^2\Bigr)(1+\abs{q})^\nu\,dx\\
&\qquad+\frac{1}{2}\int_{\Sigma_t^e}\!\Bigl(\jb{t+r}^2\abs{\frac{1}{r}L(r\varphi)}^2+\left(\jb{t+r}^2+\jb{t-r}^2\right)\abs{\slashed{\partial}\varphi}^2+\jb{t-r}^2\abs{\frac{1}{r}\underline{L}(r\varphi)}^2\Bigr)\,dx.
\end{split}
\end{equation}
\begin{cor}\label{cor:energyestimate_for_u}
Let $\varphi$ be a solution to the equation ~\eqref{eq:waveeqn} with $\lim_{\abs{x}\to\infty}\abs{x}^{3/2}\varphi(t, x)=0$, and let $E_\nu[\varphi](t)$ be defined by ~\eqref{eq:specificenergy_at_t} with $\nu>0$. Then for all $\nu>0$, $\delta>0$ and $t_1\leq t_2$,
\begin{equation}
\begin{split}
&E_\nu[\varphi](t_1)+\nu\int_{t_1}^{t_2}\!\int_{\Sigma_t^i}\!\frac{\jb{t+r}^2\abs{\frac{1}{r}L(r\varphi)}^2}{\jb{q}^{1-\nu}}\,dx\,dt+\delta\int_{t_1}^{t_2}\!\int_{\Sigma_t^e}\!\frac{\jb{t+r}^2\abs{\frac{1}{r}L(r\varphi)}^2}{\jb{q}^{1+\delta}}\,dx\,dt\\
&\qquad\lesssim E_\nu[\varphi](t_2)+\biggl\lvert\int_{t_1}^{t_2}\!\int_{\Sigma_t^i}\!\Re\Bigl(\frac{1}{r}\overline{K}_0(r\varphi)\overline{F}\Bigr)(1+\abs{q})^\nu\,dx\,dt\biggr\rvert+\biggl\lvert\int_{t_1}^{t_2}\!\int_{\Sigma_t^e}\!\Re\Bigl(\frac{1}{r}\overline{K}_0(r\varphi)\overline{F}\Bigr)\,dx\,dt\biggr\rvert.
\end{split}
\end{equation}
\end{cor}
\begin{proof}
For $\nu>0$ and $\delta>0$ we let
\[w=
\begin{cases}
1+(1+q)^{-\delta}, \quad&q\geq0,\\
1+(1+\abs{q})^\nu, \quad&q<0.
\end{cases}
\]
Then we have
\[w'(q)=
\begin{cases}
-\delta(1+\abs{q})^{-1-\delta}, \quad &q>0,\\
-\nu(1+\abs{q})^{\nu-1}, \quad &q<0.
\end{cases}
\]
Therefore the statement of the Corollary follows from Proposition ~\ref{prop:weightedenergyestimate}.
\end{proof}
Now let
\begin{align}
\norm{\varphi(t, \cdot)}_{k}&:=\sum_{\abs{I}\leq k}\norm{(Z^I\varphi)(t, \cdot)}_{L^2_x(w)}\label{eq:norm1}\\
\norm{\varphi(t, \cdot)}^2_{1,+}&:=\int_{\Sigma_t}\!\Bigl(\jb{t+r}^2\bigl(\abs{\frac{1}{r}L(r\varphi)}^2+\abs{\slashed{\partial}\varphi}^2\bigr)+\jb{t-r}^2\abs{\frac{1}{r}\underline{L}(r\varphi)}^2+\frac{\jb{t-r}^2+r^2}{r^2}\abs{\varphi}^2\Bigr)w(q)\,dx.\label{eq:norm2}
\end{align}
where $\norm{f}_{L^2_x(w)}:=\int_{\BR^3} \abs{f}^2w\,dx$ and $k\in\BN$. 
\begin{prop}\label{prop:control_of_norm}
Let  $\lim_{\abs{x}\to\infty}\abs{x}^{3/2}\varphi(t, x)=0$, and let $E^w[\varphi](t)^{1/2}$, $\norm{\varphi(t, \cdot)}_{1}$ and $\norm{\varphi(t, \cdot)}^2_{1,+}$ be defined by ~\eqref{eq:weightedenergy_at_t}, ~\eqref{eq:norm1} and ~\eqref{eq:norm2} respectively with $w'(q)\leq0$ if $q\leq 0$, $w'(q)=0$ if $q>0$. Then we have
\begin{align}
\norm{\varphi(t, \cdot)}_{1,+}&\lesssim E^w[\varphi](t)^{1/2},\label{eq:control_of_norm1}\\
\norm{\varphi(t, \cdot)}_{1}&\lesssim\norm{\varphi(t, \cdot)}_{1,+}\,.\label{eq:control_of_norm2}
\end{align}
\end{prop}
To prove Proposition ~\ref{prop:control_of_norm} we need the following two lemmas to estimate the zeroth order terms in ~\eqref{eq:norm2}.
\begin{lem}\label{lem:control_of_varphi1}
Let  $\lim_{\abs{x}\to\infty}\abs{x}^{3/2}\varphi(t, x)=0$, and let $w(q)$ satisfy that $w'(q)\leq 0$ if $q\leq0$ and $w'(q)=0$ if $q>0$, then one has
\begin{equation}
\int_{\Sigma_t}\!\abs{\varphi}^2wdx\lesssim\int_{\Sigma_t}\!\Bigl(\jb{t+r}^2\abs{\frac{1}{r}L(r\varphi)}^2+\jb{t-r}\abs{\frac{1}{r}\underline{L}(r\varphi)}^2\Bigr)w\,dx.
\end{equation}
\end{lem}
\begin{proof}
We rewrite
\[
\int_{\Sigma_t}\!\abs{\varphi}^2w(q)dx=\int_0^\infty\!\int_{\BS^2}\!\abs{r\varphi}^2w(q)\,dr\,dS(\omega).
\]
Integrating by parts in $r$ yields
\begin{align*}
\int_0^\infty\!\abs{r\varphi}^2w(q)\,dr&=\int_0^\infty\!\partial_r(r-t)\abs{r\varphi}^2w(q)\,dr\\
&=-\int_0^\infty\!(r-t)2\Re(r\varphi\overline{\partial_r(r\varphi)})w(q)\,dr-\int_0^\infty\!(r-t)\abs{r\varphi}^2w'(q)\,dr\\
&\leq-\int_0^\infty\!(r-t)2\Re(r\varphi\overline{\partial_r(r\varphi)})w(q)\,dr\\
&\lesssim\sqrt{\int_0^\infty\!\abs{r\varphi}^2w\,dr}\sqrt{\int_0^\infty\!(r-t)^2\abs{\partial_r(r\varphi)}^2w\,dr}.
\end{align*}
Since $2\partial_r=(L-\underline{L})$ and $(r-t)^2\leq (r+t)^2$, it follows that 
\[
\int_0^\infty\!\abs{\varphi}^2w\,dr\lesssim\int_0^\infty\!\Bigl(\jb{t+r}^2\abs{\frac{1}{r}L(r\varphi)}^2+\jb{t-r}\abs{\frac{1}{r}\underline{L}(r\varphi)}^2\Bigr)w\,r^2dr.
\]
This finishes the proof of the Lemma.
\end{proof}
\begin{lem}\label{lem:control_of_varphi2}
Let  $\lim_{\abs{x}\to\infty}\abs{x}^{3/2}\varphi(t, x)=0$, and let $w(q)$ satisfy that $w'(q)\leq 0$ if $q\leq0$ and $w'(q)=0$ if $q>0$, then one has
\begin{equation}
\int_{\Sigma_t}\!\frac{\jb{t-r}^2}{r^2}\abs{\varphi}^2wdx\lesssim\int_{\Sigma_t}\!\Bigl(\jb{t+r}^2\abs{\frac{1}{r}L(r\varphi)}^2+\jb{t-r}\abs{\frac{1}{r}\underline{L}(r\varphi)}^2\Bigr)w\,dx.
\end{equation}	
\end{lem}
\begin{proof}
We have $\abs{\varphi}^2=2\Re(\varphi\overline{\partial_r(r\varphi)})-\partial_r(r\abs{\varphi}^2)$ and thus by Cauchy-Schwarz
\begin{align*}
\int_0^\infty\!\jb{t-r}^2\abs{\varphi}^2w(q)\,dr&=\int_0^\infty\!\jb{t-r}^22\Re(\varphi\overline{\partial_r(r\varphi)})w(q)\,dr-\int_0^\infty\!\jb{t-r}^2\partial_r(r\abs{\varphi}^2)w(q)\,dr\\
&=\int_0^\infty\!\jb{t-r}^22\Re(\varphi\overline{\partial_r(r\varphi)})w(q)\,dr+\int_0^\infty\!\Bigl(2(r-t)w(q)+\jb{t-r}^2w'(q)\Bigr)r\abs{\varphi}^2\,dr\\
&\leq \int_0^\infty\!\jb{t-r}^22\Re(\varphi\overline{\partial_r(r\varphi)})w(q)\,dr+\int_0^\infty\!2(r-t)w(q)r\abs{\varphi}^2\,dr\\
&\lesssim\Bigl(\int_0^\infty\!\jb{t-r}^2\abs{\varphi}^2w(q)\,dr\Bigr)^{1/2}\Bigl(\int_0^\infty\!\jb{t-r}^2\abs{\partial_r(r\varphi)}^2\,dr+\int_0^\infty\!\abs{r\varphi}^2w(q)\,dr\Bigr)^{1/2}
\end{align*}
Then the assertion of Lemma ~\ref{lem:control_of_varphi2} follows from Lemma ~\ref{lem:control_of_varphi1}.
\end{proof}

\begin{proof}[Proof of Proposition ~\ref{prop:control_of_norm}]
Lemma ~\ref{lem:control_of_varphi1} and Lemma ~\ref{lem:control_of_varphi2} give the first estimate ~\eqref{eq:control_of_norm1}. We notice that
\[
\int_{\Sigma_t}\!\Bigl(\jb{t+r}^2\abs{L(\varphi)}^2+\left(\jb{t+r}^2+\jb{t-r}^2\right)\abs{\slashed{\partial}\varphi}^2+\jb{t-r}^2\abs{\underline{L}(\varphi)}^2+\abs{\varphi}^2\Bigr)w\,dx\lesssim \norm{\varphi(t,\cdot)}_{1,+}.
\]
Then the second estimate ~\eqref{eq:control_of_norm2} follows from the identities ~\eqref{eq:vectorfield_identity}.
\end{proof}

\subsection{Spacetime version of Hardy inequality and Klainerman-Sobolev inequality}\label{subsec:poincareineq}

We now establish a spacetime version of the Hardy inequality which will be repeatedly used in the course of the proof of our main theorem.
\begin{lem}\label{lem:poincareineq}
For any differentiable function $\varphi$ satisfying $\varphi=\partial_t\varphi=0$ at $t=t_2$ and any nonnegative function $w=w(q)$, we have \begin{equation}
\int_{t_1}^{t_2}\!\int_{\Sigma_t}\!\abs{\varphi}^2w\,dx\,dt\lesssim \int_{t_1}^{t_2}\!\int_{\Sigma_t}\jb{t+r}^2\abs{\frac{1}{r}L(r\varphi)}^2w\,dx\,dt.
\end{equation}
\end{lem}
\begin{proof}
	We rewrite the spacetime integral in null coordinates $p=t+r$ and $q=r-t$ as
\[
	\int_{t_1}^{t_2}\!\int_{\Sigma_t}\!\abs{\varphi}^2w(q)\,dx\,dt
	=\frac{1}{2}\int_{\mathbb{S}^2}\!\int_{-t_2}^\infty\!\int_{p_1}^{p_2}\!\abs{r\varphi}^2w(q)\,dp\,dq\,dS(\omega),
	\]
	and applying integration by parts we find that 
	\begin{align*}
\int_{p_1}^{p_2}\!\abs{r\varphi}^2w(q)\,dp&=\int_{p_1}^{p_2}\!\partial_p((t+r))\abs{r\varphi}^2w(q)\,dp\\
&=(t+r)\abs{r\varphi}^2w(q)\big|_{p_1}^{p_2}-\int_{p_1}^{p_2}\!(t+r)\Re(r\varphi\overline{L(r\varphi)})w(q)\,dp.
	\end{align*}
By the assumption that $\varphi=\partial_t\varphi=0$ at $t=t_2$	and Cauchy Schwarz we obtain that 
\[
\int_{p_1}^{p_2}\!\abs{r\varphi}^2w(q)\,dp\leq\Bigl(\int_{p_1}^{p_2}\!\abs{r\varphi}^2w(q)\,dp\Bigr)^{1/2}\Bigl(\int_{p_1}^{p_2}\!(t+r)^2\abs{\frac{1}{r}L(r\varphi)}^2w(q)\,r^2dp\Bigr)^{1/2}
\]
Integrating over the sphere and along $q$-direction yields the statement of the Lemma.
\end{proof}
We now record the Klainerman-Sobolev inequality, see for instance \cite[Chapter \Romanupper{6}]{Horm97} and  \cite[Chapter \Romanupper{2}]{Sogge08} for a proof.
\begin{prop}\label{prop:KSineq}
Let $\varphi\in C^\infty(\BR^{1+3})$ vanish when $\abs{x}$ is large. Then 
\begin{equation}
(1+t+\abs{x})(1+\bigl\lvert t-\abs{x}\bigr\rvert)^{1/2}\abs{\varphi(t, x)}\leq C\sum_{\abs{I}\leq 2}\norm{Z^I\varphi(t,\cdot)}_{L^2(\BR^3)}.
\end{equation}
\end{prop}
Finally we derive a weighted version of Klainerman-Sobolev inequality that is adapted to the norm ~\eqref{eq:norm1}.
\begin{lem}\label{lem:KS_with_weight}
Let the weight function $w(q)$ be defined as 
\[w=
\begin{cases}
1+(1+q)^{-\delta}, \quad&q\geq0,\\
1+(1+\abs{q})^\nu, \quad&q<0.
\end{cases}
\]
with $\delta>0,  0<\nu<1$. For any function $\varphi\in C^\infty(\BR^{1+3})$ which vanishes when $\abs{x}$ is large, we have
\begin{equation}
(1+t+\abs{x})(1+\bigl\lvert t-\abs{x}\bigr\rvert)^{1/2}w(q)^{1/2}\abs{\varphi(t, x)}\lesssim \sum_{\abs{I}\leq 2}\norm{Z^I\varphi(t,\cdot)}_{L^2(w)}.
\end{equation}
\end{lem}
\begin{proof}
It suffices to prove the Lemma in the interior region $q\leq0$. Let $\chi$ be a smooth cutoff such that $\chi(q)=1$ when $q\leq0$, and $\chi(q)=0$ when $q\geq1/2$. Define $\widetilde{\varphi}(t, x):=\chi(q)\varphi(t, x)w(q)^{1/2}$. By Proposition ~\ref{prop:KSineq}, for any $(t, x)$ such that $\abs{x}\leq t$, we find that 
\begin{equation}\label{eq:KS}
(1+t+\abs{x})(1+\bigl\lvert t-\abs{x}\bigr\rvert)^{1/2}\abs{\widetilde{\varphi}(t, x)}\leq C\sum_{\abs{I}\leq 2}\norm{Z^I\widetilde{\varphi}(t,\cdot)}_{L^2(\BR^3)}.
\end{equation}
Since $w(q)$ and $\chi(q)$ depend only on $q$, only the $\underline{L}$ component of their derivative makes contribution. In view of the identities ~\eqref{eq:vectorfield_identity}, we see that for any $Z$
\[
\abs{Z\chi}\lesssim\abs{q\chi'(q)}\lesssim 1, \quad \abs{Zw^{1/2}}\lesssim w^{1/2}.
\]
As a result, the right-hand side of ~\eqref{eq:KS} is $\lesssim \sum_{\abs{I}\leq 2}\norm{Z^I\varphi(t,\cdot)}_{L^2(w)}$, which proves the Lemma. 
\end{proof}

\section{Asymptotics at infinity}\label{sec:3}

In this section we consider the asymptotics in the forward problem. First we express the charge in terms of the radiation field $\Phi$ in Subsection ~\ref{subsec:charge}. Then, we establish the asymptotics result which is a refinement of the result from ~\cite[Theorem 1.1]{CKL19} in Subsection ~\ref{subsec:asymptotics}. Finally, we deduce an ODE the radiation set $(\Phi, \mathcal{A}_\alpha)$ satisfies due to the \emph{asymptotic Lorenz condition} in Subsection ~\ref{subsec:asym_lorenz}.

\subsection{The charge}\label{subsec:charge}
In the forward problem, the charge $\BFq$ is defined in terms of the initial data at $t=0$. In contrast, for the backward problem, we expect to define the charge $\BFq$ in terms of the scattering data. 

In the forward problem, it was observed that the asymptotics of $rA_L$  close to the light cone is $\frac{\BFq}{4\pi}$. On the other hand it can also be deduced from the asymptotics for the wave equation with a quadratic source term given by $\Im(\Phi\overline{\partial_q\Phi})$. Therefore we expect that there is a relation between $\BFq$ and $\Im(\Phi\overline{\partial_q\Phi})$. Following the proof of \cite[Proposition 28]{Lind17} , we establish the relation between $\BFq$ and $\Im(\Phi\overline{\partial_q\Phi})$.
\begin{prop}\label{prop:definition_of_charge}
If $(\phi, A_\alpha)$ is a global solution to the reduced MKG equations ~\eqref{eq:rmkg}, let $\BFq$ and $\Phi(q, \omega)$ be defined by ~\eqref{eq:def_of_q_for} and ~\eqref{eq:asym_of_phi} respectively. Then we have
\begin{equation}
\BFq=\int_{-\infty}^\infty\!\int_{\BS^2}\!\mathcal{J}_0(q, \omega)\, dS(\omega)\,dq=\int_{-\infty}^\infty\!\int_{\BS^2}\!-\Im\bigl(\Phi(q, \omega)\overline{\partial_q\Phi(q, \omega)}\bigr)\, dS(\omega)\,dq.
\end{equation} 
\end{prop}
\begin{proof}
Following ~\cite[Proposition 9.3]{CKL19} we have
\[
\widetilde{A}_\alpha(t, x)=\frac{1}{4\pi}\int_q^\infty\!\int_{\BS^2}\!\frac{\mathcal{J}_\alpha(\rho, \omega)}{t+\rho-\angles{x}{\omega}}\,dS(\omega)\chi\bigl(\frac{\jb{\rho}}{t+r}\bigr)\,d\rho
\]
with $\abs{A_\alpha-\widetilde{A}_\alpha}\lesssim \ep^2\jb{t+r}^{-1}\jb{t-r}^{-\gamma}$ for some $1/2<\gamma<1$ and small positive constant $\ep$. Here $\mathcal{J}_\alpha(\rho, \omega)=L_\alpha(\omega)\Im(\Phi(\rho, \omega)\overline{\partial_q\Phi(\rho, \omega)})
$ with $\norm{\Phi}_{N, \infty, \gamma}\lesssim\ep$ and $\chi(s)$ is a smooth cutoff such that $\chi(s)=1$ when $s\leq1/2$, and $\chi(s)=0$ when $s\geq 3/4$. Then we find that 
\[
\widetilde{A}_L(t, r\omega)=\frac{1}{r}\int_q^\infty\!\int_{\BS^2}\!\frac{(1-\angles{\omega}{\sigma})\mathcal{J}_0(\rho, \sigma)}{(\rho-q)/r+1-\angles{\omega}{\sigma}}\frac{dS(\sigma)}{4\pi}\chi\bigl(\frac{\jb{\rho}}{t+r}\bigr)\,d\rho.
\]
Thus
\begin{align*}
\int_{\BS^2}\!r\widetilde{A}_L(t, r\omega)\frac{dS(\omega)}{4\pi}&=\int_q^\infty\!\int_{\BS^2}\!\int_0^2\!\frac{\omega_1\mathcal{J}_0(\rho, \sigma)}{(\rho-q)/r+\omega_1}\,d\omega_1\frac{dS(\sigma)}{8\pi}\chi\bigl(\frac{\jb{\rho}}{t+r}\bigr)\,d\rho\\
&=\int_q^\infty\!\int_{\BS^2}\!\Bigl(2-a\ln\frac{2+a}{a}\Bigr)\mathcal{J}_0(\rho, \sigma)\frac{dS(\sigma)}{8\pi}\chi\bigl(\frac{\jb{\rho}}{t+r}\bigr)\,d\rho
\end{align*}
where $a=(\rho-q)/r$. Now we estimate the following term when $q\ll0$
 \begin{align*}
& \Bigl\lvert\int_{\BS^2}\!r\widetilde{A}_L(t, r\omega)\frac{dS(\omega)}{4\pi}-\int_q^\infty\!\int_{\BS^2}\!\mathcal{J}_0(\rho, \sigma)\bigr)\frac{dS(\sigma)}{4\pi}\,d\rho\Bigr\rvert\\
 &\qquad=\Bigl\lvert\int_q^\infty\!\int_{\BS^2}\!\Bigl(\bigl(1-\frac{a}{2}\ln\frac{2+a}{a}\bigr)\chi\bigl(\frac{\jb{\rho}}{t+r}\bigr)-1\Bigr)\mathcal{J}_0(\rho, \sigma)\frac{dS(\sigma)}{4\pi}\,d\rho\Bigr\rvert
 \end{align*}
The integral over $\jb{\rho}\geq(t+r)/2$ is easily bounded by 
\[
\int_q^{-\frac{t+r}{2}+1}+\int_{\frac{t+r}{2}-1}^\infty\!2\abs{\mathcal{J}_0(\rho, \sigma)}\,d\rho\lesssim\frac{\ep^2}{\jb{r}}.
\]
So it remains to integrate over $\jb{\rho}\leq (t+r)/2$. In this case, $\chi=1$ and we are left with 
\begin{align*}
\Bigl\lvert\int_q^{\frac{t+r}{2}}\!\int_{\BS^2}\frac{a}{2}\ln\frac{a+2}{a}\mathcal{J}_0(\rho, \sigma)\frac{dS(\sigma)}{4\pi}\,d\rho\Bigr\rvert&\lesssim\int_q^{\frac{t+r}{2}}\!\frac{a}{2}\ln\frac{a+2}{a}\frac{1}{\jb{\rho}^{1+2\gamma}}\,d\rho\\
&\lesssim\int_0^{\frac{3t-r}{2r}}\frac{a}{2}\ln\frac{a+2}{a}\frac{r}{(1+\abs{ar+q})^{1+2\gamma}}\,da\\
&\lesssim_\gamma\ep^2\Bigl(\frac{\abs{q}}{r}\Bigr)^{1/2}
\end{align*}
where we use the fact that $(a/2)\ln[(a+2)/2]$ is an increasing function of $a$ and $\ln[(a+2)/a]\leq 2a^{-1/2}$ for $a>0$. Therefore we conclude that 
\begin{equation}\label{eq:bound1}
\Bigl\lvert\int_{\BS^2}\!r\widetilde{A}_L(t, r\omega)\frac{dS(\omega)}{4\pi}-\int_q^\infty\!\int_{\BS^2}\!\mathcal{J}_0(\rho, \sigma)\bigr)\frac{dS(\sigma)}{4\pi}\,d\rho\Bigr\rvert\lesssim_\gamma\ep^2\Bigl(\frac{\abs{q}}{r}\Bigr)^{1/2}.
\end{equation}
By ~\cite[Theorem 1.1]{CKL19}, when $q\leq0$ we have
\begin{equation}\label{eq:bound2}
\abs{r\widetilde{A}_L(q, \omega, r)-\frac{\BFq}{4\pi}}\lesssim\frac{\ep}{\jb{q_-}^\gamma}+\ep\frac{\jb{q}^{1-\ep}}{\jb{r}^{1-\ep}}.
\end{equation}
Collecting ~\eqref{eq:bound1} and ~\eqref{eq:bound2} and passing to the limit $r\to\infty$ for fixed $q$ we arrive at
\[
\Bigl\lvert\frac{\BFq}{4\pi}-\int_q^\infty\!\int_{\BS^2}\!\mathcal{J}_0(\rho, \sigma)\bigr)\frac{dS(\sigma)}{4\pi}\,d\rho\Bigr\rvert\lesssim\frac{\ep}{\jb{q_-}^\gamma}.
\]
Letting $q\to-\infty$ yields the Proposition.
\end{proof}

\subsection{Asymptotics of the solution to the wave equation with quadratic sources}\label{subsec:asymptotics}

In ~\cite{Lind17}, Lindblad studied the solution to the wave equation with quadratic sources and made a decomposition into a part coming from the source term and a remainder. It was shown that the radiation field for the remainder exists, but only decays in the exterior. In this subsection we establish a refinement of the result from ~\cite[Proposition 19]{Lind17} in the sense that the radiation field for the remainder decays both in the interior and the exterior.
\begin{prop}\label{prop:refined_asymptotics}
	Let $\chi(s)$ be a smooth cutoff satisfying $\chi(s)=1$ when $s\leq 1/2$, and $\chi(s)=0$ when $s\geq 3/4$. Set
	\[
	F[n](t, x)=\frac{n(q, \omega)}{r^2}\chi\bigl(\frac{\jb{q}}{r}\bigr)
	\]
	where $n(q, \omega)$ is a smooth function satisfying \[\norm{n}_{N, \infty, 1+a}\lesssim \ep^2 \quad \text{for} \quad 0<a<2, \quad 0<\ep\ll 1, \quad N\geq5.
	\]
	Let $\varPhi[n]$ be the solution of $\Box \varPhi[n]=-F[n]$ with vanishing data and let
	\begin{align*}
	\varPhi_1[n]&=\biggl(\frac{1}{4\pi}\int_{-\infty}^{\infty}\!\int_{\BS^2}\!\frac{n(\eta, \sigma)}{t-r\angles{\omega}{\sigma}}\, dS(\sigma)\,d\eta\chi_{ex}(\frac{t}{12})-\frac{1}{2r}\ln\frac{2r}{\jb{t-r}}\int_{-\infty}^q\!n(\eta, \omega)\, d\eta\chi\Bigl(\frac{\jb{t-r}}{r}\Bigr)\Biggr)\chi_{in}(r-t)\\
	&\qquad+\frac{1}{2r}\ln\frac{2r}{\jb{t-r}}\int_{q}^\infty\!n(\eta, \omega)\, d\eta\chi\Bigl(\frac{\jb{t-r}}{r}\Bigr)\Bigl(1-\chi_{in}(r-t)\Bigr)
	\end{align*}
	where $\chi_{in}(q), \chi_{ex}(s)\in C_0^\infty$ such that $\chi_{in}(q)=1$ if $q\leq -1$, and $\chi_{in}(q)=0$ if $q\geq-1/2$, while $\chi_{ex}(s)=1$ for $s\geq 1$, and $\chi_{ex}(s)=0$ for $s\leq 1/2$. Let $\psi_0=\varPhi_0[n]=\varPhi[n]-\varPhi_1[n]$. Then the limit
	\[
	\Psi_0^\infty(q, \omega)=\lim_{r\to\infty} \Psi_0(q, \omega, r) \quad \text{where}\quad \Psi_0(q, \omega, r)=r\psi_0(t, r\omega)
	\]	
	exists and satisfies
	\begin{equation}\label{eq:decay}
	\bigl\lvert\Omega_{ij}^I(q\partial_q)^J\partial_q^K\Psi_0^\infty(q, \omega)\bigr\rvert\lesssim
	\begin{cases}
	\ep^2\jb{q}^{-a}, \quad &0<a<1,\\
	C_\mu\ep^2\jb{q_+}^{-a}\jb{q_-}^{-\mu}, \quad &1\leq a<2.
	\end{cases}
	\end{equation}
 with any $0<\mu<1$ and $\abs{I}+\abs{J}+\abs{K}\leq N-5$.
\end{prop}

\begin{rem}\label{rem:refined_asym}
	 We can obtain the refinement of the asymptotics of $A_\alpha$ at null infinity ~\eqref{eq:K_ex}, ~\eqref{eq:K_in} from Proposition ~\ref{prop:refined_asymptotics}. In ~\cite{CKL19}, the decompositions $A_\alpha=A^1_\alpha+\BFq/(4\pi r)\delta_{\alpha 0}\chi_{ex}(q)$ and $A^1_\alpha=A^2_\alpha+A^0_\alpha$ where $A_\alpha^0$ solves the homogeneous wave equation with the same data as $A_\alpha^1$, and $A_\alpha^2$ solves the inhomogeneous wave equation $\Box A_\alpha^2=-J_\alpha$ with vanishing data, were introduced. First, since $A_\alpha^0$ is a solution to a linear homogeneous wave equation, $A_\alpha^0$ has a radiation filed $A_\alpha^0\sim r^{-1}\mathcal{A}_\alpha^0(q, \omega)$. Next, we consider an intermediate approximation $A_\alpha^{as}$ of $A_\alpha^2$ satisfying $\Box A_\alpha^{as}=-J_\alpha^\infty:=-r^{-2}\mathcal{J}_\alpha(q, \omega)\chi(\jb{q}/r)$ where $\mathcal{J}_\alpha(q, \omega)=L_\alpha(\omega)\Im(\Phi(q, \omega)\overline{\partial_q\Phi(q, \omega)})$. In view of the decay of $J_\alpha-J_\alpha^\infty$ and Lemma ~\ref{lem:existence_of_radiationfield}, we see that $A_\alpha^2-A_\alpha^{as}$ has a radiation field $(A_\alpha^2-A_\alpha^{as})\sim r^{-1}\mathcal{A}_\alpha^{as}(q, \omega)$ (see details in \cite[Theorem 9.1]{CKL19}). Finally, applying Proposition ~\ref{prop:refined_asymptotics} yields the asymptotics of $A_\alpha^{as}$. Putting everything together, we find the refined asymptotics of $A_\alpha$ at null infinity.
\end{rem}
First we establish a technical lemma on which the proof of Proposition ~\ref{prop:refined_asymptotics} relies.
\begin{lem}\label{lem:taylor_expansion_in_n}
Suppose $n(q, \omega)$ is a smooth function satisfying 
\begin{equation}\label{eq:condition_on_norm}
\norm{n}_{N, \infty, 1+b}\lesssim \ep^2 \quad \text{for} \quad b>0, \quad 0<\ep\ll 1, \quad N\geq 1.
\end{equation}
Set
\[
\varphi_k(t, r\omega)=\int_{-\infty}^{\infty}\!\int_{\BS^2}\!\frac{n(\eta, \sigma)-n(\eta, \omega)}{\left(t-r\angles{\omega}{\sigma}\right)^k}\, dS(\sigma)\,d\eta.
\]
Then for $\abs{I}\leq N-1$, in the region $\{(t, x)\mid t-r\geq 1/2, \,t<cr\}$ with $c>1$ we have
\begin{equation}\label{eq:estimate_of_phik}
\begin{split}
\abs{Z^I\varphi_k}&\lesssim_c
\ep^2\jb{t+r}^{-k}\Bigl(\ln\frac{\jb{t+r}}{\jb{t-r}}\Bigr)^{k-1} \quad\text{if} \quad k=1, 2,\\
\abs{Z^I\varphi_k}&\lesssim_c\ep^2\jb{t+r}^{-2}\jb{t-r}^{2-k}\quad\text{if}\quad k\geq 3.
\end{split}
\end{equation}
	\end{lem}
\begin{proof}
Rewriting $\omega=O(1,0,0)$ where $O^TO=I_3$ and making an angular change of variables $\sigma\to O\sigma$ give
\begin{equation}\label{eq:rewrite}
t^k\varphi_1(t, r\omega)=\int_{-\infty}^\infty\! \int_{\BS^2}\!\frac{n(\eta, O\sigma)-n(\eta,O(1,0,0))}{\left(1-(r/t)\sigma_1\right)^k}\,dS(\sigma)\,d\eta.
\end{equation}	
For any smooth function $m(q, \sigma)$ satisfying ~\eqref{eq:condition_on_norm}, we can consider the Taylor expansion of $m$ in a neighborhood of $(1, 0,0)$ with respect to $(\sigma_2, \sigma_3)$, namely, 
\[
m(q, \sigma)=m(q, (1,0,0))+m_2(q)\sigma_2+m_3(q)\sigma_3+\mathcal{O}(\sigma_2^{2}+\sigma_3^{2})
\]
where $m_2(q)$ and $m_3(q)$ can be expressed in terms of the angular derivatives of $m$ evaluated at $(1, 0, 0)$. Here we can relate the coordinates $(\sigma_2, \sigma_3)$
 to the polar coordinates $(\theta, \phi)$ as well as $(\sigma_1, \phi)$ as follows
\[\sigma_2=\sin\theta\cos\phi=\sqrt{1-\sigma_1^2}\cos\phi, \quad \sigma_3=\sin\theta\sin\phi=\sqrt{1-\sigma_1^2}\sin\phi.\]
Therefore, on a neighborhood $\mathcal{U}$ of $(1,0,0)$, say $\mathcal{U}=\{(\sigma_1, \phi)\mid -1\leq d\leq\sigma_1\leq 1, 0\leq\phi<2\pi\}$, we obtain that 
\begin{align*}
&\int_{-\infty}^\infty\! \int_{\mathcal{U}}\!\frac{n(\eta, O\sigma)-n(\eta,O(1,0,0))}{\left(1-(r/t)\sigma_1\right)^k}\,dS(\sigma)\,d\eta\\
&\qquad=\int_c^1\!\int_0^{2\pi}\!\frac{\sqrt{1-\sigma_1^2}\left(n_2(\omega)\cos\phi+n_3(\omega)\sin\phi\right)+\mathcal{O}_{\ep^2}(\sigma_2^2+\sigma_3^2)}{\left(1-(r/t)\sigma_1\right)^k}\,d\phi\,d\sigma_1\\
&\qquad\lesssim_{\ep^2}\int_{-1}^1\!\frac{(1-\sigma_1)}{\left(1-(r/t)\sigma_1\right)^k}\,d\sigma_1
%&\qquad\lesssim_{c, \ep^2}
%\begin{cases}
%1, \quad &k=1,\\
%ln\frac{t+r}{t-r},\quad &k=2,\\
%(t+r)^{2-k}(t-r)^{2-k}, \quad &k\geq 3.
%\end{cases}
\end{align*}
We see that the above integral is $\mathcal{O}_c(1)$ when $k=1$, $\mathcal{O}_c(\ln\frac{t+r}{t-r})$ when $k=2$ and $\mathcal{O}_c((t+r)^{k-2}(t-r)^{2-k})$ when $k\geq 3$.
Here $n_2(\omega)$ and $n_3(\omega)$ can be written in terms of the angular derivatives of $n$ and of size $\mathcal{O}_{\ep^2}(1)$. It remains to control the integral over the region $\BS^2 \setminus\mathcal{U}$ and this is easy since $1-(r/t)\sigma_1\lesssim1$ in that region. This proves ~\eqref{eq:estimate_of_phik} for $\abs{I}=0$. Now let us turn to the case $\abs{I}>1$. If $Z=\Omega_{ij}$, by ~\eqref{eq:rewrite}, we see that any angular derivative on $\omega$ falls on $n$ which we can control as previously by assumption ~\eqref{eq:condition_on_norm}. If $Z=S$, we note that $Sf(t, r)=\partial_af(at, ar)\big|_{a=1}$. By a change of variables $n(\eta, \omega)\to n(a\eta, \omega)$ we see that $S\varphi$ corresponds to replacing $n$ by $q\partial_q n$. If $Z=\partial_t$ or $\partial_i=\omega_i\partial_r+\slashed{\partial}_i$, obviously, it holds true. Finally, it remains to consider the case $Z=\Omega_{0i}=\omega^i(t\partial_r+r\partial_t)+t\slashed{\partial}_i$ and it suffices to calculate the $\omega_i(t\partial_r+r\partial_t)$ component. We observe that for any $l\in\BZ_+$
\begin{align*}
\omega_i(t\partial_r+r\partial_t)\varphi_l&=\omega_il\int_{-\infty}^{\infty}\!\int_{\BS^2}\!\frac{\bigl(t\angles{\omega}{\sigma}-r\bigr)\bigl(n(\eta, \sigma)-n(\eta, \omega)\bigl)}{\left(t-r\angles{\omega}{\sigma}\right)^{l+1}}\, dS(\sigma)\,d\eta\\
&=-\omega_il\frac{t}{r}\int_{-\infty}^{\infty}\!\int_{\BS^2}\!\frac{n(\eta, \sigma)-n(\eta, \omega)}{\left(t-r\angles{\omega}{\sigma}\right)^l}\, dS(\sigma)\,d\eta+\omega_il\frac{t^2-r^2}{r}\int_{-\infty}^{\infty}\!\int_{\BS^2}\!\frac{n(\eta, \sigma)-n(\eta, \omega)}{\left(t-r\angles{\omega}{\sigma}\right)^{l+1}}\, dS(\sigma)\,d\eta.
\end{align*}
This ends the proof of Lemma ~\eqref{lem:taylor_expansion_in_n}.
\end{proof}
Next, we record two lemmas which are needed in the course of the proof of ~\eqref{prop:refined_asymptotics}.
\begin{lem}[{\cite[Lemma 18]{Lind17}}]\label{lem:existence_of_radiationfield}
	Suppose that for some $0\leq\delta<1-\gamma$ and $0\leq\gamma'\leq\gamma$
	\begin{align*}
	\abs{\Box \Omega_{ij}^IS^J\partial_t^K\varphi}&\lesssim\frac{\ep}{\jb{t+r}^{3-\delta}\jb{r-t}^\delta\jb{(r-t)_+}^\gamma\jb{(r-t)_-}^{\gamma'}}\\
	\abs{\Delta_\omega\Omega_{ij}^IS^J\partial_t^K\varphi}&\lesssim\frac{\ep}{\jb{t+r}^{1-\delta}\jb{r-t}^\delta\jb{(r-t)_+}^\gamma\jb{(r-t)_-}^{\gamma'}}
	\end{align*}
for $\abs{I}+\abs{J}+\abs{K}\leq N$ and $r>t/4$, and
\[
\abs{(\partial_t+\partial_r)(\Omega_{ij}^IS^J\partial_t^Kr\varphi)}
\lesssim\ep(1+r)^{-1-\gamma}, \quad \text{when}\quad t=0.
\]
Then the limit 
\[
\varPhi^\infty(q, \omega)=\lim_{r\to\infty}\varPhi(q, \omega, r), \quad \varPhi(q, \omega, r)=r\varphi(t, r\omega),
\]
exists and satisfies, for $r>t/4$ and $\abs{I}+\abs{J}+\abs{K}\leq N$, 
\begin{equation}
\abs{\Omega_{ij}^I(q\partial_q)^J\partial_q^K\varPhi^\infty(q, \omega)}\lesssim\ep^2\jb{q_-}^{\gamma-\gamma'}\jb{q}^{-\gamma}.
\end{equation}
	\end{lem}
\begin{lem}[{\cite[Lemma 10.1]{CKL19}}]\label{lem:radial_estimate}
	If $\Box \varphi=-F$ with vanishing initial data where
	\[
	\abs{F}\leq\frac{C}{(1+r)(1+t+r)(1+\abs{r-t})^{1+\delta}}, \quad \delta>0,
	\]
	then we have
	\begin{equation}
	\abs{\varphi}\leq\frac{CS^0(t,r)}{(1+t+r)(1+q_+)^\delta} \quad \text{where} \quad S^0(t, r)=\frac{t}{r}\ln\frac{\jb{t+r}}{\jb{r-t}}.
	\end{equation}
	On the other hand suppose that for some $\mu>0$
	\[
	\abs{F}\leq\frac{C}{(1+r)(1+t+r)^{1+\mu}(1+\abs{r-t})^{1-\mu}(1+q_+)^{\delta_+}(1+q_-)^{\delta_-}} 
	\]
	Then if $0<\delta_+<\mu$, $0\leq\delta_-\leq\delta_+$ we have
	\begin{equation}
	\abs{\varphi}\leq\frac{C}{(1+t+r)(1+q_+)^{\delta_+}(1+q_-)^{\delta_-}}
	\end{equation}
	Finally, if $0<\mu<\delta_-<\delta_+$, then 
	\begin{equation}
	\abs{\varphi}\leq\frac{C}{(1+t+r)(1+\abs{q})^\mu(1+q_+)^{\delta_+-\mu}}.
	\end{equation}
\end{lem}

\medskip

\begin{proof}[Proof of Proposition ~\ref{prop:refined_asymptotics}]
We compute that
	\begin{align*}
	\Box\varPhi_1[n]&=\Box \frac{1}{4\pi}\int_{-\infty}^{\infty}\!\int_{\BS^2}\!\frac{n(\eta, \sigma)}{t-r\angles{\omega}{\sigma}}\, dS(\sigma)\,d\eta\chi_{ex}(\frac{t}{12})\chi_{in}(q)\\
	&\qquad-\Box\frac{1}{2r}\ln\frac{2r}{\jb{t-r}}\int_{-\infty}^q\!n(\eta, \omega)\, d\eta\chi\Bigl(\frac{\jb{q}}{r}\Bigr)\chi_{in}(q)+\frac{1}{2r}\ln\frac{2r}{\jb{t-r}}\int_{q}^\infty\!n(\eta, \omega)\, d\eta\chi\Bigl(\frac{\jb{q}}{r}\Bigr)\Bigl(\chi_{in}(q)-1\Bigr)\\
	&\equiv\Romanupper{1}-\Romanupper{2}.
	\end{align*}
For the term $\Romanupper{1}$ we find that 
\begin{align*}
\Romanupper{1}&=\chi_{ex}(\frac{t}{12})\chi_{in}(q)\Box \frac{1}{4\pi}\int_{-\infty}^\infty \!\int_{\BS^2}\!\frac{n(\eta,\sigma)}{t-\angles{x}{\sigma}}\,dS(\sigma)\,d\eta+\frac{1}{4\pi}\int_{-\infty}^\infty\! \int_{\BS^2}\!\frac{n(\eta,\sigma)}{t-\angles{x}{\sigma}}\,dS(\sigma)\,d\eta\Box\bigl(\chi_{ex}(\frac{t}{12})\chi_{in}(q)\bigr)\\
&\qquad+2\partial^\mu\frac{1}{4\pi}\int_{-\infty}^\infty\ \int_{\BS^2}\!\frac{n(\eta,\sigma)}{t-\angles{x}{\sigma}}\!dS(\sigma)\,d\eta\partial_\mu\bigl(\chi_{ex}(\frac{t}{12})\chi_{in}(q)\bigr)\\
&=\Romanupper{1}^{(1)}+\Romanupper{1}^{(2)}+\Romanupper{1}^{(3)}.
\end{align*}
And $\Romanupper{1}^{(1)}=0$,
\begin{align*}
 \Romanupper{1}^{(2)}&=\frac{1}{4\pi}\int_{-\infty}^\infty\int_{\BS^2}\!\frac{n(\eta,\sigma)}{t-\angles{x}{\sigma}}\,dS(\sigma)\,d\eta\bigl(\frac{2}{r}\chi_{in}'(q)\chi_{ex}(\frac{t}{12})+\frac{1}{6}\chi'_{in}(q)\chi'_{ex}(\frac{t}{12})-\frac{1}{144}\chi_{in}(q)\chi''_{ex}(\frac{t}{12})\bigr),\\
\Romanupper{1}^{(3)}&=\frac{1}{2\pi}\int_{-\infty}^\infty\ \int_{\BS^2}\!\frac{n(\eta,\sigma)}{\left(t-\angles{x}{\sigma}\right)^2}\Bigl(\bigl(\angles{\omega}{\sigma}-1\bigr)\chi_{in}'(q)\chi_{ex}(\frac{t}{12})+\frac{1}{12}\chi_{in}(q)\chi'_{ex}(\frac{t}{12})\Bigr)\,dS(\sigma)\,d\eta.
\end{align*} Then we conclude that 
\begin{align*}
\Romanupper{1}
%&=\frac{1}{2\pi}\frac{t-r}{t^2r}\int_{-\infty}^\infty\! \int_{\BS^2}\!\frac{n(\eta,\sigma)}{\left(1-(r/t)\angles{\omega}{\sigma}\right)^2}\,dS(\sigma)\,d\eta\chi_{in}'(q)\chi_{ex}(\frac{t}{12})+\mathcal{R}_{\Romanupper{1}}\\
&=\frac{1}{2\pi}\frac{t-r}{t^2r}\int_{-\infty}^\infty\! \int_{\BS^2}\!\frac{n(\eta,\omega)}{\left(1-(r/t)\angles{\omega}{\sigma}\right)^2}\,dS(\sigma)\,d\eta\chi_{in}'(q)\chi_{ex}(\frac{t}{12})\\
&\qquad+\frac{1}{2\pi}\frac{t-r}{t^2r}\int_{-\infty}^\infty\! \int_{\BS^2}\!\frac{n(\eta, \sigma)-n(\eta,\omega)}{\left(1-(r/t)\angles{\omega}{\sigma}\right)^2}\,dS(\sigma)\,d\eta\chi_{in}'(q)\chi_{ex}(\frac{t}{12})+\mathcal{R}_{\Romanupper{1}}\\
&\equiv\Romanupper{1}_a+\Romanupper{1}_b+\mathcal{R}_{\Romanupper{1}}.
\end{align*}
We note that the cutoff $\chi_{ex}$ is used to remove the singularity at $r=0$ generated by differentiating $\chi_{in}$ and we can control 
\begin{equation}\label{eq:R_I}
\abs{Z^J\mathcal{R}_{\Romanupper{1}}}\lesssim\frac{\ep^2}{\jb{t+r}^m}\quad \text{for any}\quad m.
\end{equation}
It is clear that the term $\Romanupper{1}_a$ contributes the leading order behavior of the term $\Romanupper{1}$. In view of the fact that $\chi(\jb{q}/r)=1$ in the region $\text{supp }\chi_{in}'(q)\cap\text{supp }\chi_{ex}(t/12)$, we have 
\begin{align*}
\Romanupper{1}_a&=\frac{t-r}{t^2r}\int_{-\infty}^\infty\!n(\eta,\omega)\,d\eta\int_{-1}^1 \frac{1}{\left(1-(r/t)\sigma_1\right)^2}\,d\sigma_1\chi_{in}'(q)\chi_{ex}(\frac{t}{12})\\
&=\frac{2}{r(t+r)}\int_{-\infty}^\infty\!n(\eta,\omega)\,d\eta\chi_{in}'(q)\chi_{ex}(\frac{t}{12})\chi(\frac{\jb{q}}{r}).
\end{align*}
By the estimate ~\eqref{eq:estimate_of_phik} from Lemma ~\ref{lem:taylor_expansion_in_n}, it follows that
\begin{equation}\label{eq:I_b}
\abs{Z^J\Romanupper{1}_b}\lesssim\frac{\ep^2}{\jb{t+r}^3\jb{t-r}^m}\ln\frac{\jb{t+r}}{\jb{t-r}}\quad \text{for any}\quad m.
\end{equation}
Next, let us calculate the term $\Romanupper{2}$. Since $\Box \varphi=\frac{1}{r}(\partial_r-\partial_t)(\partial_r+\partial_t)(r\varphi)+\frac{1}{r^2}\Delta_\omega\varphi$, we see that
\begin{align*}
\Romanupper{2}&=\frac{1}{r}(\partial_r-\partial_t)(\partial_r+\partial_t)\Bigl(\frac{1}{2}\ln\frac{2r}{\jb{t-r}}\int_{-\infty}^{q}\!n(\eta, \omega)\,d\eta\chi\Bigl(\frac{\jb{q}}{r}\Bigr)\chi_{in}(q)\Bigr)\\
&\qquad+\frac{1}{r}(\partial_r-\partial_t)(\partial_r+\partial_t)\Bigl(\frac{1}{2}\ln\frac{2r}{\jb{t-r}}\int_{q}^{\infty}\!n(\eta, \omega)\,d\eta\chi\Bigl(\frac{\jb{q}}{r}\Bigr)\bigl(\chi_{in}(q)-1\bigr)\Bigr)\\
&\qquad+\frac{1}{2r^3}\ln\frac{2r}{\jb{t-r}}\chi\Bigl(\frac{\jb{q}}{r}\Bigr)\biggl(\int_{-\infty}^{q}\!\Delta_\omega n(\eta, \omega)\,d\eta\chi\Bigl(\frac{\jb{q}}{r}\Bigr)\chi_{in}(q)+\int_{q}^{\infty}\!\Delta_\omega n(\eta, \omega)\,d\eta\bigl(\chi_{in}(q)-1\bigr)\biggr)\\
&\equiv\Romanupper{2}_a+\Romanupper{2}_b+\Romanupper{2}_c.
\end{align*}
It is easy to see that
\[
\abs{\Romanupper{2}_c}\lesssim\frac{\ep^2}{\jb{t+r}^3\jb{t-r}^a}\ln\frac{\jb{t+r}}{\jb{t-r}}.
\]
For the term $\Romanupper{2}_a$ we have
\begin{align*}
\Romanupper{2}_a&=\frac{1}{2r}(\partial_r-\partial_t)\Bigl(\frac{1}{r}\int_{-\infty}^{q}\!n(\eta, \omega)\,d\eta\chi\Bigl(\frac{\jb{q}}{r}\Bigr)\chi_{in}(q)+\ln\frac{2r}{\jb{t-r}}\int_{-\infty}^{q}\!n(\eta, \omega)\,d\eta\chi'\Bigl(\frac{\jb{q}}{r}\Bigr)\frac{-\jb{t-r}}{r^2}\chi_{in}(q)\Bigr)\\
&=\frac{1}{r^2}n(q, \omega)\chi\Bigl(\frac{\jb{q}}{r}\Bigr)\chi_{in}(q)+\frac{1}{r^2}\int_{-\infty}^{q}\!n(\eta, \omega)\,d\eta\chi\Bigl(\frac{\jb{q}}{r}\Bigr)\chi'_{in}(q)-\frac{1}{2r^3}\int_{-\infty}^{q}\!n(\eta, \omega)\,d\eta\chi\Bigl(\frac{\jb{q}}{r}\Bigr)\chi_{in}(q)\\
&\qquad+\frac{1}{2r^2}\int_{-\infty}^{q}\!n(\eta, \omega)\,d\eta\chi'\Bigl(\frac{\jb{q}}{r}\Bigr)\Bigl(\frac{2(r-t)}{\jb{r-t}r}-\frac{\jb{r-t}}{r^2}\Bigr)\chi_{in}+\frac{1}{r}\ln\frac{2r}{\jb{t-r}}n(q, \omega)\chi'\Bigl(\frac{\jb{q}}{r}\Bigr)\frac{-\jb{t-r}}{r^2}\chi_{in}\\
&\qquad+\frac{1}{2r}\Bigl(\frac{1}{r}+\frac{2(t-r)}{\jb{t-r}^2}\Bigr)\int_{-\infty}^{q}\!n(\eta, \omega)\,d\eta\chi'\Bigl(\frac{\jb{q}}{r}\Bigr)\frac{-\jb{t-r}}{r^2}\chi_{in}(q)\\
&\qquad+\frac{1}{r}\ln\frac{2r}{\jb{t-r}}\int_{-\infty}^q\!n(q, \omega)\,d\eta\chi'\Bigl(\frac{\jb{q}}{r}\Bigr)\frac{-\jb{t-r}}{r^2}\chi'_{in}(q)\\
&\qquad+\frac{1}{2r}\ln\frac{2r}{\jb{t-r}}\int_{-\infty}^q\!n(q, \omega)\,d\eta\biggl(\Bigl(\frac{2(t-r)}{r^3}+\frac{\jb{r-t}^2}{r^4}\Bigr)\chi''+\Bigl(\frac{2\jb{t-r}}{r^3}+\frac{4(t-r)}{r^2}\Bigr)\chi'\biggr)\chi_{in}\\
&=\frac{1}{r^2}n(q, \omega)\chi\Bigl(\frac{\jb{q}}{r}\Bigr)\chi_{in}(q)+\frac{1}{r^2}\int_{-\infty}^{q}\!n(\eta, \omega)\,d\eta\chi\Bigl(\frac{\jb{q}}{r}\Bigr)\chi'_{in}(q)+\mathcal{R}_a.
\end{align*}
By similar calculation we can obtain 
\[
\Romanupper{2}_b=\frac{1}{r^2}n(q, \omega)\chi\Bigl(\frac{\jb{q}}{r}\Bigr)\bigl(1-\chi_{in}(q)\bigr)+\frac{1}{r^2}\int_{q}^{\infty}\!n(\eta, \omega)\,d\eta\chi\Bigl(\frac{\jb{q}}{r}\Bigr)\chi'_{in}(q)+\mathcal{R}_b.
\]
Owing to the assumption ~\eqref{eq:condition_on_norm} and the fact that in the support of $\chi$, $t\sim r\sim\jb{t+r}$, we conclude that 
\[
\abs{\mathcal{R}_a}+\abs{\mathcal{R}_b}\lesssim\frac{\ep^2}{\jb{t+r}^3\jb{t-r}^a}.
\]
Putting all of the above estimates together, we arrive at the estimate
\[
\Box\varPhi_1[n]=-\frac{1}{r^2}n(q, \omega)\chi\Bigl(\frac{\jb{q}}{r}\Bigr)+\mathcal{R}
\]
with
\[
\abs{\mathcal{R}}\lesssim\frac{\ep^2}{\jb{t+r}^3\jb{t-r}^a}\ln\frac{\jb{t+r}}{\jb{t-r}}.
\]
Therefore we have
\[
\abs{\Box \varPhi_0[n]}=\abs{\Box\varPhi[n]-\Box\varPhi_1[n]}=\abs{-\mathcal{R}}\lesssim\frac{\ep^2}{\jb{t+r}^3\jb{t-r}^a}\ln\frac{\jb{t+r}}{\jb{t-r}}.
\]
Moreover, thanks to the following estimates
\begin{equation}\label{eq:vectorfield_on_radiation}
\begin{split}
\abs{\partial_\mu n(q, \omega)}\leq \abs{\partial_qn}+\abs{\frac{\partial_\omega n}{r}}, \quad \abs{\Omega_{ij}n(q, \omega)}\leq\abs{\partial_\omega n}, \\
\quad Sn(q, \omega)=q\partial_q n, \quad \abs{\Omega_{0i} n(q, \omega)}\leq\abs{q\partial_qn}+\frac{t}{r}\abs{\partial_\omega n},
\end{split}
\end{equation}
the commutation identities ~\eqref{eq:com_iden} and the estimates ~\eqref{eq:R_I}, ~\eqref{eq:I_b},  it holds that 
\begin{equation}\label{eq:condition1}
\abs{\Box Z^J \varPhi_0[n]}\lesssim\frac{\ep^2}{\jb{t+r}^3\jb{t-r}^a}\ln\frac{\jb{t+r}}{\jb{t-r}}.
\end{equation}
Using the radial estimate from Lemma ~\ref{lem:radial_estimate} we obtain
\begin{equation}\label{eq:condition2}
\abs{Z^J\varPhi_0[n]}\lesssim
\begin{cases}
\ep^2\jb{t+r}\jb{t-r}^{-a}, \quad &0<a<1,\\
C_\mu\ep^2\jb{t+r}^{-1}\jb{q_+}^{-a}\jb{q_-}^{-\mu}, \quad &1\leq a<2,
\end{cases}
\end{equation}
with any $0<\mu<1$. Due to the estimates ~\eqref{eq:condition1} and ~\eqref{eq:condition2} we can exploit Lemma ~\ref{lem:existence_of_radiationfield} to prove the existence of $\Psi_0^\infty(q, \omega)$. The conclusion ~\eqref{eq:decay} follows from the estimate ~\eqref{eq:condition2}.
	\end{proof}

\subsection{Asymptotic Lorenz gauge condition}\label{subsec:asym_lorenz}

In view of the asymptotics result in ~\cite{CKL19} and the refined asymptotics we established in Proposition ~\ref{prop:refined_asymptotics}, we obtain that in the wave zone
\begin{equation}\label{eq:restate_asymtotics}
\begin{split}
\phi(t, r\omega)&\sim e^{-i\frac{\BFq}{4\pi}\ln(1+r)}\frac{\Phi(q, \omega)}{r}:=\phi_{asy}\\
A_{\alpha, asy}&=\frac{1}{4\pi}\int_{-\infty}^\infty\!\int_{\BS^2}\!\frac{\mathcal{J}_{\alpha}(\eta, \sigma)}{t-r\angles{\omega}{\sigma}}\, dS(\sigma)\,dq\chi_{in}(q)-\frac{1}{2r}\ln\frac{\jb{t+r}}{\jb{t-r}}\int_{-\infty}^q\!\mathcal{J}_{\alpha}(\eta, \omega)\, d\eta\chi_{in}(q)\\
&\quad+\Bigl(\frac{1}{2r}\ln\frac{\jb{t+r}}{\jb{t-r}}\int_{q}^\infty\!\mathcal{J}_{\alpha}(\eta, \omega)\, d\eta\Bigr)\bigl(1-\chi_{in}(q)\bigr)+\frac{\mathcal{A}_\alpha(q, \omega)}{r}+\frac{\BFq}{4\pi r }\chi_{ex}(q)\delta_{\alpha 0}.
\end{split}
\end{equation}
where $\mathcal{J}_\alpha(\eta, \omega)=L_\alpha(\omega)\Im(\Phi(\eta, \omega)\overline{\partial_q\Phi(\eta, \omega)})$, and $\chi_{in}$ and $\chi_{ex}$ are two smooth cutoffs such that $\chi_{in}(q)=1$ when $q\leq -1$ and $\chi_{in}(q)=0$
when $q\geq -1/2$, while $\chi_{ex}(q)=1$ for $q\geq 1$ and $\chi_{ex}(q)=0$ for $q\leq 1/2$.

Meanwhile, the \emph{asymptotic Lorenz gauge condition} impose a requirement that the radiation set $(\Phi, \mathcal{A}_\alpha)$ should be chosen such that modulo the terms of order $\mathcal{O}(r^{-2}\ln r)$, 
$\partial^\alpha A_{\alpha, asy}(t-r, r\omega)=0$ for each $q$ and large $r$.
\begin{prop}\label{prop:asy_Lorenz}
Suppose the  radiation set $(\Phi(q, \omega), \mathcal{A}_\alpha(q, \omega))$ satisfies the asymptotic Lorenz gauge condition, then we have
\begin{equation}
\partial_q\Bigl(L^\alpha(\omega)\mathcal{A}_{\alpha}(q, \omega)+\frac{\BFq}{4\pi}\chi_{in}(q)+\frac{\BFq}{4\pi}\chi_{ex}(q)\Bigr)=0.
\end{equation}
\end{prop}
\begin{proof}
Since we do all of the calculation in the region when $r$ is large, we may assume that $r>1$. We compute that 
\begin{align*}
&\partial^\alpha A_{\alpha, asy}(t, r\omega)\\
&\qquad=\partial^\alpha\Bigl(\frac{1}{4\pi}\int_{-\infty}^\infty\!\int_{\BS^2}\!\frac{\mathcal{J}_{\alpha}(\eta, \sigma)}{t-r\angles{\omega}{\sigma}}\, dS(\sigma)\,d\eta\chi_{in}(q)\Bigr)-\partial^\alpha\Bigl(\frac{1}{2r}\ln\frac{\jb{t+r}}{\jb{t-r}}\int_{-\infty}^q\!\mathcal{J}_{\alpha}(\eta, \omega)\, d\eta\chi_{in}(q)\Bigr)\\
&\qquad\qquad+\partial^\alpha\Bigl(\frac{1}{2r}\ln\frac{\jb{t+r}}{\jb{t-r}}\int_{q}^\infty\!\mathcal{J}_{\alpha}(\eta, \omega)\, d\eta\bigl(1-\chi_{in}(q)\bigr)\Bigr)+\partial^\alpha\Bigl(\frac{\BFq}{4\pi r}\chi_{ex}(q)\delta_{\alpha 0}+\frac{\mathcal{A}_\alpha(q, \omega)}{r}\Bigr)\\
&\qquad\equiv\Romanupper{1}+\Romanupper{2}+\Romanupper{3}+\Romanupper{4}.
\end{align*}
For the term $\Romanupper{1}$, we have
\begin{align*}
\Romanupper{1}&=\frac{1}{4\pi}\int_{-\infty}^\infty\!\int_{\BS^2}\!\frac{L^\alpha(\sigma)\mathcal{J}_{\alpha}(\eta, \sigma)}{\left(t-r\angles{\omega}{\sigma}\right)^2}\, dS(\sigma)\,d\eta\chi_{in}(q)+\frac{1}{4\pi}\int_{-\infty}^\infty\!\int_{\BS^2}\!\frac{L^\alpha(\omega)\mathcal{J}_{\alpha}(\eta, \sigma)}{t-r\angles{\omega}{\sigma}}\, dS(\sigma)\,d\eta\chi_{in}'(q)\\
&=\frac{1}{4\pi}\int_{-\infty}^\infty\!\int_{\BS^2}\!\frac{\left(1-\angles{\omega}{\sigma}\right)\mathcal{J}_0(\eta, \sigma)}{t-r\angles{\omega}{\sigma}}\, dS(\sigma)\,d\eta\chi_{in}'(q)\\
&=\frac{1}{4\pi r}\int_{-\infty}^\infty\!\int_{\BS^2}\!\mathcal{J}_0(\eta, \sigma)\,d\eta\,dS(\sigma)\chi'_{in}(q)+\frac{r-t}{4\pi r}\int_{-\infty}^\infty\!\int_{\BS^2}\!\frac{\mathcal{J}_0(\eta, \sigma)}{t-r\angles{\omega}{\sigma}}\, dS(\sigma)\,d\eta\chi_{in}'(q)\\
&=\frac{\BFq}{4\pi r}\chi'_{in}(q)+\mathcal{R}_{\Romanupper{1}}.
\end{align*}
By the estimate ~\eqref{eq:estimate_of_phik} from Lemma ~\ref{lem:taylor_expansion_in_n}, it follows that
\begin{align*}
\mathcal{R}_{\Romanupper{1}}&=\frac{r-t}{4\pi tr}\int_{-\infty}^\infty\!\int_{\BS^2}\!\frac{\mathcal{J}_0(\eta, \omega)}{1-(r/t)\angles{\omega}{\sigma}}\, dS(\sigma)\,d\eta\chi_{in}'(q)+\frac{r-t}{4\pi r}\int_{-\infty}^\infty\!\int_{\BS^2}\!\frac{\mathcal{J}_0(\eta, \sigma)-\mathcal{J}_0(\eta, \omega)}{t-r\angles{\omega}{\sigma}}\, dS(\sigma)\,d\eta\chi_{in}'(q)\\
&=\frac{r-t}{4\pi r^2}\ln\frac{t+r}{t-r}\int_{-\infty}^\infty\!\mathcal{J}_0(\eta, \omega)\,d\eta\chi_{in}'(q)+\frac{r-t}{4\pi r}\int_{-\infty}^\infty\!\int_{\BS^2}\!\frac{\mathcal{J}_0(\eta, \sigma)-\mathcal{J}_0(\eta, \omega)}{t-r\angles{\omega}{\sigma}}\, dS(\sigma)\,d\eta\chi_{in}'(q)\\
&\lesssim\frac{1}{r^2}\left(1+\ln r\right).
\end{align*}
By straightforward calculation we find that
\begin{align*}
\Romanupper{2}&=\frac{1}{2r^2}\ln\frac{\jb{t+r}}{\jb{t-r}}\int_{-\infty}^q\!-\mathcal{J}_0(\eta, \omega)\, d\eta\chi_{in}(q)+\frac{1}{2r}\frac{(r-t)}{\jb{t-r}^2}\int_{-\infty}^q\!L^\alpha(\omega)\mathcal{J}_{\alpha}(\eta, \omega)\, d\eta\chi_{in}(q)\\
&\qquad+\frac{1}{2r}\frac{(t+r)}{\jb{t+r}^2}\int_{-\infty}^q\!2\mathcal{J}_0(\eta, \omega)\, d\eta\chi_{in}(q)\\
&\qquad-\frac{1}{2r}\ln\frac{\jb{t+r}}{\jb{t-r}}L^\alpha(\omega)\mathcal{J}_{\alpha}(\eta, \omega)\chi_{in}(q)-\frac{1}{2r}\ln\frac{\jb{t+r}}{\jb{t-r}}\int_{-\infty}^q\!L^\alpha(\omega)\mathcal{J}_{\alpha}(\eta, \omega)\, d\eta\chi'_{in}(q)\\
&\lesssim\frac{1}{r^2}\ln \frac{\jb{t+r}}{\jb{t-r}}
\end{align*}
and
\begin{align*}
\Romanupper{3}&=\frac{1}{2r^2}\ln\frac{\jb{t+r}}{\jb{t-r}}\int_{q}^\infty\!\mathcal{J}_0(\eta, \omega)\, d\eta(1-\chi_{in})-\frac{1}{2r}\frac{(r-t)}{\jb{t-r}^2}\int_q^\infty\!L^\alpha(\omega)\mathcal{J}_{\alpha}(\eta, \omega)\, d\eta(1-\chi_{in})\\
&\qquad-\frac{1}{2r}\frac{(t+r)}{\jb{t+r}^2}\int_q^\infty\!2\mathcal{J}_0(\eta, \omega)\, d\eta(1-\chi_{in})\\
&\qquad-\frac{1}{2r}\ln\frac{\jb{t+r}}{\jb{t-r}}L^\alpha(\omega)\mathcal{J}_{\alpha}(\eta, \omega)(1-\chi_{in})-\frac{1}{2r}\ln\frac{\jb{t+r}}{\jb{t-r}}\int_{-\infty}^q\!L^\alpha(\omega)\mathcal{J}_{\alpha}(\eta, \omega)\, d\eta\chi'_{in}(q)\\
&\lesssim\frac{1}{r^2}\ln \frac{\jb{t+r}}{\jb{t-r}}.
\end{align*}
Now let us turn to the term $\Romanupper{4}$
\begin{align*}
\Romanupper{4}&=\frac{\BFq}{4\pi r}\chi'_{ex}(q)+L^\alpha(\omega)\frac{\partial_q\mathcal{A}(q, \omega)}{r}+\frac{\slashed{\partial}^i\mathcal{A}_i(q, \omega)}{r}-\frac{\omega^i\mathcal{A}_i(q, \omega)}{r^2}\\
&=\frac{\BFq}{4\pi r}\chi'_{ex}(q)+L^\alpha(\omega)\frac{\partial_q\mathcal{A}(q, \omega)}{r}+\mathcal{O}(\frac{1}{r^2}).\\
\end{align*}
Putting all of the above estimates together and modulo the terms of order $\mathcal{O}(r^{-2}\ln r)$, we arrive at
\[
\partial^\alpha A_{\alpha, asy}=\frac{1}{r}\partial_q\Bigl(\frac{\BFq}{4\pi }\chi_{in}(q)+L^\alpha(\omega)\mathcal{A}_\alpha(q, \omega)+\frac{\BFq}{4\pi }\chi_{ex}(q)\Bigr).
\]
This finishes the proof of Proposition ~\ref{prop:asy_Lorenz}.
\end{proof}
From the decay of $\mathcal{A}_\alpha$ in $q$, we see that \[\lim_{q\to\infty}\bigl(\frac{\BFq}{4\pi }\chi_{in}(q)+L^\alpha(\omega)\mathcal{A}_\alpha(q, \omega)+\frac{\BFq}{4\pi }\chi_{ex}(q)\bigr)=\frac{\BFq}{4\pi}
\]
from which we can infer that 
\begin{equation}\label{eq:expression_for_A_L}
\mathcal{A}_L(q, \omega)=L^\alpha(\omega)\mathcal{A}_\alpha(q, \omega)=
\begin{cases}
\frac{\BFq}{4\pi}(1-\chi_{in}(q)), \quad &q<0,\\
\frac{\BFq}{4\pi}(1-\chi_{ex}(q)), \quad &q\geq0.
\end{cases}
\end{equation}

\section{Reduced MKG equations}\label{sec:4}

In this section our goal is to solve the reduced MKG equations ~\eqref{eq:rmkg} from infinity with the asymptotics given of the form ~\eqref{eq:asym_of_two} in wave zone and ~\eqref{eq:K_timelike} in the far interior.

We construct the approximate solution associated to the asymptotics ~\eqref{eq:asym_of_two} and ~\eqref{eq:K_timelike} as follows
\begin{equation}\label{eq:app_solution}
\begin{split}
\phi_{app}(t, r\omega)&=e^{-i\frac{\BFq}{4\pi}\ln r}\frac{\Phi(q, \omega)}{r}\chi\Bigl(\frac{\jb{q}}{r}\Bigr)\\
A_{\alpha, app}(t, r\omega)&=\frac{1}{4\pi}\int_{-\infty}^{\infty}\!\int_{\BS^2}\!\frac{\mathcal{J}_{\alpha}(\eta, \sigma)}{t-r\angles{\omega}{\sigma}}\, dS(\sigma)\,d\eta\chi_{ex}(\frac{t}{12})\chi_{in}(q)\\
&\qquad-\frac{1}{2r}\ln\frac{2r}{\jb{t-r}}\int_{-\infty}^q\!\mathcal{J}_{\alpha}(\eta, \omega)\, d\eta\chi\Bigl(\frac{\jb{q}}{r}\Bigr)\chi_{in}(q)\\
&\qquad+\frac{1}{2r}\ln\frac{2r}{\jb{t-r}}\int_{q}^\infty\!\mathcal{J}_{\alpha}(\eta, \omega)\, d\eta\chi\Bigl(\frac{\jb{q}}{r}\Bigr)\Bigl(1-\chi_{in}(q)\Bigr)\\
&\qquad+\frac{\BFq}{4\pi r}\delta_{\alpha 0}\chi_{ex}(q)+\frac{\mathcal{A}_{\alpha}(q, \omega)}{r}\chi\Bigl(\frac{\jb{q}}{r}\Bigr).
\end{split}
\end{equation}
with $(\Phi, \mathcal{A}_\alpha)$ satisfying the \emph{asymptotic Lorenz gauge condition} and the smallness condition on the norm defined in ~\eqref{eq:norm_of_radiation}. Here $\chi(s)=1$ when $s\leq 1/2$ and $\chi(s)=0$ when $s\geq 3/4$, which is used to localize in the wave zone, away from the origin. 

Then we set a smooth cutoff $\widetilde{\chi}$ such that $\widetilde{\chi}(s)=1$ when $s\leq 1$, and $\widetilde{\chi}(s)=0$ when $s\geq 2$. We expect to find a sequence of solutions $(\phi_T, A_{\alpha T})=(\phi_{app}+u_T, A_{\alpha, app}+v_{\alpha T})$ to the reduced MKG equations ~\eqref{eq:rmkg}. To achieve this, we consider the following equations for $(u_T, v_{\alpha T})$ with trivial data at $t=2T$ 
\begin{equation}\label{eq:rmkg_with_cutoff}
\begin{split}
\Box u &=\widetilde{\chi}\bigl(\frac{t}{T}\bigr)\Bigl(-2iA^\alpha\partial_\alpha\phi+A^\alpha A_\alpha \phi-\Box\phi_{app}\Bigr)\\
\Box v_{\alpha} &=\widetilde{\chi}\big(\frac{t}{T}\big)\Bigl(-J_\alpha-\Box A_{\alpha, app}\Bigr), \quad J_{\alpha}=\Im(\phi\overline{(\partial_\alpha+iA_\alpha)\phi}).
\end{split}
\end{equation}

We will see later on that $(\phi_{app}, A_{\alpha, app})$ is a ``good approximate solution" in the sense that module terms of order $\mathcal{O}(r^{-3}\ln r)$, $(\phi_{app}, A_{\alpha, app})$ solves the reduced MKG equations ~\eqref{eq:rmkg}, i.e.
\begin{align*}
\begin{split}
\Box \phi_{app} &\sim-2iA^{\alpha, app}\partial_\alpha\phi_{app}+A^{\alpha, app} A_{\alpha, app}\phi_{app}\\
\Box A_{\alpha, app} &\sim-\Im(\phi_{app}\overline{\partial_\alpha\phi_{app}})+\abs{\phi_{app}}^2A_{\alpha, app}.
\end{split}
\end{align*}
We first point out some basic results
\begin{itemize}
	\item For any vector field $Z$ and multiindex $I$, we have $Z^I u_T=Z^I v_{\alpha T}=0$ at $t=2T$.
	\item Since $\Box\phi_{app}=0$ and $\Box A_{\alpha, app}=0$ when $r\geq 4t$, by finite speed of propagation, we obtain that $u_T=0$ and $v_{\alpha T}=0$ when $r>6T-t$. This also indicates that $Z^I\widetilde{\chi}(t/T)\sim 1$ in the support of $u_T$ or $v_{\alpha T}$.
	\item $(\phi_T, A_{\alpha T})=(\phi_{app}+u_T, A_{\alpha, app}+v_{\alpha T})$ is an exact solution to the reduced MKG equations ~\eqref{eq:rmkg} when $t\leq T$.
\end{itemize}

The main result of this section is the following proposition.
\begin{prop}\label{prop:estimate_for_rmkg}
	Suppose the assumptions of Theorem ~\ref{thm:mainthm} are in place.
	Then there exists a small absolute constant $\ep_0$, which depends on $N, \mu, \gamma$, such that for any radiation field set $(\Phi, \mathcal{A}_\alpha)$ given above which satisfies the asymptotic Lorenz gauge condition and the smallness condition
	\[
	\ep:=\norm{\Phi}_{N, \infty, \gamma}+\norm{\mathcal{A}_\alpha}_{N, \infty, \gamma}\leq\ep_0,
	\]
~\eqref{eq:rmkg_with_cutoff} has a solution $(u, v_\alpha)=(u_T, v_{\alpha T})$ for $0\leq t\leq 2T$, where $(u, v_\alpha)$ has trivial data at $t=2T$. And for all $0\leq t\leq 2T$, we have 
	\begin{align}\label{eq:energy_in_prop}
	\norm{w^{\frac{1}{2}}Z^Iu_T}_{L^2(\Sigma_t)}+\norm{w^{\frac{1}{2}}Z^Iv_{\alpha T}}_{L^2(\Sigma_t)}&\leq C\ep\jb{t}^{-\frac{\gamma}{2}+\frac{\mu}{2}+\frac{1}{4}} \quad \text{for} \quad \abs{I}\leq N-2,\\\label{eq:decay_in_prop}
	\abs{w^{\frac{1}{2}}Z^Iu_T}+\abs{w^{\frac{1}{2}}Z^Iv_{\alpha T}}&\leq	C\ep\jb{t+r}^{-1}\jb{t-r}^{-\frac{1}{2}}\jb{t}^{-\frac{\gamma}{2}+\frac{\mu}{2}+\frac{1}{4}} \quad \text{for}  \quad \abs{I}\leq N-4
	\end{align}
		where $C$ only depends on $\gamma, \mu, N$ and 
	\[
	w(q)=\begin{cases}
	1+(1+\abs{q})^\mu, \quad &q<0,\\
	1+(1+\abs{q})^{-\gamma+\frac{1}{2}}, \quad &q\geq 0.
	\end{cases}
	\]
\end{prop}
The proof of Proposition ~\ref{prop:estimate_for_rmkg} proceeds via a continuity argument. More specifically, we define these two bootstrap quantities
\begin{equation}
\mathcal{E}_k(t)=\sup_{t\leq \tau\leq 2T}\sum_{\abs{I}\leq k}E^w[Z^Iv_\alpha](\tau)+E^w[Z^Iu](\tau), \quad k=N-3
\end{equation}
and 
\begin{equation}
S_k(t)=\sum_{\abs{I}\leq k}\int_t^{2T}\!\int_{\Sigma_\tau}\!\jb{t+r}^2\abs{\frac{1}{r}L(rZ^Iu)}^2+\jb{t+r}^2\abs{\frac{1}{r}L(rZ^Iv_{\alpha})}^2\abs{w'}\,dx\,d\tau, \quad k=N-3.
\end{equation}
with $w$ defined in Proposition ~\ref{prop:estimate_for_rmkg}. We clearly have 
\begin{equation}
\abs{w'}\sim\mu w(1+\abs{q})^{-1} \quad \text{for} \quad q<0;\quad\quad \abs{w'}\sim(\gamma-1/2)w(1+\abs{q})^{-\gamma-1/2}
\quad\text{for}\quad q>0.
\end{equation}
Then we apply Corollary ~\ref{cor:energyestimate_for_u} with $\delta=\gamma-1/2$ and $\nu=\mu$ to $v_{\alpha}$ and $u$ to close the bootstrap argument.

\medskip

We commute the equations ~\eqref{eq:rmkg_with_cutoff} with $Z^I$ for $\abs{I}\leq k=N-3$ and rewrite it as
\begin{equation}\label{eq:expansion_of_eqn_for_v}
\begin{split}
\Box v_\alpha&=\bigl(Z^I+\sum_{\abs{J}\leq\abs{I}-1}c_J^IZ^J\bigr)\biggl[\widetilde{\chi}(\frac{t}{T})\Bigl(-\Im(\phi_{app}\overline{\partial_\alpha\phi_{app}})+\abs{\phi_{app}}^2A_{\alpha, app}-\Box A_{\alpha, app}\Bigr)\biggr]\\
&\qquad+\bigl(Z^I+\sum_{\abs{J}\leq\abs{I}-1}c_J^IZ^J\bigr)\biggl[\widetilde{\chi}(\frac{t}{T})\Bigl(-\Im(u\overline{\partial_\alpha\phi_{app}})-\Im(\phi_{app}\overline{\partial_\alpha u})-\Im(u\overline{\partial_\alpha u})\Bigr)\biggr]\\
&\qquad+\bigl(Z^I+\sum_{\abs{J}\leq\abs{I}-1}c_J^IZ^J\bigr)\biggl[\widetilde{\chi}(\frac{t}{T})\Bigl(\Re(\phi_{app}\overline{u})A_{\alpha, app}+\abs{\phi_{app}}^2v_\alpha\Bigr)\biggr]\\
&\qquad+\bigl(Z^I+\sum_{\abs{J}\leq\abs{I}-1}c_J^IZ^J\bigr)\biggl[\widetilde{\chi}(\frac{t}{T})\Bigl(\Re(\phi_{app}\overline{u})v_\alpha+\abs{u}^2A_{\alpha, app}\Bigr)\biggr]\\
&\qquad+\bigl(Z^I+\sum_{\abs{J}\leq\abs{I}-1}c_J^IZ^J\bigr)\Bigl[\widetilde{\chi}(\frac{t}{T})\abs{u}^2v_\alpha\Bigr]\\
&\equiv \mathcal{V}_{\alpha 1}+\mathcal{V}_{\alpha 2}+\mathcal{V}_{\alpha 3}+\mathcal{V}_{\alpha 4}+\mathcal{V}_{\alpha 5},
\end{split}
\end{equation}
and
\begin{equation}\label{eq:expansion_of_eqn_for_u}
\begin{split}
\Box Z^Iu=&\bigl(Z^I+\sum_{\abs{J}\leq\abs{I}-1}c_J^IZ^J\bigr)\biggl[\widetilde{\chi}(\frac{t}{T})\Big(-2iA^\alpha_{app}\partial_\alpha\phi_{app}+A^\alpha_{app}A_{\alpha, app}\phi_{app}-\Box\phi_{app}\Big)\biggr]\\
&\qquad+\bigl(Z^I+\sum_{\abs{J}\leq\abs{I}-1}c_J^IZ^J\bigr)\biggl[\widetilde{\chi}(\frac{t}{T})\Big(-2iv^\alpha\partial_\alpha\phi_{app}-2iA^\alpha_{app}\partial_\alpha u-2iv^\alpha\partial_\alpha u\Big)\biggr]\\
&\qquad+\bigl(Z^I+\sum_{\abs{J}\leq\abs{I}-1}c_J^IZ^J\bigr)\biggl[\widetilde{\chi}(\frac{t}{T})\Big(v^\alpha A_{\alpha, app}\phi_{app}+A^\alpha_{app}v_\alpha\phi_{app}+A^\alpha_{app}A_{\alpha, app}u\Big)\biggr]\\
&\qquad+\bigl(Z^I+\sum_{\abs{J}\leq\abs{I}-1}c_J^IZ^J\bigr)\biggl[\widetilde{\chi}(\frac{t}{T})\Big(v^\alpha v_\alpha\phi_{app}+v^\alpha A_{\alpha, app} u+A^\alpha_{app}v_\alpha u\Bigr)\biggr]\\
&\qquad+\bigl(Z^I+\sum_{\abs{J}\leq\abs{I}-1}c_J^IZ^J\bigr)\Bigl[\widetilde{\chi}(\frac{t}{T})v^\alpha v_\alpha u\Bigr]\\
&\equiv\mathcal{U}_1+\mathcal{U}_2+\mathcal{U}_3+\mathcal{U}_4+\mathcal{U}_5.
\end{split}
\end{equation}

For ease of notation in the remainder of this section, we denote the norm $\norm{\Phi}_{k, \infty, \gamma}+\norm{\mathcal{A}_j}_{k, \infty, \gamma}$ by $D_k$ which is a small constant. We also make a convention in the sequel that  the implicit constant in $A\lesssim B$ depends only on $\gamma$, $\mu$ and $N$. Without explicitly mentioning it, we will frequently use the Leibniz's rule ~\eqref{eq:leibniz}, the commutation identities ~\eqref{eq:com_iden}, the estimate ~\eqref{eq:vectorfield_on_radiation} and the fact that $t\sim r\sim\jb{t+r}$ in the support of $\chi$ and $\jb{t-r}\sim\jb{t+r}$ in the support of $\chi'$.

\subsection{Estimates for the approximate solutions}

We first analyze the error terms generated by the approximate solutions, namely $\mathcal{V}_{\alpha 1}$ and $\mathcal{U}_1$ in ~\eqref{eq:expansion_of_eqn_for_v} and ~\eqref{eq:expansion_of_eqn_for_u} respectively. 
\begin{lem}\label{lem:box_of_app}
	We have
	\begin{align}\label{eq:estimate_of_phiapp}
	\biggl\lvert Z^I\Bigl(\Box \phi_{app}+ie^{-i\frac{\BFq}{4\pi}\ln r}\frac{\BFq}{2\pi }\frac{\partial_q\Phi}{r^2}\chi\bigl(\frac{\jb{t-r}}{r}\bigr)\Bigr)\biggr\rvert&\lesssim\frac{D_{\abs{I}+2}}{\jb{t+r}^3\jb{t-r}^\gamma}\chi\Bigl(\frac{\jb{q}}{r}\Bigr),\\\label{eq:estimate_of_Aapp}
		\biggl\lvert Z^I\Bigl(\Box A_{\alpha, app}+\frac{1}{r^2}\mathcal{J}_\alpha\chi\bigl(\frac{\jb{t-r}}{r}\bigr)\Bigr)\biggr\rvert&\lesssim\frac{D_{\abs{I}+2}}{\jb{t+r}^3\jb{t-r}^\gamma}\chi\bigl(\frac{\jb{q}}{r}\bigr)+\frac{D^2_{\abs{I}+1}}{\jb{t+r}^m}\\
		&\qquad+\frac{D_{\abs{I}+3}^2}{\jb{t+r}^3\jb{t-r}^{2\gamma}}\Bigl(1+\ln\frac{\jb{t+r}}{\jb{t-r}}\Bigr)\chi\bigl(\frac{\jb{q}}{r}\bigr)\notag
	\end{align}
	for any $m>0$.
\end{lem}
\begin{proof}
We first consider $\phi_{app}$. To this end we calculate
\begin{align*}
\Box \phi_{app}&=-\frac{1}{r}\underline{L}L\Bigl(e^{-i\frac{\BFq}{4\pi}\ln r}\Phi(q, \omega)\chi\bigl(\frac{\jb{q}}{r}\bigr)\Bigr)+\frac{1}{r^3}e^{-i\frac{\BFq}{4\pi}\ln r}\Delta_\omega\Phi(q, \omega)\chi\bigl(\frac{\jb{q}}{r}\bigr)\\
&=-ie^{-i\frac{\BFq}{4\pi}\ln r}\frac{\BFq}{2\pi}\frac{\partial_q\Phi}{r^2}\chi\bigl(\frac{\jb{q}}{r}\bigr)+ie^{-i\frac{\BFq}{4\pi}\ln r}\left(\frac{\BFq}{4\pi }+\frac{i\BFq^2}{16\pi^2}\right)
\frac{\Phi}{r^3}\chi\bigl(\frac{\jb{q}}{r}\bigr)+\frac{1}{r^3}e^{-i\frac{\BFq}{4\pi}\ln r}\Delta_\omega\Phi(q, \omega)\chi\bigl(\frac{\jb{q}}{r}\bigr)\\
&\qquad+ie^{-i\frac{\BFq}{4\pi}\ln r}\frac{\BFq}{4\pi}\frac{\Phi}{r^2}\chi'\bigl(\frac{\jb{q}}{r}\bigr)\Bigl(\frac{\jb{t-r}}{r^2}+\frac{2(t-r)}{\jb{t-r}r}\Bigr)+ie^{-i\frac{\BFq}{4\pi}\ln r}\frac{\BFq}{4\pi}\frac{\Phi}{r^2}\chi'\bigl(\frac{\jb{q}}{r}\bigr)\frac{\jb{t-r}}{r^2}\\
&\qquad-e^{-i\frac{\BFq}{4\pi}\ln r}\frac{2\partial_q\Phi}{r}\chi'\bigl(\frac{\jb{q}}{r}\bigr)\frac{\jb{t-r}}{r^2}+e^{-i\frac{\BFq}{4\pi}\ln r}\frac{\Phi}{r}\underline{L}\Bigl(\chi'\bigl(\frac{\jb{q}}{r}\bigr)\frac{\jb{t-r}}{r^2}\Bigr).
\end{align*}
Then the first estimate ~\eqref{eq:estimate_of_phiapp} follows from the formula above.

Next let us turn to $A_{\alpha, app}$. We rewrite $\Box A_{\alpha, app}$ as
\[
\Box A_{\alpha, app}\equiv \Romanupper{1}+\Romanupper{2}
\]
where
\begin{align*}
\Romanupper{1}&=\Box \frac{1}{4\pi}\int_{-\infty}^{\infty}\!\int_{\BS^2}\!\frac{n(\eta, \sigma)}{t-r\angles{\omega}{\sigma}}\, dS(\sigma)\,d\eta\chi_{ex}(\frac{t}{12})\chi_{in}(q)\\
&\qquad-\Box\frac{1}{2r}\ln\frac{2r}{\jb{t-r}}\int_{-\infty}^q\!n(\eta, \omega)\, d\eta\chi\Bigl(\frac{\jb{q}}{r}\Bigr)\chi_{in}(q)+\frac{1}{2r}\ln\frac{2r}{\jb{t-r}}\int_{q}^\infty\!n(\eta, \omega)\, d\eta\chi\Bigl(\frac{\jb{q}}{r}\Bigr)\Bigl(\chi_{in}(q)-1\Bigr)\\
\end{align*}
and
\[
\Romanupper{2}=\Box \frac{\BFq}{4\pi r}\delta_{\alpha 0}\chi_{ex}(q)+\Box\frac{\mathcal{A}_{\alpha}(q, \omega)}{r}\chi\bigl(\frac{\jb{q}}{r}\bigr).
\]
By the calculation and estimates we did in the proof of Proposition ~\eqref{prop:refined_asymptotics} we obtain for any $m>0$ that
\[
\biggl\lvert Z^J\Bigl(\Romanupper{1}+\frac{1}{r^2}\mathcal{J}_\alpha(q, \omega)\chi\bigl(\frac{\jb{t-r}}{r}\bigr)\Bigr)\biggl\rvert
\lesssim\frac{D_{\abs{I}+1}^2}{\jb{t+r}^m}
+\frac{D_{\abs{I}+3}^2}{\jb{t+r}^3\jb{t-r}^{2\gamma}}\Bigl(1+\ln\frac{\jb{t+r}}{\jb{t-r}}\Bigr)\chi\bigl(\frac{\jb{q}}{r}\bigr)
\]
So we are left with calculating $\Romanupper{2}$. Indeed, we have
\begin{align*}
\Romanupper{2}&=-\frac{1}{r}\underline{L}L\Bigl(\frac{\BFq}{4\pi }\delta_{\alpha 0}\chi_{ex}(q)+\mathcal{A}_{\alpha}(q, \omega)\chi\bigl(\frac{\jb{q}}{r}\bigr)\Bigr)+\frac{1}{r^3}\Delta_\omega\mathcal{A}_{\alpha}(q, \omega)\chi\bigl(\frac{\jb{q}}{r}\bigr)\\
&=-\frac{\mathcal{A}_\alpha}{r}\underline{L}L\Bigl(\chi\bigl(\frac{\jb{q}}{r}\bigr)\Bigr)+\frac{2\partial_q\mathcal{A}_\alpha}{r}L\Bigl(\chi\bigl(\frac{\jb{q}}{r}\bigr)\Bigr)+\frac{1}{r^3}\Delta_\omega\mathcal{A}_{\alpha}(q, \omega)\chi\bigl(\frac{\jb{q}}{r}\bigr).
\end{align*}
Then we find that 
\[
\abs{Z^J\Romanupper{2}}\lesssim\frac{D_{\abs{I}+2}}{\jb{t+r}^3\jb{t-r}^\gamma}.
\]
Putting things together yields the second estimate ~\eqref{eq:estimate_of_Aapp}.
\end{proof}
In the next lemma we establish some decay estimates for $A_{\alpha,app}$.
\begin{lem}\label{lem:estimate_of_Aapp}
We have
\begin{equation}\label{eq:estimate_of_Aalpha_app}
\abs{Z^IA_{\alpha, app}}\lesssim_c
\begin{cases}
\frac{D_1^2}{\jb{t+r}}, \quad &t\geq cr,\\
\frac{D^2_{\abs{I}+2}}{\jb{t+r}}\ln\frac{\jb{t+r}}{\jb{t-r}}+\frac{D_{\abs{I}}}{\jb{t+r}\jb{t-r}^\gamma}\chi\bigl(\frac{\jb{q}}{r}\bigr), \quad &t<cr\text{ and } t-r\geq 1/2,\\
\Bigl(\frac{D^2_{\abs{I}+1}}{\jb{t+r}\jb{t-r}^{2\gamma}}\ln\frac{\jb{t+r}}{\jb{t-r}}+\frac{D_{\abs{I}}}{\jb{t+r}\jb{t-r}^\gamma}\Bigr)\chi\bigl(\frac{\jb{q}}{r}\bigr)+\frac{D_1^2\delta_{\alpha 0}}{4\pi\jb{t+r}}, \quad &t-r< 1/2.
\end{cases}
\end{equation}	
for any $c>1$. In addition,  for the corresponding null decomposition, we have 
\begin{equation} \label{eq:decomposition_of_A_L}
A_{L, app}=\frac{\BFq}{4\pi r}+\mathcal{R}_L \quad \text{for}\quad r\geq 1/4
\end{equation}
where
\begin{equation}\label{eq:estimate_for_remainder_of_A_L}
\abs{Z^I\mathcal{R}_L}\lesssim\frac{D^2_{\abs{I}+2}\jb{t-r}}{\jb{t+r}^2}\ln\frac{\jb{t+r}}{\jb{t-r}}\chi_{in}(q)+\frac{\BFq}{\jb{t+r}^m} \quad\text{for}\quad t<4r-1\quad\text{and any}\quad m>0.
\end{equation}
For $t\leq 4r-1$, we also have
\begin{equation}\label{eq:estimate_for_A_B}
\begin{split}
\abs{Z^IA_{e_B, app}}&\lesssim\frac{D^2_{\abs{I}+2}}{\jb{t+r}}\chi_{in}(q)+\frac{D_{\abs{I}}}{\jb{t+r}\jb{t-r}^\gamma}\chi\bigl(\frac{\jb{q}}{r}\bigr),\\
\abs{Z^IA_{\underline{L}, app}}&\sim \abs{Z^IA_{\alpha, app}}.
\end{split}
\end{equation}
\end{lem}
\begin{proof}
	Observing the formula ~\eqref{eq:app_solution} for $A_{\alpha, app}$, we find that when $t>cr$ with $c>1$, the estimate ~\eqref{eq:estimate_of_Aalpha_app} follows from the fact that $\jb{t-r}\sim_c\jb{t+r}$ in that region. If $t\leq cr$ and $t-r\geq 1/2$, the second and the third terms are bounded by $D^2_{\abs{I}+1}\jb{t+r}\jb{t-r}^{-\gamma}\ln\frac{\jb{t+r}}{\jb{t-r}}.$ The last two terms are controlled by $D_{\abs{I}}\jb{t+r}^{-1}\jb{t-r}^{-\gamma}$. Regarding the first term, we rewrite it as
	\begin{equation}\label{eq:decompose_into_two}
	\begin{split}
&\frac{1}{4\pi}\int_{-\infty}^{\infty}\!\int_{\BS^2}\!\frac{\mathcal{J}_\alpha(\eta, \sigma)}{t-r\angles{\omega}{\sigma}}\, dS(\sigma)\,d\eta\\
&\qquad=\frac{1}{4\pi r}\ln\frac{t+r}{t-r}\int_{-\infty}^{\infty}\!\mathcal{J}_{\alpha}(\eta, \omega)\,d\eta+\frac{1}{4\pi}\int_{-\infty}^{\infty}\!\int_{\BS^2}\!\frac{\mathcal{J}_{\alpha}(\eta, \sigma)-\mathcal{J}_{\alpha}(\eta, \omega)}{t-r\angles{\omega}{\sigma}}\, dS(\sigma)\,d\eta,
\end{split}
	\end{equation}
	by the estimate ~\eqref{eq:estimate_of_phik} from Lemma ~\ref{lem:taylor_expansion_in_n}, the conclusion follows. It remains to consider the region $t-r<1/2$. In this case, the first term is $0$ and the estimate ~\eqref{eq:estimate_of_Aalpha_app} follows directly from the formula of $A_{\alpha, app}$.
	
	Next, we turn to the estimates for the null decomposition of $A_{\alpha, app}$. We compute that 
	\begin{align*}
	A_{L, app}&=\frac{1}{4\pi}\int_{-\infty}^{\infty}\!\int_{\BS^2}\!\frac{(1-\angles{\sigma}{\omega})\mathcal{J}_0(\eta, \sigma)}{t-r\angles{\omega}{\sigma}}\, dS(\sigma)\,d\eta\chi_{ex}(\frac{t}{12})\chi_{in}(q)\\
	&\qquad+\frac{\BFq}{4\pi r}\chi_{ex}(q)+\frac{L^\alpha(\omega)\mathcal{A}_{\alpha}(q, \omega)}{r}\chi\Bigl(\frac{\jb{q}}{r}\Bigr)\\
	&=\frac{\BFq}{4\pi r}\chi_{ex}(\frac{t}{12})\chi_{in}(q)+\frac{\BFq}{4\pi r}\delta_{\alpha 0}\chi_{ex}(q)+\frac{L^\alpha(\omega)\mathcal{A}_{\alpha}(q, \omega)}{r}\chi\Bigl(\frac{\jb{q}}{r}\Bigr)\\
	&\qquad+\frac{1}{4\pi}\frac{r-t}{r^2}\ln \frac{t+r}{t-r}\int_{-\infty}^{\infty}\!\mathcal{J}_0(\eta, \omega)\,d\eta\chi_{ex}(\frac{t}{12})\chi_{in}(q)\\
	&\qquad+\frac{1}{4\pi}\frac{r-t}{r}\int_{-\infty}^{\infty}\!\int_{\BS^2}\!\frac{\mathcal{J}_0(\eta, \sigma)-\mathcal{J}_0(\eta, \omega)}{t-r\angles{\omega}{\sigma}}\, dS(\sigma)\,d\eta\chi_{ex}(\frac{t}{12})\chi_{in}(q)\\
	&=\frac{\BFq}{4\pi r}\chi_{in}(q)+\frac{\BFq}{4\pi r}\delta_{\alpha 0}\chi_{ex}(q)+\frac{L^\alpha(\omega)\mathcal{A}_{\alpha}(q, \omega)}{r}\chi\Bigl(\frac{\jb{q}}{r}\Bigr)+\mathcal{R}
	\end{align*}
	Using the estimates ~\eqref{eq:estimate_of_phik} we find that when $t<4r-1$,
	\[
	\abs{Z^I\mathcal{R}_L}\lesssim D^2_{\abs{I}+2}\jb{t-r}\jb{t+r}^{-2}\ln\frac{\jb{t+r}}{\jb{t-r}}\chi_{in}(q).
	\]
	Owing to the \emph{asymptotic Lorenz gauge condition} and the expression ~\eqref{eq:expression_for_A_L} we arrive at
	\[
	\frac{\BFq}{4\pi r}\chi_{in}(q)+\frac{\BFq}{4\pi r}\delta_{\alpha 0}\chi_{ex}(q)+\frac{L^\alpha(\omega)\mathcal{A}_{\alpha}(q, \omega)}{r}\chi\Bigl(\frac{\jb{q}}{r}\Bigr)=\frac{\BFq}{4\pi r}
+\mathcal{O}(\frac{\BFq}{\jb{t+r}^m})
	\]
	for $r\geq 1/4$ and any $m>0$. Putting all of the above estimates yields ~\eqref{eq:decomposition_of_A_L} and ~\eqref{eq:estimate_for_remainder_of_A_L}.
	
	Finally, the estimates for $Z^IA_{\underline{L}, app}$ is straightforward. It remains to control $Z^IA_{e_B, app}$. To this end we calculate
	\begin{align*}
	A_{e_B, app}&=\frac{1}{4\pi}\int_{-\infty}^{\infty}\!\int_{\BS^2}\!\frac{\omega_B^\alpha\mathcal{J}_\alpha(\eta, \sigma)}{t-r\angles{\omega}{\sigma}}\, dS(\sigma)\,d\eta\chi_{ex}(\frac{t}{12})\chi_{in}(q)+\frac{\omega_B^\alpha\mathcal{A}_{\alpha}(q, \omega)}{r}\chi\Bigl(\frac{\jb{q}}{r}\Bigr)\\
	&=\frac{1}{4\pi}\int_{-\infty}^{\infty}\!\int_{\BS^2}\!\frac{\omega_B^\alpha\mathcal{J}_\alpha(\eta, \sigma)-\omega_B^\alpha\mathcal{J}_\alpha(\eta, \omega)}{t-r\angles{\omega}{\sigma}}\, dS(\sigma)\,d\eta\chi_{ex}(\frac{t}{12})\chi_{in}(q)+\frac{\omega_B^\alpha\mathcal{A}_{\alpha}(q, \omega)}{r}\chi\Bigl(\frac{\jb{q}}{r}\Bigr)
	\end{align*}
	where we use the fact that $\omega_B^jw_j=0$. Using the estimate ~\eqref{eq:estimate_of_phik} we arrive at the conclusion ~\eqref{eq:estimate_for_A_B}.
	\end{proof}
Now we are at a position to estimate $\mathcal{V}_{\alpha 1}$ and $\mathcal{U}_1$ which are the error terms determined completely by the approximate solutions.
\begin{lem}\label{lem:estiamte_of_u1_v1}
We have
\begin{align}\label{eq:estimate_of_valpha1}
\abs{\mathcal{V}_{\alpha 1}}&\lesssim\frac{D_{\abs{I}+3}}{\jb{t+r}^3\jb{t-r}^\gamma}\chi\bigl(\frac{\jb{q}}{r}\bigr)+\frac{D^2_{\abs{I}+1}}{\jb{t+r}^m}\\
&\qquad+\frac{D_{\abs{I}+3}^2}{\jb{t+r}^3\jb{t-r}^{2\gamma}}\ln\frac{\jb{t+r}}{\jb{t-r}}\chi\bigl(\frac{\jb{q}}{r}\bigr)\notag\\
\label{eq:estimate_of_u1}
\abs{\mathcal{U}_1}&\lesssim\frac{D_{\abs{I}+2}}{\jb{t+r}^3\jb{t-r}^\gamma}\chi\bigl(\frac{\jb{q}}{r}\bigr)+\frac{D^3_{\abs{I}+2}}{\jb{t+r}^3\jb{t-r}^\gamma}\ln\frac{\jb{t+r}}{\jb{t-r}}\chi\bigl(\frac{\jb{q}}{r}\bigr)
\end{align}
for any $m>0$.
\end{lem}
\begin{proof}
	For $\mathcal{V}_{\alpha 1}$, in view of Lemma ~\ref{lem:box_of_app} we are left with dealing with $-\Im(\phi_{app}\overline{\partial_\alpha\phi_{app}})+\abs{\phi_{app}}^2A_{\alpha, app}$. Direct calculation implies
	\begin{align*}
	&-\Im(\phi_{app}\overline{\partial_\alpha\phi_{app}})+\abs{\phi_{app}}^2A_{\alpha, app}\\
	&\qquad=-L_\alpha\Im(\phi_{app}\overline{\partial_q\phi_{app}})+\frac{\underline{L}_\alpha}{2}\Im(\phi_{app}\overline{L\phi_{app}})-\omega_\alpha^B\Im(\phi_{app}\overline{e_B(\phi_{app})})+\abs{\phi_{app}}^2A_{\alpha, app}.
	\end{align*}
	Since
	\begin{align*}
	-L_\alpha\Im(\phi_{app}\overline{\partial_q\phi_{app}})&=-\frac{\mathcal{J}_\alpha}{r^2}\chi^2\bigl(\frac{\jb{q}}{r}\bigr)-L_\alpha\frac{\BFq}{8\pi}\frac{\abs{\Phi}^2}{r^3}\chi^2\bigl(\frac{\jb{q}}{r}\bigr),\\
	\frac{\underline{L}_\alpha}{2}\Im(\phi_{app}\overline{L\phi_{app}})&=\underline{L}_\alpha\frac{\BFq}{8\pi}\frac{\abs{\Phi}^2}{r^3}\chi^2\bigl(\frac{\jb{q}}{r}\bigr),\\
	-\omega_\alpha^B\Im(\phi_{app}\overline{e_B(\phi_{app})})&=-\omega_\alpha^B\frac{\Im(\Phi\overline{e_B(\Phi)})}{r^2}\chi^2\bigl(\frac{\jb{q}}{r}\bigr),\\
	\abs{\phi_{app}}^2A_{\alpha, app}&=\frac{\abs{\Phi}^2}{r^2}\chi^2\bigl(\frac{\jb{q}}{r}\bigr)A_{\alpha, app},
	\end{align*}
	by ~\eqref{eq:vectorfield_on_radiation} and the estimate ~\eqref{eq:estimate_of_Aalpha_app} from Lemma ~\ref{lem:estimate_of_Aapp}, we find that 
	\begin{align*}
	&\biggl\lvert Z^I\Bigl(-\Im(\phi_{app}\overline{\partial_\alpha\phi_{app}})+\abs{\phi_{app}}^2A_{\alpha, app}+\frac{\mathcal{J}_\alpha}{r^2}\chi\bigl(\frac{\jb{q}}{r}\bigr)\Bigr)\biggr\rvert\\
	&\qquad\lesssim\frac{1}{\jb{t+r}^3\jb{t-r}^{2\gamma}}\Bigl(D^2_{\abs{I}+1}+\bigl(D^4_{\abs{I}+2}\chi_{in}(q)+\frac{D^4_{\abs{I}+1}}{\jb{t-r}^{2\gamma}}\bigr)\ln\frac{\jb{t+r}}{\jb{t-r}}\Bigr)\chi\bigl(\frac{\jb{q}}{r}\bigr).
	\end{align*}
Combining the above estimate with ~\eqref{eq:estimate_of_Aapp} and the fact $Z^I\widetilde{\chi}(t/T)\sim 1$ in the support of $\chi_{in}$ or $\chi$, we prove  ~\eqref{eq:estimate_of_valpha1}.

Moving on, we consider $\mathcal{U}_1$ and express
\begin{align*}
&-2iA^\alpha_{app}\partial_\alpha\phi_{app}+A^\alpha_{app}A_{\alpha, app}\phi_{app}\\
&\qquad=-2iA_{L, app}\partial_q\phi_{app}+iA_{\underline{L}, app}L(\phi_{app})-2iA^{e_B}_{app}e_B(\phi_{app})\\
&\qquad\qquad+(-A_{L, app}A_{\underline{L}, app}+A^{e_B}_{app}A_{e_B, app})\phi_{app}.
\end{align*}
Then we see that 
\begin{align*}
A_{L, app}\partial_q\phi_{app}&=(\frac{\BFq}{4\pi r}+\mathcal{R}_L)\partial_q\phi_{app}\\
&=e^{-i\frac{\BFq}{4\pi}\ln r}\frac{\BFq}{4\pi}\frac{\partial_q\Phi}{r^2}\chi\bigl(\frac{\jb{q}}{r}\bigr)+e^{-i\frac{\BFq}{4\pi}\ln r}\mathcal{R}_L\frac{\partial_q\Phi}{r}\chi\bigl(\frac{\jb{q}}{r}\bigr)\\
&\qquad+e^{-i\frac{\BFq}{4\pi}\ln r}(\frac{\BFq}{4\pi r}+\mathcal{R}_L)\biggl(-i\frac{\BFq}{8\pi}\frac{\Phi}{r^2}\chi\bigl(\frac{\jb{q}}{r}\bigr)-\frac{\Phi}{2r^2}\chi\bigl(\frac{\jb{q}}{r}\bigr)+\frac{\Phi}{r}\partial_q\Bigl(\chi\bigl(\frac{\jb{q}}{r}\bigr)\Bigr)\biggr),\\
A_{\underline{L}, app}L(\phi_{app})&=e^{-i\frac{\BFq}{4\pi}\ln r}A_{\underline{L}, app}\biggl(-i\frac{\BFq}{4\pi}\frac{\Phi}{r^2}\chi\bigl(\frac{\jb{q}}{r}\bigr)-\frac{\Phi}{r^2}\chi\bigl(\frac{\jb{q}}{r}\bigr)+\frac{\Phi}{r}L\Bigl(\chi\bigl(\frac{\jb{q}}{r}\bigr)\Bigr)\biggr)\\
A^{e_B}_{app}e_B(\phi_{app})&=e^{-i\frac{\BFq}{4\pi}\ln r}A_{e_B, app}\frac{e_B(\Phi)}{r^2}\chi\bigl(\frac{\jb{q}}{r}\bigr).
\end{align*}
By the estimates for the null decomposition of $A_{\alpha, app}$ from Lemma ~\ref{lem:estimate_of_Aapp} we obtain that 
\begin{align*}
&\biggl\lvert Z^I\Bigl(-2iA^\alpha_{app}\partial_\alpha\phi_{app}+A^\alpha_{app}A_{\alpha, app}\phi_{app}+ie^{-i\frac{\BFq}{4\pi}\ln r}\frac{\BFq}{2\pi}\frac{\partial_q\Phi}{r^2}\chi\bigl(\frac{\jb{q}}{r}\bigr)\Bigr)\biggr\rvert\\
&\qquad\lesssim\frac{D^2_{\abs{I}+2}}{\jb{t+r}^3\jb{t-r}^\gamma}\chi\bigl(\frac{\jb{q}}{r}\bigr)+\frac{D^3_{\abs{I}+2}}{\jb{t+r}^3\jb{t-r}^\gamma}\ln\frac{\jb{t+r}}{\jb{t-r}}\chi\bigl(\frac{\jb{q}}{r}\bigr)
\end{align*}
Putting the above estimate and ~\eqref{eq:estimate_of_phiapp} together yields ~\eqref{eq:estimate_of_u1}.
\end{proof}
As we apply Corollary ~\ref{cor:energyestimate_for_u} to $Z^Iu$ and $Z^Iv_\alpha$, this means that we need to control
\[
\int_t^{2T}\!\int_{\Sigma_\tau}\! \biggl\lvert\Re\bigl(\frac{1}{r}\overline{K}_0(rZ^Iu)\overline{\mathcal{U}_1}\bigr)w\biggr\rvert\,dx\,d\tau
\quad \text{and} \quad\int_t^{2T}\!\int_{\Sigma_\tau}\! \biggl\lvert\frac{1}{r}\overline{K}_0(rZ^Iv_\alpha)\mathcal{V}_{\alpha 1}w\biggr\rvert\,dx\,d\tau
\]
with $w$ defined in Proposition ~\ref{prop:estimate_for_rmkg}.
\begin{lem}\label{lem:error_estimate_of_u1v1}
We have 
	\begin{align}\label{eq:error_estimate_of_u1}
	\int_t^{2T}\!\int_{\Sigma_\tau}\! \biggl\lvert\Re\bigl(\frac{1}{r}\overline{K}_0(rZ^Iu)\overline{\mathcal{U}_1}\bigr)w\biggr\rvert\,dx\,d\tau&\lesssim\frac{D_{\abs{I}+2}}{\jb{t}^{\frac{\gamma}{2}-\frac{\mu}{2}-\frac{1}{4}}}S_{\abs{I}}(t)^{1/2}+\frac{D_{\abs{I}+2}}{\jb{t}^{\gamma-\frac{\mu}{2}-\frac{1}{2}}}\mathcal{E}_{\abs{I}}(t)^{1/2},\\
	\label{eq:error_estimate_of_v1}
		\int_t^{2T}\!\int_{\Sigma_\tau}\! \biggl\lvert\frac{1}{r}\overline{K}_0(rZ^Iv_\alpha)\mathcal{V}_{\alpha 1}w\biggr\rvert\,dx\,d\tau&\lesssim \frac{D_{\abs{I}+3}}{\jb{t}^{\frac{\gamma}{2}-\frac{\mu}{2}-\frac{1}{4}}}S_{\abs{I}}(t)^{1/2}+\frac{D_{\abs{I}+3}}{\jb{t}^{\gamma-\frac{\mu}{2}-\frac{1}{2}}}\mathcal{E}_{\abs{I}}(t)^{1/2}.
	\end{align}
\end{lem}

\begin{proof}
	Recall that $\overline{K}_0=(\jb{t+r}^2+\jb{t-r}^2)/2$, we have
	\begin{align*}
	&\int_t^{2T}\!\int_{\Sigma_\tau}\! \biggl\lvert\Re\bigl(\frac{1}{r}\overline{K}_0(rZ^Iv_\alpha)\overline{\mathcal{U}_1}\bigr)w\biggr\rvert\,dx\,d\tau\\
	&\qquad\lesssim\int_t^{2T}\!\int_{\Sigma_\tau}\! \biggl\lvert\frac{\jb{t-r}^2}{r}\underline{L}(rZ^Iu)\mathcal{U}_1w\biggr\rvert\,dx\,d\tau+\int_t^{2T}\!\int_{\Sigma_\tau}\! \biggl\lvert\frac{\jb{t+r}^2}{r}L(rZ^Iu)\mathcal{U}_1w\biggr\rvert\,dx\,d\tau\\
	&\qquad\equiv A+B.
	\end{align*}
	By Lemma ~\ref{lem:estiamte_of_u1_v1} we can bound $A$ by Cauchy-Schwarz in spatial variables
	\[
A\lesssim\mathcal{E}_{\abs{I}}(t)^{1/2}\int_t^{2T}\!\Bigl(\int_{\Sigma_\tau}\!\jb{t-r}^2\abs{\mathcal{U}_1}^2w\,dx\Bigr)^{1/2}\lesssim\frac{D_{\abs{I}+3}}{\jb{t}^{\gamma-\frac{\mu}{2}-\frac{1}{2}}}\mathcal{E}_{\abs{I}}(t)^{1/2}.
	\]
	For the term $B$, we apply Cauchy-Schwarz to the whole spacetime integral to obtain that
	\begin{align*}
B&\lesssim\biggl(\int_t^{2T}\!\int_{\Sigma_\tau}\!\frac{\jb{t+r}^2\abs{\frac{1}{r}L(rZ^Iv_\alpha)}^2}{\jb{q}^{\gamma+1/2}}w\,dx\,d\tau\biggr)^{1/2}\\
	&\qquad\times\biggl(\int_t^{2T}\!\int_{\Sigma_\tau}\!\jb{t-r}^{\gamma+\frac{1}{2}}\jb{t+r}^2\abs{Z^J\mathcal{V}_{\alpha 1}}^2w\,dx\,d\tau\biggr)^{1/2}\\
	&\lesssim\frac{D_{\abs{I}+3}}{\jb{t}^{\frac{\gamma}{2}-\frac{\mu}{2}-\frac{1}{4}}}S_{\abs{I}}(t)^{1/2}.
	\end{align*}
Here we use $\ln x\leq 4x^{1/4}$ for $x\geq 1$ to deal with the logarithm terms in the above two estimates. This proves the first estimate ~\eqref{eq:error_estimate_of_u1} in this Lemma. The second estimate ~\eqref{eq:error_estimate_of_v1} can be shown in an analogous manner.
\end{proof}

\subsection{Energy estimates for the remainder of the gauge potential}

In this subsection we proceed to bound the other error terms $\mathcal{V}_{\alpha 2}$, $\mathcal{V}_{\alpha 3}$, $\mathcal{V}_{\alpha 4}$, $\mathcal{V}_{\alpha 5}$ in the equation ~\eqref{eq:expansion_of_eqn_for_v}.

\subsubsection*{Error estimates for $\mathcal{V}_{\alpha 2}$}

We have 
\begin{equation}\label{eq:expression_of_v2}
\mathcal{V}_{\alpha 2}=\widetilde{\chi}(\frac{t}{T})\Bigl(-\Im(\phi_{app}\overline{\partial_\alpha Z^I u})-\Im(u\overline{\partial_\alpha Z^Iu})\Bigr)+\mathcal{R}_v
\end{equation}
where
\begin{align*}
\abs{\mathcal{R}_v}&\lesssim\sum_{\abs{J}+\abs{K}+\abs{L}\leq\abs{I}}\abs{Z^L\widetilde{\chi}(\frac{t}{T})}\abs{\Im(Z^K\partial_\alpha\phi_{app}\overline{Z^Ju})}+\sum_{\abs{J}\leq\abs{I}-1}\sum_{\abs{J}+\abs{K}+\abs{L}\leq\abs{I}}\abs{Z^L\widetilde{\chi}(\frac{t}{T})}\abs{\Im(Z^K\phi_{app}\overline{\partial Z^Ju})}\\
&\qquad+\sum_{\abs{J}\leq\abs{I}-1}\sum_{\abs{J}+\abs{K}+\abs{L}\leq\abs{I}}\abs{Z^L\widetilde{\chi}(\frac{t}{T})}\abs{\Im(Z^Ku\overline{\partial Z^Ju})}\\
&\lesssim\sum_{\abs{J}+\abs{K}\leq \abs{I}}\frac{D_{\abs{K}+1}}{\jb{t+r}\jb{t-r}^{\gamma+1}}\abs{Z^Ju}+\sum_{\abs{J}\leq \abs{I}-1}\sum_{\abs{J}+\abs{K}\leq \abs{I}}\frac{D_{\abs{K}}}{\jb{t+r}\jb{t-r}^{\gamma}}\abs{\partial Z^Ju}\\
&\qquad+\sum_{\abs{J}\leq \abs{I}-1}\sum_{\abs{J}+\abs{K}\leq\abs{I}}\abs{Z^Ku}\abs{\partial Z^J u}\\
&\lesssim\sum_{\abs{J}+\abs{K}\leq \abs{I}}\frac{D_{\abs{K}+1}}{\jb{t+r}\jb{t-r}^{\gamma+1}}\abs{Z^Ju}+\sum_{1\leq\abs{J}\leq\abs{I}}\sum_{\abs{J}+\abs{K}\leq\abs{I}+1}\frac{1}{\jb{t-r}}\abs{Z^Ku}\abs{ Z^J u}
\end{align*}
Here we use Lemma ~\ref{lem:pointwise_estimate_of_der} in the last inequality.

We begin with the $(\jb{t-r}^2/r)\underline{L}$ component in the multiplier $\overline{K}_0$. Since $\abs{I}\leq k=N-3$ with $N\geq 7$, we notice that $\min{\{\abs{J}, \abs{K}\}}\leq \lfloor{(\abs{I}+1)/2}\rfloor\leq\lfloor{(k+1)/2}\rfloor\leq k-1$. We assume $\abs{K}\leq\abs{J}$, then we can employ Lemma ~\ref{lem:KS_with_weight} and Proposition ~\ref{prop:control_of_norm} to obtain that 
\[
\abs{Z^Ku}\lesssim\frac{\sum_{\abs{K'}\leq \abs{K}+2}\norm{Z^{K'} u}_{L^2(w)}}{\jb{t+r}\jb{t-r}^{1/2}w^{1/2}}\lesssim\frac{\mathcal{E}_{\abs{K}+1}(t)^{1/2}}{\jb{t+r}\jb{t-r}^{1/2}w^{1/2}}\lesssim\frac{\mathcal{E}_k(t)^{1/2}}{\jb{t+r}\jb{t-r}^{1/2}w^{1/2}}.
\]
Then we can easily bound 
\begin{equation}\label{eq:treat_of_Lbar}
\begin{split}
\int_t^{2T}\!\int_{\Sigma_\tau}\! \biggl\lvert\frac{\jb{t-r}^2}{r}\underline{L}(rZ^Iv_\alpha)\mathcal{V}_{\alpha 2}w\biggr\rvert\,dx\,d\tau&\lesssim \int_t^{2T}\!\frac{D_{\abs{I}+1}}{\jb{\tau}}\Bigl(\mathcal{E}_{\abs{I}}(\tau)+\sum_{\abs{J}\leq\abs{I}+1}\norm{Z^Ju}^2_{L^2(w)}\Bigr)\,d\tau\\
&\qquad+\int_t^{2T}\!\frac{\mathcal{E}_k(t)^{1/2}}{\jb{\tau}}\Bigl(\mathcal{E}_{\abs{I}}(\tau)+\sum_{\abs{J}\leq\abs{I}+1}\norm{Z^Ju}^2_{L^2(w)}\Bigr)\,d\tau\\
&\lesssim\int_t^{2T}\!\frac{D_{\abs{I}+1}+\mathcal{E}_k(t)^{1/2}}{\jb{\tau}}\mathcal{E}_{\abs{I}}(\tau)\,d\tau
\end{split}
\end{equation}
where we use Proposition ~\ref{prop:control_of_norm} in the last step. 

Next let us treat the $(\jb{t+r}^2/r)L$ component of $\overline{K}_0$.  To this end we first exploit the spacetime version of Hardy inequality established in Lemma ~\ref{lem:poincareineq} to handle $\mathcal{R}_v$. Indeed, we apply Cauchy-Schwarz to the whole spacetime integral and Lemma ~\ref{lem:poincareineq} to find that 
\begin{equation}\label{eq:treat_of_L}
\begin{split}
&\int_t^{2T}\!\int_{\Sigma_\tau}\! \biggl\lvert\frac{\jb{t+r}^2}{r}L(rZ^Iv_\alpha)\mathcal{R}_vw\biggr\rvert\,dx\,d\tau\\
&\qquad\lesssim D_{\abs{I}+1}S_{\abs{I}}(t)^{1/2}\sum_{\abs{J}\leq\abs{I}}\Bigl(\int_t^{2T}\!\int_{\Sigma_\tau}\! \frac{\abs{Z^J u}^2w}{\jb{t-r}^{\gamma+3/2}}\,dx\,d\tau\Bigr)^{1/2}\\
&\qquad\qquad+\mathcal{E}_k(t)^{1/2}S_{\abs{I}}(t)^{1/2}\sum_{\abs{J}\leq\abs{I}}\Bigl(\int_t^{2T}\!\int_{\Sigma_\tau}\! \frac{\abs{Z^J u}^2}{\jb{t-r}^{5/2-\gamma}}\,dx\,d\tau\Bigr)^{1/2}\\
&\qquad\lesssim D_{\abs{I}+1}S_{\abs{I}}(t)^{1/2}\sum_{\abs{J}\leq\abs{I}}\Bigl(\int_t^{2T}\!\int_{\Sigma_\tau}\! \frac{\jb{t+r}^2\abs{\frac{1}{r}L(rZ^J u)}^2w}{\jb{t-r}^{\gamma+3/2}}\,dx\,d\tau\Bigr)^{1/2}\\
&\qquad\qquad+\mathcal{E}_k(t)^{1/2}S_{\abs{I}}(t)^{1/2}\sum_{\abs{J}\leq\abs{I}}\Bigl(\int_t^{2T}\!\int_{\Sigma_\tau}\! \frac{\jb{t+r}^2\abs{\frac{1}{r}L(rZ^J u)}^2}{\jb{t-r}^{5/2-\gamma}}\,dx\,d\tau\Bigr)^{1/2}\\
&\qquad\lesssim (D_{\abs{I}+1}+\mathcal{E}_k(t)^{1/2})S_{\abs{I}}(t).
\end{split}
\end{equation}
It remains to control the first term on the right-hand side of ~\eqref{eq:expression_of_v2}, which has to be done more carefully. First we rewrite 
\[
\partial_\alpha Z^I u=L_\alpha\partial_qZ^Iu-\frac{\underline{L}_\alpha}{2}L(Z^I u)+\omega_\alpha^Be_B(Z^I u)\equiv L_\alpha\partial_qZ^Iu+\overline{\partial}_\alpha Z^I u.
\]
Then for the tangential derivatives, we are able to use Lemma ~\ref{lem:pointwise_estimate_of_der} and proceed as in ~\eqref{eq:treat_of_Lbar} to conclude that 
\begin{equation}\label{eq:treat_of_tangential_component}
\begin{split}
&\int_t^{2T}\!\int_{\Sigma_\tau}\! \biggl\lvert\frac{\jb{t+r}^2}{r}L(rZ^Iv_\alpha)\widetilde{\chi}(\frac{t}{T})\Bigl(-\Im(\phi_{app}\overline{\overline{\partial}_\alpha Z^I u})-\Im(u\overline{\overline{\partial}_\alpha Z^Iu})\Bigr)w\biggr\rvert\,dx\,d\tau\\
&\qquad\lesssim\int_t^{2T}\!\frac{D_0+\mathcal{E}_1(t)^{1/2}}{\jb{\tau}}\mathcal{E}_{\abs{I}}(\tau)\,d\tau.
\end{split}
\end{equation}
Finally we are left with estimating $L_\alpha\partial_qZ^Iu$. To this end we integrate by parts in $q=r-t$ direction to transfer $\partial_q$ to the term $L(rZ^Iv_\alpha)$ and then use the identity $\Box \varphi=2r^{-1}\partial_qL(r\varphi)+r^{-2}\Delta_\omega \varphi$. Then we insert the equations ~\eqref{eq:expansion_of_eqn_for_v} and apply integration by parts in $q$ direction again to get the desired bound. Specifically, we compute that 
\begin{align*}
&\int_t^{2T}\!\int_{\Sigma_\tau}\! \frac{\jb{t+r}^2}{r}L(rZ^Iv_\alpha)L_\alpha(\omega)\widetilde{\chi}(\frac{t}{T})\Bigl(-\Im(\phi_{app}\overline{\partial_qZ^I u})-\Im(u\overline{\partial_qZ^Iu})\Bigr)w\,dx\,d\tau\\
&\qquad=\int_{\Sigma_t}\! \frac{\jb{t+r}^2}{2r}L(rZ^Iv_\alpha)L_\alpha(\omega)\widetilde{\chi}(\frac{t}{T})\Bigl(-\Im(\phi_{app}\overline{Z^I u})-\Im(u\overline{Z^Iu})\Bigr)w\,dx\\
&\qquad\qquad+\int_t^{2T}\!\int_{\Sigma_\tau}\! \frac{\jb{t+r}^2}{r^2}L(rZ^Iv_\alpha)\Im\Bigl(\partial_q\bigl(L_\alpha(\omega)\widetilde{\chi}(\frac{t}{T})r\phi_{app}w\bigr)\overline{Z^I u}\Bigr)\,dx\\
&\qquad\qquad+\int_t^{2T}\!\int_{\Sigma_\tau}\! \frac{\jb{t+r}^2}{r^2}L(rZ^Iv_\alpha)\Im\Bigl(\partial_q\bigl(L_\alpha(\omega)\widetilde{\chi}(\frac{t}{T})ruw\bigr)\overline{Z^I u}\Bigr)\,dx\\
&\qquad\qquad+\int_t^{2T}\!\int_{\Sigma_\tau}\! -\frac{\jb{t+r}^2}{2r^2}(\Delta_\omega Z^Iv_\alpha) L_\alpha(\omega)\widetilde{\chi}(\frac{t}{T})\Bigl(\Im(\phi_{app}\overline{Z^I u})+\Im(u\overline{Z^Iu})\Bigr)w\,dx\,d\tau\\
&\qquad\qquad+\frac{1}{2}\int_t^{2T}\!\int_{\Sigma_\tau}\! \jb{t+r}^2(\Box Z^Iv_\alpha) L_\alpha(\omega)\widetilde{\chi}(\frac{t}{T})\Bigl(\Im(\phi_{app}\overline{Z^I u})+\Im(u\overline{Z^Iu})\Bigr)w\,dx\,d\tau\\
&\qquad\equiv \mathcal{V}_a+\mathcal{V}_b+\mathcal{V}_c+\mathcal{V}_d+\mathcal{V}_e.
\end{align*}
For the term $\mathcal{V}_a$, by the pointwise bounds $\abs{\phi_{app}}\lesssim D_0\jb{t+r}^{-1}$, $\abs{u}\lesssim \mathcal{E}_1(t)\jb{t+r}^{-1}$ and Proposition ~\ref{prop:control_of_norm} we see that
\begin{equation}\label{eq:control_of_Va}
\abs{\mathcal{V}_a}\lesssim(D_0+\mathcal{E}_1(t)^{1/2})\int_{\Sigma_t}\jb{t+r}^2\frac{\abs{L(rv_\alpha)}^2}{r^2}+\abs{Z^I u}^2\, dx\lesssim D_0\mathcal{E}_{\abs{I}}(t)+\mathcal{E}_1(t)^{1/2}\mathcal{E}_{\abs{I}}(t).
\end{equation}
For the term $\mathcal{V}_b$, as we have the pointwise estimate $\abs{\partial_q(L_\alpha(\omega)\widetilde{\chi}r\phi_{app}w)}\lesssim D_1\jb{t-r}^{-1-\gamma}w$, we can treat it in the same way as in ~\eqref{eq:treat_of_L} to find that 
\begin{equation}\label{eq:control_of_Vb}
\abs{\mathcal{V}_b}\lesssim D_1S_{\abs{I}}(t).
\end{equation}
For the  term $\mathcal{V}_c$, since 
\[
\abs{\partial_q\bigl(L_\alpha(\omega)\widetilde{\chi}(\frac{t}{T})ruw\bigr)}\lesssim \abs{u}w+r\jb{t-r}^{-1}\sum_{\abs{J}\leq 1}\abs{Z^Iu}w\lesssim\frac{\mathcal{E}_1(t)^{1/2}w}{\jb{t+r}}+\frac{r\mathcal{E}_2(t)^{1/2}w}{\jb{t+r}\jb{t-r}^{3/2}w^{1/2}},
\]
we arrive at the estimate
\begin{equation}\label{eq:control_of_Vc}
\abs{\mathcal{V}_c}\lesssim \int_t^{2T}\!\frac{\mathcal{E}_{1}(t)^{1/2}}{\jb{\tau}}\mathcal{E}_{\abs{I}}(\tau)\,d\tau+\mathcal{E}_2(t)^{1/2}S_{\abs{I}}(t).
\end{equation}
For the term $\mathcal{V}_d$, in view of the fact that $\Delta_\omega=\sum_{i, j=1}^3\Omega_{ij}^2/2$, we integrate by parts with respect to the spatial variables and thus can transfer one rotation vector field from $\Delta_\omega Z^Iv_\alpha$ to the other terms. Specifically, we calculate
\begin{equation}\label{eq:control_of_Vd}
\begin{split}
\mathcal{V}_d&=\frac{1}{4}\sum_{i,j=1}^3\int_t^{2T}\!\int_{\Sigma_\tau}\! \frac{\jb{t+r}^2}{r^2} \Omega_{ij}Z^Iv_\alpha L_\alpha(\omega)\widetilde{\chi}(\frac{t}{T})\Bigl(\Im(\phi_{app}\overline{\Omega_{ij}Z^I u})+\Im(u\overline{\Omega_{ij}Z^Iu})\Bigr)w\,dx\,d\tau\\
&\qquad+\frac{1}{4}\sum_{i,j=1}^3\int_t^{2T}\!\int_{\Sigma_\tau}\! \frac{\jb{t+r}^2}{r^2} \Omega_{ij}Z^Iv_\alpha \widetilde{\chi}(\frac{t}{T})\Im\bigl(\Omega_{ij}(L_\alpha(\omega)\phi_{app}+L_\alpha(\omega)u)\overline{Z^I u}\bigr)w\,dx\,d\tau\\
&\lesssim\int_t^{2T}\!\frac{D_1+\mathcal{E}_2(t)^{1/2}}{\jb{\tau}}\mathcal{E}_{\abs{I}}(\tau)\,d\tau
\end{split}
\end{equation}
where we use the identities ~\eqref{eq:vectorfield_identity} and Proposition ~\ref{prop:control_of_norm} in the last step.

Finally we estimate the term $\mathcal{V}_e$. To this end we insert the equations ~\eqref{eq:expansion_of_eqn_for_v} into $\mathcal{V}_e$ to obtain that 
\[
\mathcal{V}_e=\frac{1}{2}\int_t^{2T}\!\int_{\Sigma_\tau}\! \jb{t+r}^2(\mathcal{V}_{\alpha 1}+\mathcal{V}_{\alpha 2}+\mathcal{V}_{\alpha 3}+\mathcal{V}_{\alpha 4}+\mathcal{V}_{\alpha 5}) L_\alpha(\omega)\widetilde{\chi}(\frac{t}{T})\Bigl(\Im(\phi_{app}\overline{Z^I u})+\Im(u\overline{Z^Iu})\Bigr)w\,dx\,d\tau.
\]
We introduce the following short-hand notations
\begin{equation}\label{eq:shorthand_notations}
\begin{split}
\mathcal{V}_e^1&:=\frac{1}{2}\int_t^{2T}\!\int_{\Sigma_\tau}\! \jb{t+r}^2\mathcal{V}_{\alpha 1} L_\alpha(\omega)\widetilde{\chi}(\frac{t}{T})\Bigl(\Im(\phi_{app}\overline{Z^I u})+\Im(u\overline{Z^Iu})\Bigr)w\,dx\,d\tau,\\
\mathcal{V}_e^2&:=\frac{1}{2}\int_t^{2T}\!\int_{\Sigma_\tau}\! \jb{t+r}^2\mathcal{V}_{\alpha 2} L_\alpha(\omega)\widetilde{\chi}(\frac{t}{T})\Bigl(\Im(\phi_{app}\overline{Z^I u})+\Im(u\overline{Z^Iu})\Bigr)w\,dx\,d\tau,\\
\mathcal{V}_e^3&:=\frac{1}{2}\int_t^{2T}\!\int_{\Sigma_\tau}\! \jb{t+r}^2\mathcal{V}_{\alpha 3} L_\alpha(\omega)\widetilde{\chi}(\frac{t}{T})\Bigl(\Im(\phi_{app}\overline{Z^I u})+\Im(u\overline{Z^Iu})\Bigr)w\,dx\,d\tau,\\
\mathcal{V}_e^4&:=\frac{1}{2}\int_t^{2T}\!\int_{\Sigma_\tau}\! \jb{t+r}^2\mathcal{V}_{\alpha 4} L_\alpha(\omega)\widetilde{\chi}(\frac{t}{T})\Bigl(\Im(\phi_{app}\overline{Z^I u})+\Im(u\overline{Z^Iu})\Bigr)w\,dx\,d\tau,\\
\mathcal{V}_e^5&:=\frac{1}{2}\int_t^{2T}\!\int_{\Sigma_\tau}\! \jb{t+r}^2\mathcal{V}_{\alpha 5} L_\alpha(\omega)\widetilde{\chi}(\frac{t}{T})\Bigl(\Im(\phi_{app}\overline{Z^I u})+\Im(u\overline{Z^Iu})\Bigr)w\,dx\,d\tau.
\end{split}
\end{equation}
Following the way we treated the term $B$ in the proof of Lemma ~\ref{lem:error_estimate_of_u1v1} and then exploiting the spacetime version of Hardy inequality from Lemma ~\ref{lem:poincareineq} we conclude that 
\[
\abs{\mathcal{V}_e^1}\lesssim\frac{D_0D_{\abs{I}+3}}{\delta\jb{t}^{1/2-\delta}}S_{\abs{I}}(t)^{1/2}+\frac{\mathcal{E}_1(t)^{1/2}D_{\abs{I}+3}}{\delta\jb{t}^{1/2-\delta}}S_{\abs{I}}(t)^{1/2}
\]
for any $0<\delta<1/2$. Next we turn to the term $\mathcal{V}_e^2$. As we can see from ~\eqref{eq:expression_of_v2}, there are two parts in $\mathcal{V}_{\alpha 2}$. When $\mathcal{V}_e^2 $ involves $\mathcal{R}_v$, we are able to treat it by using the argument in ~\eqref{eq:treat_of_L} combined with the spacetime version of Hardy inequality from Lemma ~\ref{lem:poincareineq}. When $\mathcal{R}_e^2$ contains the tangential components of $\widetilde{\chi}(\frac{t}{T})\Bigl(-\Im(\phi_{app}\overline{\partial_\alpha Z^I u})-\Im(u\overline{\partial_\alpha Z^Iu})\Bigr)$ we control it in the same fashion as in  ~\eqref{eq:treat_of_Lbar}. Indeed, we have
\begin{align*}
&\abs{\mathcal{V}_e^2-\underbrace{\frac{1}{2}\int_t^{2T}\!\int_{\Sigma_\tau}\! \jb{t+r}^2\Bigl(-\Im(\phi_{app}\overline{\partial_q Z^I u})-\Im(u\overline{\partial_qZ^Iu})\Bigr) L^2_\alpha(\omega)\widetilde{\chi}^2(\frac{t}{T})\Bigl(\Im(\phi_{app}\overline{Z^I u})+\Im(u\overline{Z^Iu})\Bigr)w\,dx\,d\tau}_{:=\mathcal{N}_v}}\\
&\qquad\lesssim(D_0D_{\abs{I}+1}+\mathcal{E}_1(t)^{1/2}\mathcal{E}_k(t)^{1/2})S_{\abs{I}}(t
)+\int_t^{2T}\!\frac{(D_0+\mathcal{E}_1(t)^{1/2})^2}{\jb{\tau}}\mathcal{E}_{\abs{I}}(\tau)\,d\tau.
\end{align*}
We denote by $\mathcal{N}_v$ as above the term containing the non tangential component of the first term on the right-hand side of ~\eqref{eq:expression_of_v2}. In order to handle $\mathcal{N}_v$, we rewrite it as 
\[
\mathcal{N}_v=\frac{1}{8}\int_t^{2T}\!\int_{\Sigma_\tau}\! \jb{t+r}^2L^2_\alpha(\omega)\widetilde{\chi}^2(\frac{t}{T})\Bigl(\Re\bigl((\phi_{app}+u)^2\partial_q(\overline{Z^Iu})^2\bigr)-\abs{\phi_{app}+u}^2\partial_q\abs{Z^I u}^2\Bigr)w\,dx\,d\tau.
\]
Thanks to the above expression we can apply integration by parts in $q$ direction to find that 
\begin{align*}
\mathcal{N}_v&=\frac{1}{16}\int_{\Sigma_t}\! \jb{t+r}^2L^2_\alpha(\omega)\widetilde{\chi}^2(\frac{t}{T})\Bigl(\Re\bigl((\phi_{app}+u)^2(\overline{Z^Iu})^2\bigr)-\abs{\phi_{app}+u}^2\abs{Z^I u}^2\Bigr)w\,dx\\
&\qquad+\frac{1}{8}\biggl(\int_t^{2T}\!\int_{\Sigma_\tau}\!\frac{\jb{t+r}^2}{r^2}\partial_q\bigl(wL^2_\alpha(\omega)\widetilde{\chi}^2(\frac{t}{T})\abs{r\phi_{app}+ru}^2\bigr)\abs{Z^I u}^2\,dx\,d\tau\\
&\qquad\qquad-\int_t^{2T}\!\int_{\Sigma_\tau}\!\frac{\jb{t+r}^2}{r^2}\Re\Bigl(\partial_q\bigl(wL^2_\alpha(\omega)\widetilde{\chi}^2(\frac{t}{T})(r\phi_{app}+ru)^2\bigr)(\overline{Z^I u})^2\Bigr)\,dx\,d\tau\biggr)\\
&\equiv\mathcal{N}_v^1+\mathcal{N}_v^2
\end{align*}
Similar to the control of the term $\mathcal{V}_a$ in ~\eqref{eq:control_of_Va}, by Proposition ~\ref{prop:control_of_norm} we can get that
\[
\abs{\mathcal{N}_v^1}\lesssim (D_0+\mathcal{E}_1(t)^{1/2})^2
\mathcal{E}_{\abs{I}}(t).
\]
Analogous to the estimates of the terms $\mathcal{V}_b$ and $\mathcal{V}_c$ in ~\eqref{eq:control_of_Vb} and ~\eqref{eq:control_of_Vc} respectively, we arrive at 
\[
\abs{\mathcal{N}_v^2}\lesssim \int_t^{2T}\!\frac{(D_0+\mathcal{E}_{1}(t)^{1/2})^2}{\jb{\tau}}\mathcal{E}_{\abs{I}}(\tau)\,d\tau+(D_1+\mathcal{E}_2(t)^{1/2})^2S_{\abs{I}}(t).
\]
We first consider the error estimates for $\mathcal{V}_{\alpha 3}$, $\mathcal{V}_{\alpha 4}$ and $\mathcal{V}_{\alpha 5}$ before we turn to the analysis of $\mathcal{V}_e^3$, $\mathcal{V}_e^4$ and $\mathcal{V}_e^5$.

\subsubsection*{Error estimates for $\mathcal{V}_{\alpha 3}$} By the estimate ~\eqref{eq:estimate_of_Aalpha_app} from Lemma ~\ref{lem:estimate_of_Aapp} we have
\begin{align*}
\abs{\mathcal{V}_{\alpha 3}}&\lesssim\sum_{\abs{J}+\abs{K}+\abs{L}\leq\abs{I}}\bigl\lvert Z^L\widetilde{\chi}(\frac{t}{T})\bigr\rvert\Bigl(\bigl\lvert Z^K\abs{\phi_{app}}^2\bigr\rvert\bigl\lvert Z^J v_\alpha\bigr\rvert+\bigl\lvert Z^K(\phi_{app}A_{\alpha, app})\bigr\rvert\bigl\lvert Z^J u\bigr\rvert\Bigr)\\
&\lesssim\sum_{\abs{J}+\abs{K}\leq\abs{I}}\frac{D^2_{\abs{K}}}{\jb{t+r}^2}\abs{Z^Jv_\alpha}+\sum_{\abs{J}+\abs{K}\leq\abs{I}}\frac{D^2_{\abs{K}+1}}{\jb{t+r}^2}\abs{Z^J u}\\
&\qquad+\sum_{\abs{J}+\abs{K}\leq\abs{I}}\frac{D^3_{\abs{K}+2}}{\jb{t+r}^2\jb{t-r}^\gamma\jb{(r-t)_+}^{2\gamma}}\ln\frac{\jb{t+r}}{\jb{t-r}}\abs{Z^J u}\\
&\lesssim\sum_{\abs{J}+\abs{K}\leq\abs{I}}\frac{D^2_{\abs{K}+1}}{\jb{t+r}^2}(\abs{Z^J v_\alpha}+\abs{Z^J u})+\sum_{\abs{J}+\abs{K}\leq\abs{I}}\frac{D^3_{\abs{K}+2}}{\jb{t+r}\jb{t-r}^{1+\gamma}\jb{(r-t)_+}^{2\gamma}}\abs{Z^J u}
\end{align*}
where $a_+=\max\{0, a\}$. In view of the above estimate for $\mathcal{V}_{\alpha 3}$ we can easily find that 
\begin{equation}
\int_t^{2T}\!\int_{\Sigma_\tau}\! \biggl\lvert\frac{1}{r}\overline{K}_0(rZ^Iv_\alpha)\mathcal{V}_{\alpha 3}w\biggr\rvert\,dx\,d\tau\lesssim D^2_{\abs{I}+2}\Bigl(\int_t^{2T}\!\frac{\mathcal{E}_{\abs{I}}(\tau)}{\jb{\tau}}\,d\tau+S_{\abs{I}}(t)\Bigr).
\end{equation}

\subsubsection*{Error estimates for $\mathcal{V}_{\alpha 4}$}By the expression ~\eqref{eq:expansion_of_eqn_for_v} and Lemma ~\ref{lem:estimate_of_Aapp} it follows that 
\begin{align*}
\abs{\mathcal{V}_{\alpha 4}}&\lesssim\sum_{\abs{J}+\abs{K}+\abs{L}\leq\abs{I}}\bigl\lvert Z^L\widetilde{\chi}(\frac{t}{T})\bigr\rvert\Bigl(\bigl\lvert Z^K\phi_{app}\bigr\rvert\bigl\lvert Z^J (v_\alpha u)\bigr\rvert+\bigl\lvert Z^KA_{\alpha, app}\bigr\rvert\bigl\lvert Z^J \abs{u}^2\bigr\rvert\Bigr)\\
&\lesssim\sum_{\abs{J}+\abs{K}\leq\abs{I}}\frac{D_{\abs{K}}}{\jb{t+r}}\sum_{\abs{J'}+\abs{J''}\leq\abs{J}}\abs{Z^{J'}v_\alpha}\abs{Z^{J''}u}\\
&\qquad+\sum_{\abs{J}+\abs{K}\leq\abs{I}}\frac{D_{\abs{K}}}{\jb{t+r}}\sum_{\abs{J'}+\abs{J''}\leq\abs{J}}\abs{Z^{J'}u}\abs{Z^{J''}u}\\
&\qquad+\sum_{\abs{J}+\abs{K}\leq\abs{I}}\frac{D^2_{\abs{K}+2}}{\jb{t+r}\jb{(r-t)_+}^{2\gamma}}\ln\frac{\jb{t+r}}{\jb{t-r}}\sum_{\abs{J'}+\abs{J''}\leq\abs{J}}\abs{Z^{J'}u}\abs{Z^{J''}u}\\
&\lesssim\sum_{\abs{J}+\abs{K}\leq\abs{I}}\frac{\mathcal{E}_k(t)^{1/2}D_{\abs{K}+1}}{\jb{t+r}^2}(\abs{Z^J v_\alpha}+\abs{Z^J u})\\
&\qquad+\sum_{\abs{J}+\abs{K}\leq\abs{I}}\frac{D^2_{\abs{K}+2}\mathcal{E}_k(t)^{1/2}}{\jb{t+r}\jb{t-r}^{3/2}\jb{(r-t)_+}^{2\gamma}w^{1/2}}\abs{Z^J u}.
\end{align*}
Similarly we obtain that
\begin{equation}
\int_t^{2T}\!\int_{\Sigma_\tau}\! \biggl\lvert\frac{1}{r}\overline{K}_0(rZ^Iv_\alpha)\mathcal{V}_{\alpha 4}w\biggr\rvert\,dx\,d\tau\lesssim D_{\abs{I}+2}\mathcal{E}_k(t)^{1/2}\Bigl(\int_t^{2T}\!\frac{\mathcal{E}_{\abs{I}}(\tau)}{\jb{\tau}}\,d\tau+S_{\abs{I}}(t)\Bigr).
\end{equation}

\subsubsection*{Error estimates for $\mathcal{V}_{\alpha 5}$} We see that
\begin{align*}
\abs{\mathcal{V}_{\alpha 5}}\lesssim\sum_{\abs{J}+\abs{K}\leq\abs{I}}\bigl\lvert Z^L\widetilde{\chi}(\frac{t}{T})\bigr\rvert\sum_{\abs{J_1}+\abs{J_2}+\abs{J_3}\leq\abs{J}}\abs{Z^{J_1}v_\alpha}\abs{Z^{J_2}u}\abs{Z^{J_3}u}.
\end{align*}
Suppose $\abs{J_1}\leq\abs{J_2}\leq\abs{J_3}$, we can consider the $L^\infty$ estimates for $Z^{J_1}v_\alpha$ and $Z^{J_2}u$ and then the $L^2$ estimates for $Z^{J_3}$. Indeed, we have $\abs{J_1}\leq\abs{J_2}\leq\lfloor\abs{I}/3\rfloor\leq\lfloor k/3\rfloor\leq k-1$ for $k\leq 2$. Thus by Proposition ~\ref{prop:control_of_norm} it holds that
\[
\abs{Z^{J_1}v_\alpha}+\abs{Z^{J_2}u}\lesssim\frac{\sum_{\abs{K}\leq k+1}\norm{Z^Kv_\alpha}_{L^2(w)}+\norm{Z^Ku}_{L^2(w)}}{\jb{t+r}\jb{t-r}^{1/2}w^{1/2}}\lesssim\frac{\mathcal{E}_k(t)^{1/2}}{\jb{t+r}\jb{t-r}^{1/2}w^{1/2}}.
\]
Then we arrive at
\begin{equation}
\int_t^{2T}\!\int_{\Sigma_\tau}\! \biggl\lvert\frac{1}{r}\overline{K}_0(rZ^Iv_\alpha)\mathcal{V}_{\alpha 5}w\biggr\rvert\,dx\,d\tau\lesssim\mathcal{E}_k(t)\int_t^{2T}\!
\frac{\mathcal{E}_{\abs{I}}(\tau)}{\jb{\tau}}\,d\tau.	\end{equation}
\medskip

Finally, having the error estimates for $\mathcal{V}_{\alpha 3}$, $\mathcal{V}_{\alpha 4}$ and $\mathcal{V}_{\alpha 5}$ at our disposal, we are at the position to analyze the terms $\mathcal{V}_e^3$, $\mathcal{V}_e^4$ and $\mathcal{V}_e^5$ in ~\eqref{eq:shorthand_notations}. Correspondingly, we can obtain that
\begin{equation}
\begin{split}
\mathcal{V}_e^3&\lesssim(D_0+\mathcal{E}_1(t)^{1/2})D^2_{\abs{I}+2}\Bigl(\int_t^{2T}\!\frac{\mathcal{E}_{\abs{I}}(\tau)}{\jb{\tau}}\,d\tau+S_{\abs{I}}(t)\Bigr),\\
\mathcal{V}_e^4&\lesssim(D_0+\mathcal{E}_1(t)^{1/2})D_{\abs{I}+2}\mathcal{E}_k(t)^{1/2}\Bigl(\int_t^{2T}\!\frac{\mathcal{E}_{\abs{I}}(\tau)}{\jb{\tau}}\,d\tau+S_{\abs{I}}(t)\Bigr),\\
\mathcal{V}_e^5&\lesssim(D_0+\mathcal{E}_1(t)^{1/2})\mathcal{E}_k(t)\int_t^{2T}\!\frac{\mathcal{E}_{\abs{I}}(\tau)}{\jb{\tau}}\,d\tau.\\
\end{split}
\end{equation}

\subsection{Energy estimates for the remainder of the scalar field}In this subsection we consider the energy estimates for the remainder of the scalar field $u$, namely, we need to deal with the terms $\mathcal{U}_2$, $\mathcal{U}_3$, $\mathcal{U}_4$ and $\mathcal{U}_5$ in the equation ~\eqref{eq:expansion_of_eqn_for_u}.

\subsubsection*{Error estimates for $\mathcal{U}_3$, $\mathcal{U}_4$ and $\mathcal{U}_5$}Similar to $\mathcal{V}_{\alpha 3}$, $\mathcal{V}_{\alpha 4}$ and $\mathcal{V}_{\alpha 5}$, we have the following pointwise estimates
\begin{align*}
\abs{\mathcal{U}_3}&\lesssim\sum_{\abs{J}+\abs{K}\leq\abs{I}}\frac{D^2_{\abs{K}+1}}{\jb{t+r}^2}\bigl(\abs{Z^J v_\alpha}+\abs{Z^J u}\bigr)\\
&\qquad+\sum_{\abs{J}+\abs{K}\leq\abs{I}}\frac{D^3_{\abs{K}+2}}{\jb{t+r}\jb{t-r}\jb{(r-t)_+}^{2\gamma}}\bigl(\abs{Z^J v_\alpha}+\abs{Z^J u}\bigr),\\
\abs{\mathcal{U}_4}&\lesssim\sum_{\abs{J}+\abs{K}\leq\abs{I}}\frac{\mathcal{E}_k(t)^{1/2}D_{\abs{K}+1}}{\jb{t+r}^2}(\abs{Z^J v_\alpha}+\abs{Z^J u})\\
&\qquad+\sum_{\abs{J}+\abs{K}\leq\abs{I}}\frac{D^2_{\abs{K}+2}\mathcal{E}_k(t)^{1/2}}{\jb{t+r}\jb{t-r}^{3/2}\jb{(r-t)_+}^{2\gamma}w^{1/2}}\bigl(\abs{Z^J v_\alpha}+\abs{Z^J u}\bigr),\\
\abs{\mathcal{U}_5}&\lesssim\sum_{\abs{J}\leq\abs{I}}\frac{\mathcal{E}_k(t)}{\jb{t+r}^2}\bigl(\abs{Z^J v_\alpha}+\abs{Z^J u}\bigr).
\end{align*}
Therefore, it follows that 
\begin{equation}
\begin{split}
\int_t^{2T}\!\int_{\Sigma_\tau}\! \biggl\lvert\Re\bigl(\frac{1}{r}\overline{K}_0(rZ^Iu)\overline{\mathcal{U}_3}\bigr)w\biggr\rvert\,dx\,d\tau&\lesssim D^2_{\abs{I}+2}\Bigl(\int_t^{2T}\!\frac{\mathcal{E}_{\abs{I}}(\tau)}{\jb{\tau}}\,d\tau+S_{\abs{I}}(t)\Bigr),\\
\int_t^{2T}\!\int_{\Sigma_\tau}\! \biggl\lvert\Re\bigl(\frac{1}{r}\overline{K}_0(rZ^Iu)\overline{\mathcal{U}_4}\bigr)w\biggr\rvert\,dx\,d\tau&\lesssim D_{\abs{I}+2}\mathcal{E}_k(t)^{1/2}\Bigl(\int_t^{2T}\!\frac{\mathcal{E}_{\abs{I}}(\tau)}{\jb{\tau}}\,d\tau+S_{\abs{I}}(t)\Bigr),\\
\int_t^{2T}\!\int_{\Sigma_\tau}\! \biggl\lvert\Re\bigl(\frac{1}{r}\overline{K}_0(rZ^Iu)\overline{\mathcal{U}_5}\bigr)w\biggr\rvert\,dx\,d\tau&\lesssim
\mathcal{E}_k(t)\int_t^{2T}\!
\frac{\mathcal{E}_{\abs{I}}(\tau)}{\jb{\tau}}\,d\tau.
\end{split}
\end{equation}

\subsubsection*{Error estimates for $\mathcal{U}_2$}We decompose $\mathcal{U}_2$ into two components which are supported in the far interior and the extended exterior respectively. Specifically, let $\varphi$ be a smooth cutoff such that $\varphi(s)=1$ for $s\leq 1/4$, and $\varphi(s)=0$ for $s\geq 1/2$. We rewrite
\begin{align*}
\mathcal{U}_2&=\varphi(\frac{r}{t+1})\bigl(Z^I+\sum_{\abs{J}\leq\abs{I}-1}c_J^IZ^J\bigr)\biggl[\widetilde{\chi}(\frac{t}{T})\Big(-2iv^\alpha\partial_\alpha\phi_{app}-2iA^\alpha_{app}\partial_\alpha u-2iv^\alpha\partial_\alpha u\Big)\biggr]\\
&\qquad+(1-\varphi(\frac{r}{t+1}))\bigl(Z^I+\sum_{\abs{J}\leq\abs{I}-1}c_J^IZ^J\bigr)\biggl[\widetilde{\chi}(\frac{t}{T})\Big(-2iv^\alpha\partial_\alpha\phi_{app}-2iA^\alpha_{app}\partial_\alpha u-2iv^\alpha\partial_\alpha u\Big)\biggr]\\
&\equiv\mathcal{U}_2^{in}+\mathcal{U}_2^{ex}.
\end{align*}
First, we notice that $\jb{t-r}\sim\jb{t+r}$ in the support of $\varphi(r/(t+1))$, thus it is easy to deal with the term $\mathcal{U}_2^{in}$. Indeed, we have
\[
\int_t^{2T}\!\int_{\Sigma_\tau}\! \biggl\lvert\Re\bigl(\frac{1}{r}\overline{K}_0(rZ^Iu)\overline{\mathcal{U}_2^{in}}\bigr)w\biggr\rvert\,dx\,d\tau\lesssim(D_{\abs{I}+1}+\mathcal{E}_k(t)^{1/2})\int_t^{2T}\!\frac{\mathcal{E}_{\abs{I}}(\tau)}{\jb{\tau}}\,d\tau.
\]
Next, we focus on analyzing the term $\mathcal{U}_2^{ex}$. We notice that $r\sim\jb{t+r}$ in the support of $\mathcal{U}_2^{ex}$ and rewrite $\mathcal{U}_2^{ex}$ in the null frame as
\begin{equation}\label{eq:u2ex_in_null}
\begin{split}
\mathcal{U}_2^{ex}=(1-\varphi(\frac{r}{t+1}))\bigl(Z^I+\sum_{\abs{J}\leq\abs{I}-1}c_J^IZ^J\bigr)&\biggl[\widetilde{\chi}(\frac{t}{T})\Big(-2iv_L\partial_q\phi_{app}-2iA_{L, app}\partial_q u-2iv_L\partial_q u\Big)\\
&+\widetilde{\chi}(\frac{t}{T})\Big(iv_{\underline{L}}L\phi_{app}+iA_{\underline{L}, app}Lu+iv_{\underline{L}}Lu\Big)\\
&+\widetilde{\chi}(\frac{t}{T})\Big(-2iv_{e_B}e_B\phi_{app}-2iA_{e_B, app}e_Bu-2iv_{e_B}e_Bu\Big)\biggr]\\
&\equiv\mathcal{U}^{ex}_{2\underline{L}}+\mathcal{U}^{ex}_{2L}+\mathcal{U}^{ex}_{2B}.
\end{split}
\end{equation}
Owing to the estimate ~\eqref{eq:estimate_for_A_B} from Lemma ~\ref{lem:estimate_of_Aapp} we can control the term $\mathcal{U}^{ex}_{2B}$ as follows
\begin{align*}
&\int_t^{2T}\!\int_{\Sigma_\tau}\! \biggl\lvert\Re\bigl(\frac{1}{r}\overline{K}_0(rZ^Iu)\overline{\mathcal{U}^{ex}_{2B}}\bigr)w\biggr\rvert\,dx\,d\tau\\
&\qquad\lesssim(D_{\abs{I}+2}+\mathcal{E}_k(t)^{1/2})\int_t^{2T}\!\frac{\mathcal{E}_{\abs{I}}(\tau)}{\jb{\tau}}\,d\tau.
\end{align*}
For the term $\mathcal{U}^{ex}_{2L}$, since we only have $\abs{Z^JA_{\underline{L}, app}}\lesssim\jb{t+r}^{-1}\jb{(r-t)_+}^{-2\gamma}\ln\frac{\jb{t+r}}{\jb{t-r}}\leq \jb{t-r}^{-1}\jb{(r-t)_+}^{-2\gamma}$, we have to rewrite $L(Z^Ju)=r^{-1}L(rZ^Ju)-r^{-1}Z^Ju$ and then exploit the spacetime version of Hardy inequality to obtain that 
\[
\int_t^{2T}\!\int_{\Sigma_\tau}\! \biggl\lvert\Re\bigl(\frac{1}{r}\overline{K}_0(rZ^Iu)\overline{\mathcal{U}^{ex}_{2L	}}\bigr)w\biggr\rvert\,dx\,d\tau\lesssim(D_{\abs{I}+2}+\mathcal{E}_k(t)^{1/2})\Bigl(\int_t^{2T}\!\frac{\mathcal{E}_{\abs{I}}(\tau)}{\jb{\tau}}\,d\tau+S_{\abs{I}}(t)\Bigr).
\]
Now it remains to control the term $\mathcal{U}^{ex}_{2\underline{L}}$. We can easily see that 
\[
\int_t^{2T}\!\int_{\Sigma_\tau}\! \biggl\lvert\Re\bigl(\frac{1}{r}\jb{t-r}^2\underline{L}(rZ^Iu)\overline{\mathcal{U}^{ex}_{2\underline{L}}}\bigr)w\biggr\rvert\,dx\,d\tau\lesssim(D_{\abs{I}+2}+\mathcal{E}_k(t)^{1/2})\int_t^{2T}\!\frac{\mathcal{E}_{\abs{I}}(\tau)}{\jb{\tau}}\,d\tau.
\]
Therefore, it suffices to focus on the $(\jb{t+r}^2/r)L$ component. By Lemma ~\ref{lem:estimate_of_Aapp}, we see that $A_{L, app}=\frac{\BFq}{4\pi r}+\mathcal{R}_L$. In the presence of $\frac{\BFq}{4\pi r}$, we are unable to use the spacetime version of Hardy inequality to handle the lower order terms $\sum_{\abs{J}<\abs{I}}\frac{\BFq}{4\pi r}\partial_q Z^Ju$ as before, because we only have decay $\jb{t-r}^{-1}$
in the exterior which is not enough. In fact, the decay we expect in $q$ in the exterior should be $\jb{t-r}^{-\gamma-1/2}$. Therefore, we isolate this component and analyze it in a more delicate way. Namely, instead of a crude estimate, we will figure out what exactly the term $(Z^I+\sum_{\abs{J}<\abs{I}}c_J^IZ^J)[\widetilde{\chi}(t/T)\frac{\BFq}{4\pi r}\partial_q u]$ is.
\begin{lem}\label{lem:refinded_analysis}
Suppose $\Box f=\widetilde{\chi}(\frac{t}{T})\frac{\BFq}{4\pi r}\partial_q u$, then one has
\begin{equation}\Box Z^If=(Z^I+c_I^JZ^J)[\widetilde{\chi}(\frac{t}{T})\frac{\BFq}{4\pi r}\partial_q u]=\sum_{J+K=I}\frac{\BFq}{4\pi r}Z^K\widetilde{\chi}(\frac{t}{T})\partial_qZ^Ju+\mathcal{R}_{\abs{I}}
\end{equation}
where
\[
\abs{Z^L\mathcal{R}_{\abs{I}}}\lesssim\frac{\BFq}{4\pi r^2}\sum_{\abs{J'}\leq\abs{I}+\abs{L}}\abs{Z^{J'} u}.
\]
\end{lem}
\begin{proof}
First consider the case $\abs{I}$=1. If $Z=\partial_t$, we have
\[
\Box\partial_t f=\partial_t(\widetilde{\chi}(\frac{t}{T})\frac{\BFq}{4\pi r}\partial_q u)=\frac{\BFq}{4\pi r}(\partial_t\widetilde{\chi}(\frac{t}{T})\partial_qu+\widetilde{\chi}(\frac{t}{T})\partial_q\partial_tu).
\]
If $Z=\partial_i$,
\[
\Box\partial_i f=\partial_i(\widetilde{\chi}(\frac{t}{T})\frac{\BFq}{4\pi r}\partial_q u)=\frac{\BFq}{4\pi r}(\partial_i\widetilde{\chi}(\frac{t}{T})\partial_qu+\widetilde{\chi}(\frac{t}{T})\partial_q\partial_iu)+\widetilde{\chi}(\frac{t}{T})\frac{\BFq\omega_i^B}{8\pi r^2}e_B(u)-\widetilde{\chi}(\frac{t}{T})\frac{\BFq\omega_i}{4\pi r^2}\partial_q u.
\]
If $Z=\Omega_{ij}$,
\[
\Box\Omega_{ij} f=\Omega_{ij}(\widetilde{\chi}(\frac{t}{T})\frac{\BFq}{4\pi r}\partial_q u)=\frac{\BFq}{4\pi r}(\Omega_{ij}\widetilde{\chi}(\frac{t}{T})\partial_qu+\widetilde{\chi}(\frac{t}{T})\partial_q\Omega_{ij}u).
\]
If $Z=\Omega_{0i}$,
\[
\Box \Omega_{i0} f=\Omega_{0i}(\widetilde{\chi}(\frac{t}{T})\frac{\BFq}{4\pi r}\partial_q u)=\frac{\BFq}{4\pi r}(\Omega_{0i}\widetilde{\chi}(\frac{t}{T})\partial_qu+\widetilde{\chi}(\frac{t}{T})\partial_q\Omega_{0i}u)+\frac{\BFq\omega_i(r-t)}{4\pi r^2}\partial_q u+\frac{\omega_i^B\BFq(t+r)}{8\pi r^2}e_B(u).
\]
If $Z=S$,
\[
\Box S f=S(\widetilde{\chi}(\frac{t}{T})\frac{\BFq}{4\pi r}\partial_q u)+2\widetilde{\chi}(\frac{t}{T})\frac{\BFq}{4\pi r}\partial_q u=
\frac{\BFq}{4\pi r}S\widetilde{\chi}(\frac{t}{T})\partial_qu+\widetilde{\chi}(\frac{t}{T})\partial_qSu).
\]
The general case $\abs{I}>1$ follows from the fact $\Box ZZ^{\hat{I}}f=Z\Box Z^{\hat{I}}f+[\Box, Z] Z^{\hat{I}}f$ and the above calculation.
\end{proof}
By the above Lemma, we rewrite $\mathcal{U}^{ex}_{2\underline{L}}$
\begin{equation}\label{eq:expression_of_u2}
\mathcal{U}^{ex}_{2\underline{L}}=\bigl(1-\varphi(\frac{r}{t+1})\bigr)\Bigl(-2i\sum_{J+K=I}\frac{\BFq}{4\pi r}\partial_q\Bigl(Z^K\widetilde{\chi}(\frac{t}{T})Z^Ju\Bigr)-2i(\mathcal{R}_L+v_L)\partial_q\bigl(\widetilde{\chi}(\frac{t}{T})Z^I u\bigr)\Bigr)+\mathcal{R}_u.
\end{equation}
The error estimate for $\mathcal{R}_u$ is similar to that for $\mathcal{R}_v$ in ~\eqref{eq:treat_of_L}. In fact, we have
\[
\int_t^{2T}\!\int_{\Sigma_\tau}\! \biggl\lvert\Re\Bigl(\frac{\jb{t+r}^2}{r}L(rZ^Iu)\overline{\mathcal{R}_u}\Bigr)w\biggr\rvert\,dx\,d\tau\lesssim(D_{\abs{I}+2}+\mathcal{E}_k(t)^{1/2})\Bigl(\int_t^{2T}\!\frac{\mathcal{E}_{\abs{I}}(\tau)}{\jb{\tau}}\,d\tau+S_{\abs{I}}(t)\Bigr).
\]

\medskip

Next, we turn to the first term in ~\eqref{eq:expression_of_u2}. Analogous to the non tangential component of the first term in ~\eqref{eq:expression_of_v2}, we have
\begin{align*}
&\int_t^{2T}\!\int_{\Sigma_\tau}\!\bigl(1-\varphi(\frac{r}{t+1})\bigr)\Re\Bigl(\frac{\jb{t+r}^2}{r}L(rZ^Iu)2i\sum_{J+K=I}\frac{\BFq}{4\pi r}\partial_q\Bigl(Z^K\widetilde{\chi}(\frac{t}{T})\overline{Z^Ju}\Bigr)\Bigr)w\,dx\,d\tau\\
&\qquad+\int_t^{2T}\!\int_{\Sigma_\tau}\!\bigl(1-\varphi(\frac{r}{t+1})\bigr)\Re\Bigl(\frac{\jb{t+r}^2}{r}L(rZ^Iu)2i(\mathcal{R}_L+v_L)\partial_q\bigl(\widetilde{\chi}(\frac{t}{T})\overline{Z^I u}\bigr)\Bigr)w\,dx\,d\tau\\
&\qquad\qquad=\biggl(\int_{\Sigma_t}\!\bigl(1-\varphi(\frac{r}{t+1})\bigr)\Re\Bigl(\frac{\jb{t+r}^2}{r}L(rZ^Iu)i\sum_{J+K=I}\frac{\BFq}{4\pi r}Z^K\widetilde{\chi}(\frac{t}{T})\overline{Z^Ju}\Bigr)w\,dx\\
&\qquad\qquad\qquad\qquad+\int_{\Sigma_t}\!\bigl(1-\varphi(\frac{r}{t+1})\bigr)\Re\Bigl(\frac{\jb{t+r}^2}{r}L(rZ^Iu)i(\mathcal{R}_L+v_L)\widetilde{\chi}(\frac{t}{T})\overline{Z^I u}\Bigr)w\,dx\biggr)\\
&\qquad\qquad\qquad+\int_t^{2T}\!\int_{\Sigma_\tau}\!\partial_q\Bigl(w\bigl(\varphi(\frac{r}{t+1})-1\bigr)\Bigr)\Re\Bigl(\frac{\jb{t+r}^2}{r^2}L(rZ^Iu)2i\sum_{J+K=I}\frac{\BFq}{4\pi }Z^K\widetilde{\chi}(\frac{t}{T})\overline{Z^Ju}\Bigr)\,dx\,d\tau\\
&\qquad\qquad\qquad+\int_t^{2T}\!\int_{\Sigma_\tau}\!\partial_q\Bigl(w\bigl(\varphi(\frac{r}{t+1})-1\bigr)(r\mathcal{R}_L+rv_L)\Bigr)\Re\Bigl(\frac{\jb{t+r}^2}{r^2}L(rZ^Iu)2i\widetilde{\chi}(\frac{t}{T})\overline{Z^I u}\Bigr)\,dx\,d\tau\\
&\qquad\qquad\qquad+\biggl(\int_t^{2T}\!\int_{\Sigma_\tau}\!\bigl(1-\varphi(\frac{r}{t+1})\bigr)\Re\Bigl(\frac{\jb{t+r}^2}{r^2}(\Delta_{\omega}Z^Iu)i\sum_{J+K=I}\frac{\BFq}{4\pi r}Z^K\widetilde{\chi}(\frac{t}{T})\overline{Z^Ju}\Bigr)w\,dx\,d\tau\\
&\qquad\qquad\qquad\qquad+\int_t^{2T}\!\int_{\Sigma_\tau}\!\bigl(1-\varphi(\frac{r}{t+1})\bigr)\Re\Bigl(\frac{\jb{t+r}^2}{r^2}(\Delta_\omega Z^Iu)i(\mathcal{R}_L+v_L)\widetilde{\chi}(\frac{t}{T})\overline{Z^I u}\Bigr)w\,dx\,d\tau\biggr)\\
&\qquad\qquad\qquad+\biggl(\int_t^{2T}\!\int_{\Sigma_\tau}\!\bigl(\varphi(\frac{r}{t+1})-1\bigr)\Re\Bigl(\jb{t+r}^2(\Box Z^Iu)i\sum_{J+K=I}\frac{\BFq}{4\pi r}Z^K\widetilde{\chi}(\frac{t}{T})\overline{Z^Ju}\Bigr)w\,dx\,d\tau\\
&\qquad\qquad\qquad\qquad+\int_t^{2T}\!\int_{\Sigma_\tau}\!\bigl(\varphi(\frac{r}{t+1})-1\bigr)\Re\Bigl(\jb{t+r}^2(\Box Z^Iu)i(\mathcal{R}_L+v_L)\widetilde{\chi}(\frac{t}{T})\overline{Z^I u}\Bigr)w\,dx\,d\tau\biggr)\\
&\qquad\qquad\equiv\mathcal{U}_a+\mathcal{U}_b+\mathcal{U}_c+\mathcal{U}_d+\mathcal{U}_e.
\end{align*}
We note that the estimates for $\mathcal{U}_a$, $\mathcal{U}_b$, $\mathcal{U}_c$ and $\mathcal{U}_d$ are analogous to those for $\mathcal{V}_a$, $\mathcal{V}_b$, $\mathcal{V}_c$ and $\mathcal{V}_d$ respectively. Correspondingly, we have
\begin{equation}
\begin{split}
\abs{\mathcal{U}_a}&\lesssim (D_2^2+\mathcal{E}_1(t)^{1/2})\mathcal{E}_{\abs{I}}(t),\\
\abs{\mathcal{U}_b}&\lesssim D_1^2\Bigl(\int_t^{2T}\!\frac{\mathcal{E}_{\abs{I}}(\tau)}{\jb{\tau}}\,d\tau+S_{\abs{I}}(t)\Bigr),\\
\abs{\mathcal{U}_c}&\lesssim(D_2^2+\mathcal{E}_1(t)^{1/2})\int_t^{2T}\!\frac{\mathcal{E}_{\abs{I}}(\tau)}{\jb{\tau}}\,d\tau+(D_3^2+\mathcal{E}_2(t)^{1/2})S_{\abs{I}}(t),\\
\abs{\mathcal{U}_d}&\lesssim(D_3^2+\mathcal{E}_2(t)^{1/2})\int_t^{2T}\!\frac{\mathcal{E}_{\abs{I}}(\tau)}{\jb{\tau}}\,d\tau.
\end{split}
\end{equation}
Finally, for the term $\mathcal{U}_e$, we see that
\[
\mathcal{U}_e\equiv\mathcal{U}_e^1+\mathcal{U}_e^2+\mathcal{U}_e^3+\mathcal{U}_e^4+\mathcal{U}_e^5
\]
where
\begin{align*}
\mathcal{U}_e^j&=\int_t^{2T}\!\int_{\Sigma_\tau}\!\bigl(\varphi(\frac{r}{t+1})-1\bigr)\Re\biggl(\jb{t+r}^2\mathcal{U}_ji\Bigl(\sum_{J+K=I}\frac{\BFq}{4\pi r}Z^K\widetilde{\chi}(\frac{t}{T})\overline{Z^Ju}+(\mathcal{R}_L+v_L)\widetilde{\chi}(\frac{t}{T})\overline{Z^I u}\Bigr)\biggr)w\,dx\,d\tau.
\end{align*}
Following the way we control $\mathcal{V}_e^1$, $\mathcal{V}_e^3$, $\mathcal{V}_e^4$ and $\mathcal{V}_e^5$, we find that 
\begin{equation}
\begin{split}
\abs{\mathcal{U}_e^1}&\lesssim\frac{D_{\abs{I}+2}(D_2^2+\mathcal{E}_1(t)^{1/2})}{\jb{t}^{\frac{\gamma}{2}-\frac{\mu}{2}-\frac{1}{4}}}S_{\abs{I}}(t)^{1/2},\\
\abs{\mathcal{U}_e^3}&\lesssim (D_2^2+\mathcal{E}_1(t)^{1/2})D^2_{\abs{I}+2}\Bigl(\int_t^{2T}\!\frac{\mathcal{E}_{\abs{I}}(\tau)}{\jb{\tau}}\,d\tau+S_{\abs{I}}(t)\Bigr),\\
\abs{\mathcal{U}_e^4}&\lesssim (D_2^2+\mathcal{E}_1(t)^{1/2})D_{\abs{I}+2}\mathcal{E}_k(t)^{1/2}\Bigl(\int_t^{2T}\!\frac{\mathcal{E}_{\abs{I}}(\tau)}{\jb{\tau}}\,d\tau+S_{\abs{I}}(t)\Bigr),\\
\abs{\mathcal{U}_e^5}&\lesssim(D_2^2+\mathcal{E}_1(t)^{1/2})\mathcal{E}_k(t)\int_t^{2T}\!
\frac{\mathcal{E}_{\abs{I}}(\tau)}{\jb{\tau}}\,d\tau.
\end{split}
\end{equation}
It remains to bound $\mathcal{U}_e^2$. We define
\begin{align*}
\mathcal{N}_u&:=\int_t^{2T}\!\int_{\Sigma_\tau}\!\Re\biggl(2\jb{t+r}^2\Bigl(\sum_{J+K=I}\frac{\BFq}{4\pi r}\partial_q\bigl(Z^K\widetilde{\chi}(\frac{t}{T})Z^Ju\bigr)+(\mathcal{R}_L+v_L)\partial_q\bigl(\widetilde{\chi}(\frac{t}{T})Z^I u\bigr)\Bigr)\\
&\qquad\qquad\times\Bigl(\sum_{J+K=I}\frac{\BFq}{4\pi r}Z^K\widetilde{\chi}(\frac{t}{T})\overline{Z^Ju}+(\mathcal{R}_L+v_L)\widetilde{\chi}(\frac{t}{T})\overline{Z^I u}\Bigr)\biggr)\bigl(\varphi(\frac{r}{t+1})-1\bigr)w\,dx\,d\tau
\end{align*}
First, by ~\eqref{eq:u2ex_in_null} and ~\eqref{eq:expression_of_u2}, it holds that
\[
\abs{\mathcal{U}_e^2-\mathcal{N}_u}\lesssim(D_2^2+\mathcal{E}_1(t)^{1/2})(D_{\abs{I}+2}+\mathcal{E}_k(t)^{1/2})\Bigl(\int_t^{2T}\!\frac{\mathcal{E}_{\abs{I}}(\tau)}{\jb{\tau}}\,d\tau+S_{\abs{I}}(t)\Bigr).
\]
Next, we notice that we can express $\mathcal{N}_u$ as
\begin{align*}
\mathcal{N}_u&=\int_t^{2T}\!\int_{\Sigma_\tau}\!\jb{t+r}^2\bigl(\varphi(\frac{r}{t+1})-1\bigr)\biggl(\frac{\BFq^2}{16\pi^2 r^2}\partial_q\bigl\lvert\sum_{J+K=I}Z^K\widetilde{\chi}(\frac{t}{T})Z^Ju\bigr\rvert^2\\
&\qquad+\frac{\BFq}{2\pi r}(\mathcal{R}_L+v_L)\Re\Big(\partial_q\bigl(\widetilde{\chi}(\frac{t}{T})Z^I u\sum_{J+K=I}Z^K\widetilde{\chi}(\frac{t}{T})\overline{Z^Ju}\bigr)\Big)+(\mathcal{R}_L+v_L)^2\partial_q\bigl\lvert\widetilde{\chi}(\frac{t}{T})Z^Iu\bigr\rvert^2\biggr)w\,dx\,d\tau.
\end{align*}
Finally, by integration by parts in $q$ direction and the spacetime version of Hardy inequality, it follows that 
\[
\abs{\mathcal{N}_u}\lesssim(D_3^2+\mathcal{E}_2(t)^{1/2})^2\Bigl(\mathcal{E}_{\abs{I}}(t)+\int_t^{2T}\!\frac{\mathcal{E}_{\abs{I}}(\tau)}{\jb{\tau}}\,d\tau+S_{\abs{I}}(t)\Bigr).
\]

\subsection{Proof of Proposition ~\ref{prop:estimate_for_rmkg}} After the preparations in the previous subsections, we are now ready to complete the proof of Proposition ~\ref{prop:estimate_for_rmkg}. We set
\[B(t)=\mathcal{E}_k(t)+S_k(t) \quad \text{with}\quad k=N-3.\]
\begin{proof}[Proof of Proposition ~\ref{prop:estimate_for_rmkg}]
	Applying Corollary ~\ref{cor:energyestimate_for_u} and combining all the above estimates together, in view of the smallness assumption $D_N=\ep$ we conclude that 
	\begin{equation}\label{eq:bootstrap_ineq}
	B(t)\lesssim \frac{\ep}{\jb{t}^{\frac{\gamma}{2}-\frac{\mu}{2}-\frac{1}{4}}}B(t)^{1/2}+\ep B(t)+B(t)^{3/2}+B(t)^2+\Bigl(\ep+B(t)^{1/2}+B(t)+B(t)^{3/2}\Bigr)\int_t^{2T}\!\frac{B(\tau)}{\jb{\tau}
}\,d\tau.	
\end{equation}
Then a standard bootstrap argument yields that there exists a sufficiently small constant $\ep_0$ such that if $0<\ep\leq \ep_0$, we have
\[
B(t)\lesssim\frac{\ep^2}{\jb{t}^{\gamma-\mu-1/2}}.
\]
with $\mu<\gamma-1/2$. Then the estimates ~\eqref{eq:energy_in_prop} and ~\eqref{eq:decay_in_prop} asserted in Proposition ~\ref{prop:estimate_for_rmkg} follow from Proposition ~\ref{prop:control_of_norm} and Lemma ~\ref{lem:KS_with_weight} respectively. 
\end{proof}

\section{ MKG equations}\label{sec:5}

In Section ~\ref{sec:4}, we have obtained a sequence of solutions $(\phi_T, A_{\alpha T})=(\phi_{app}+u_T, A_{\alpha, app}+v_{\alpha T})$ to the equations ~\eqref{eq:rmkg_with_cutoff} with trivial data for $(u_T, v_T)$ at $t=2T$. In this section we show that the limit $(\phi, A_{\alpha})=\lim_{T\to\infty}(\phi_T, A_{\alpha T})$ exists in the norms defined by ~\eqref{eq:norm1} and the limit solution indeed solves the MKG equations ~\eqref{eq:mkg_in_A}.
\subsection{Existence of the limit}Let $T_2>T_1$. By Proposition ~\ref{prop:estimate_for_rmkg}, we have two corresponding solutions $(\phi_1, A_{\alpha,1})=(\phi_{T_1}, A_{\alpha T_1})$ and $(\phi_2, A_{\alpha, 2})=(\phi_{T_2}, A_{\alpha T_2})$. let $(u_j, v_{\alpha, j})=(\phi_j-\phi_{app}, A_{\alpha,j}-A_{\alpha, app})$ for $j=1,2$. We know that $(u_j, v_{\alpha, j})$ has vanishing data at $t=2T_j$ and $(\phi^j, A_{\alpha, j})$ solves the reduced MKG equations ~\eqref{eq:rmkg} for $t\leq T_j$. Our goal is to prove that the difference $(\hat{u}, \hat{v}_\alpha)=(u_2-u_1, v_{\alpha, 2}-v_{\alpha,1})$ tends to $0$ in the norm defined in ~\eqref{eq:norm1} as $T_2>T_1\to\infty$.

\medskip

First, we derive the equations which $(Z^I\hat{u}, Z^I\hat{v}_\alpha)$ satisfies for $t\leq T_1$ and $\abs{I}\leq N-5$ with $N\geq 7$
\begin{align}
\Box Z^I\hat{v}_\alpha&=(Z^I+\sum_{\abs{J}<\abs{I}-1}c_J^IZ^J)\Big[-\Im\bigl(\hat{u}\overline{\partial_\alpha\phi_2}\bigr)-\Im\big(\phi_1\overline{\partial_\alpha\hat{u}}\big)+\hat{v}_\alpha\abs{\phi_2}^2+A_{\alpha, 1}\hat{u}\overline{\phi_2}+A_{\alpha,1}\phi_1\overline{\hat{u}}\Big],\\
\Box Z^I\hat{u}&=(Z^I+\sum_{\abs{J}<\abs{I}-1}c_J^IZ^J)\Big[-2i\hat{v}^\alpha\partial_\alpha\phi_2-2iA^\alpha_1\partial_\alpha\hat{u}+\hat{v}^\alpha A_{\alpha, 2}\phi_2+A_1^\alpha\hat{v}_\alpha\phi_2+A_1^\alpha A_{\alpha, 1}\hat{u}\Big].
\end{align}
Next, we set
\[
\hat{B}(t)=\hat{\mathcal{E}}_k(t)+\hat{S}_k(t)\quad \text{for}\quad k=N-5 \quad\text{and}\quad t\leq T_1
\]
with
\begin{align*}
\hat{\mathcal{E}}_k(t)&=\sup_{t\leq \tau\leq T_1}\sum_{\abs{I}\leq k}E^w[Z^I\hat{v}_\alpha](\tau)+E^w[Z^I\hat{u}](\tau),\\
\hat{S}_k(t)&=\sum_{\abs{I}\leq k}\int_t^{T_1}\!\int_{\Sigma_\tau}\!\jb{t+r}^2\abs{\frac{1}{r}L(rZ^I\hat{u})}^2+\jb{t+r}^2\abs{\frac{1}{r}L(rZ^I\hat{v}_{\alpha})}^2\abs{w'}\,dx\,d\tau.
\end{align*}
with $w$ defined in Proposition ~\ref{prop:estimate_for_rmkg}. We also note that $(u_1, v_{\alpha, 1})$ and $(u_2, v_{\alpha, 2})$ satisfy ~\eqref{eq:energy_in_prop} and ~\eqref{eq:decay_in_prop}. Moreover, the constants in these estimates are independent of $T$.

Then we have for all $0\leq t\leq T_1$
\begin{equation}\label{eq:existence_of_limit}
	\hat{B}(t)\leq C\frac{\ep^2}{(1+T_1)^{\gamma-\mu-1/2}}+C\ep\int_t^{T_1}\!\frac{\hat{B}(\tau)}{1+\tau
}\,d\tau
\end{equation} 
where $C$ only depends on $\gamma, \mu, N$. Then by Gr\"onwall, it follows that 
\[
\hat{B}(t)\leq\frac{C\ep^2}{(1+T_1)^{\gamma-\mu-1/2}}\bigl(\frac{1+T_1}{1+t}\bigr)^{C\ep}.
\]We skip the proof of ~\eqref{eq:existence_of_limit} here as it is an analogy to the proof of ~\eqref{eq:bootstrap_ineq}. By Proposition ~\ref{prop:control_of_norm}, it holds that for $\abs{I}\leq N-4$
\[
\norm{Z^I\hat{u}}_{L^2(w)}+\norm{Z^I \hat{v}_\alpha}_{L^2(w)}\lesssim\frac{\ep}{(1+T_1)^{\frac{\gamma}{2}-\frac{\mu}{2}-\frac{1}{4}}}\bigl(\frac{1+T_1}{1+t}\bigr)^{\frac{C\ep}{2}}\to 0 
\]
as $T_2>T_1\to\infty$ if $\ep$ is sufficiently small. This proves the existence of the limit $(\phi, A_\alpha)=\lim_{T\to\infty}(\phi_T, A_{\alpha T})$. Moreover, by the weighted Klainerman-Sobolev inequality from Lemma~\ref{lem:KS_with_weight}, we can also obtain the following pointwise convergence for $\abs{I}\leq N-6$
\[
\sup_{t\leq T_1, x\in\BR^3}\abs{Z^I\hat{u}}+\abs{Z^I\hat{v}_\alpha}\lesssim\sup_{t\leq T_1, r\geq 0}\frac{\ep}{(1+T_1)^{\frac{\gamma}{2}-\frac{\mu}{2}-\frac{1}{4}}}\bigl(\frac{1+T_1}{1+t}\bigr)^{\frac{C\ep}{2}}\frac{1}{\jb{t+r}\jb{t-r}^{\frac{1}{2}}w^{\frac{1}{2}}}\to 0 
\]
as $T_2>T_1\to\infty$. We note that the estimates ~\eqref{eq:energy_in_prop} and ~\eqref{eq:decay_in_prop} also hold for $(Z^I\phi, Z^IA_\alpha)$ with $\abs{I}\leq N-4$ and $\abs{I}\leq N-6$ respectively.

By letting $T\to\infty$ in the equations ~\eqref{eq:rmkg_with_cutoff}, we arrive at the conclusion that $(\phi, A_\alpha)=\lim_{T\to\infty}(\phi_T, A_{\alpha T})$ is a solution to the reduced MKG equation ~\eqref{eq:rmkg}.
%The uniqueness of the solution follows from the arguments we used to establish the existence.

\subsection{Proof of Theorem ~\ref{thm:mainthm}}
In this subsection we show that the limit solution we obtained in the last subsection to the reduced MKG equations ~\eqref{eq:rmkg} solves the MKG equations ~\eqref{eq:mkg_in_A}  if the \emph{asymptotic Lorenz gauge condition} is satisfied.
\begin{proof}[Proof of Theorem ~\ref{thm:mainthm}]
Once we obtain a sequence of solutions $(\phi_T, A_{\alpha T})$ to the equations ~\eqref{eq:rmkg_with_cutoff}, we consider the term $\lambda_T=\partial^\alpha A_{\alpha T}$. 
Since $\lambda_T$ solves the equation ~\eqref{eq:eq_for_lorenz}
	\[
	\Box \lambda_T=\abs{\phi_T}^2\lambda_T,
	\]
	we apply Proposition ~\ref{prop:weightedenergyestimate} with $w\equiv 1$ to conclude that 
	\begin{align*}
\sum_{\abs{I}\leq 2}E[Z^I\lambda_T](t)&\leq	\sum_{\abs{I}\leq 2}E[Z^I\lambda_T](T)+\sum_{\abs{I}\leq 2}\int_t^T\!\int_{\Sigma_\tau}\bigl\lvert\frac{2}{r}\overline{K}_0(r\lambda_T)\Box Z^I \lambda_T\bigr\rvert\,dx\,d\tau\\
&\leq C\sum_{\abs{I}\leq 2}E[Z^I\lambda_T](T)+C\int_t^T\!\frac{\ep^2}{1+\tau}\sum_{\abs{I}\leq 2}E[Z^I\lambda_T](\tau)\,d\tau
	\end{align*}
	from which we can infer by Gr\"onwall that 
	\[
\sum_{\abs{I}\leq 2}E[Z^I\lambda_T](t)	\leq\sum_{\abs{I}\leq 2}CE[Z^I\lambda_T](T)
\bigl(\frac{1+T}{1+t}\bigr)^{C\ep^2}
\]
where $E[Z^I\lambda_T](t)$ denotes $E^w[Z^I\lambda_T](t)$ with $w\equiv 1$. Therefore, it remains to determine $\sum_{\abs{I}\leq 1}E[Z^I\lambda_T](T)$. By the 
\emph{asymptotic Lorenz gauge condition} and Proposition ~\ref{prop:asy_Lorenz}, we find that 
\[
\sum_{\abs{I}\leq 3}Z^I\partial^\alpha A_{\alpha, app}
=\mathcal{O}\Bigl(\frac{\ep}{\jb{t+r}^2}\ln\frac{\jb{t+r}}{\jb{t-r}}\chi\bigl(\frac{\jb{t-r}}{r}\bigr)\Bigr)+\mathcal{O}(\frac{\ep^2}{\jb{t+r}^m})
\]
for any $m>0$. This implies
\[
\sum_{\abs{I}\leq 2}E[Z^I\lambda_T](T)=\sum_{\abs{I}\leq 2}E[Z^I\partial_\alpha A_{\alpha, app}](T)+\sum_{\abs{I}\leq 2}E[Z^I\partial^\alpha v_{\alpha T}](T)\lesssim\frac{\ep^2}{\jb{T}^{\gamma-\mu-1/2}}.
\]
Then by Proposition ~\ref{prop:control_of_norm} and Klainerman-Sobolev inequality, it follows that uniformly in $(t,x)$
\[
\sum_{\abs{I}\leq 1}\abs{Z^I\lambda_T(t, x)}\to0
\]
as $T\to\infty$ if $\ep$ is sufficiently small. This finishes the proof of Theorem ~\ref{thm:mainthm}.
\end{proof}
\bibliographystyle{plain}
\bibliography{references}

\begin{thebibliography}{10}

\bibitem{AAG}
Y.~Angelopoulos, S.~Aretakis, and D.~Gajic.
\newblock A non-degenerate scattering theory for the wave equation on extremal
  {R}eissner-{N}ordstr\"{o}m.
\newblock {\em Comm. Math. Phys.}, 380(1):323--408, 2020.

\bibitem{BMS17}
L.~Bieri, S.~Miao, and S.~Shahshahani.
\newblock Asymptotic properties of solutions of the {M}axwell {K}lein {G}ordon
  equation with small data.
\newblock {\em Comm. Anal. Geom.}, 25(1):25--96, 2017.

\bibitem{CKL19}
T.~Candy, C.~Kauffman, and H.~Lindblad.
\newblock Asymptotic behavior of the {M}axwell-{K}lein-{G}ordon system.
\newblock {\em Comm. Math. Phys.}, 367(2):683--716, 2019.

\bibitem{C-BC81}
Y.~Choquet-Bruhat and D.~Christodoulou.
\newblock Existence of global solutions of the {Y}ang-{M}ills, {H}iggs and
  spinor field equations in {$3+1$} dimensions.
\newblock {\em Ann. Sci. \'{E}cole Norm. Sup. (4)}, 14(4):481--506 (1982),
  1981.

\bibitem{CK90}
D.~Christodoulou and S.~Klainerman.
\newblock Asymptotic properties of linear field equations in {M}inkowski space.
\newblock {\em Comm. Pure Appl. Math.}, 43(2):137--199, 1990.

\bibitem{DHR}
M.~Dafermos, G.~Holzegel, and I.~Rodnianski.
\newblock A scattering theory construction of dynamical vacuum black holes.
\newblock arXiv:1306.5534.

\bibitem{DRS}
M.~Dafermos, I.~Rodnianski, and Y.~Shlapentokh-Rothman.
\newblock A scattering theory for the wave equation on {K}err black hole
  exteriors.
\newblock {\em Ann. Sci. \'{E}c. Norm. Sup\'{e}r. (4)}, 51(2):371--486, 2018.

\bibitem{EM82}
D.M. Eardley and V.~Moncrief.
\newblock The global existence of {Y}ang-{M}ills-{H}iggs fields in
  {$4$}-dimensional {M}inkowski space. {I}. {L}ocal existence and smoothness
  properties.
\newblock {\em Comm. Math. Phys.}, 83(2):171--191, 1982.

\bibitem{EM82_2}
D.M. Eardley and V.~Moncrief.
\newblock The global existence of {Y}ang-{M}ills-{H}iggs fields in
  {$4$}-dimensional {M}inkowski space. {II}. {C}ompletion of proof.
\newblock {\em Comm. Math. Phys.}, 83(2):193--212, 1982.

\bibitem{FWY19}
A.~Fang, Q.~Wang, and S.~Yang.
\newblock Global solution for massive {M}axwell-{K}lein-{G}ordon equations with
  large {M}axwell field.
\newblock arXiv:1902.08927.

\bibitem{FST87}
M.~Flato, J.~Simon, and E.~Taflin.
\newblock On global solutions of the maxwell-dirac equations.
\newblock {\em Comm. Math. Phys.}, 112(1):21--49, 1987.

\bibitem{Horm87}
L.~H\"{o}rmander.
\newblock The lifespan of classical solutions of nonlinear hyperbolic
  equations.
\newblock In {\em Pseudodifferential operators ({O}berwolfach, 1986)}, volume
  1256 of {\em Lecture Notes in Math.}, pages 214--280. Springer, Berlin, 1987.

\bibitem{Horm97}
L.~H\"{o}rmander.
\newblock {\em Lectures on nonlinear hyperbolic differential equations},
  volume~26 of {\em Math\'{e}matiques \& Applications (Berlin) [Mathematics \&
  Applications]}.
\newblock Springer-Verlag, Berlin, 1997.

\bibitem{KM94}
S.~Klainerman and M.~Machedon.
\newblock On the {M}axwell-{K}lein-{G}ordon equation with finite energy.
\newblock {\em Duke Math. J.}, 74(1):19--44, 1994.

\bibitem{KWY20}
S.~Klainerman, Q.~Wang, and S.~Yang.
\newblock Global solution for massive {M}axwell-{K}lein-{G}ordon equations.
\newblock {\em Comm. Pure Appl. Math.}, 73(1):63--109, 2020.

\bibitem{Lind90}
H.~Lindblad.
\newblock Blow-up for solutions of {$\square u=|u|^p$} with small initial data.
\newblock {\em Comm. Partial Differential Equations}, 15(6):757--821, 1990.

\bibitem{Lind92}
H.~Lindblad.
\newblock Global solutions of nonlinear wave equations.
\newblock {\em Comm. Pure Appl. Math.}, 45(9):1063--1096, 1992.

\bibitem{Lind17}
H.~Lindblad.
\newblock On the asymptotic behavior of solutions to the {E}instein vacuum
  equations in wave coordinates.
\newblock {\em Comm. Math. Phys.}, 353(1):135--184, 2017.

\bibitem{LR03}
H.~Lindblad and I.~Rodnianski.
\newblock The weak null condition for {E}instein's equations.
\newblock {\em C. R. Math. Acad. Sci. Paris}, 336(11):901--906, 2003.

\bibitem{LR05}
H.~Lindblad and I.~Rodnianski.
\newblock Global existence for the {E}instein vacuum equations in wave
  coordinates.
\newblock {\em Comm. Math. Phys.}, 256(1):43--110, 2005.

\bibitem{LR10}
H.~Lindblad and I.~Rodnianski.
\newblock The global stability of {M}inkowski space-time in harmonic gauge.
\newblock {\em Ann. of Math. (2)}, 171(3):1401--1477, 2010.

\bibitem{LV17}
H.~Lindblad and V.~Schlue.
\newblock Scattering from infinity for semilinear models of {E}instein's
  equations satisfying the weak null condition.
\newblock arXiv:1711.00822.

\bibitem{LSo05}
H.~Lindblad and A.~Soffer.
\newblock A remark on long range scattering for the nonlinear {K}lein-{G}ordon
  equation.
\newblock {\em J. Hyperbolic Differ. Equ.}, 2(1):77--89, 2005.

\bibitem{LSo06}
H.~Lindblad and A.~Soffer.
\newblock Scattering and small data completeness for the critical nonlinear
  {S}chr\"{o}dinger equation.
\newblock {\em Nonlinearity}, 19(2):345--353, 2006.

\bibitem{LS06}
H.~Lindblad and J.~Sterbenz.
\newblock Global stability for charged-scalar fields on {M}inkowski space.
\newblock {\em IMRP Int. Math. Res. Pap.}, pages Art. ID 52976, 109, 2006.

\bibitem{M62}
C.S. Morawetz.
\newblock The limiting amplitude principle.
\newblock {\em Comm. Pure Appl. Math.}, 15:349--361, 1962.

\bibitem{Psa99}
M.~Psarelli.
\newblock Asymptotic behavior of the solutions of {M}axwell-{K}lein-{G}ordon
  field equations in {$4$}-dimensional {M}inkowski space.
\newblock {\em Comm. Partial Differential Equations}, 24(1-2):223--272, 1999.

\bibitem{Psa99_2}
M.~Psarelli.
\newblock Time decay of {M}axwell-{K}lein-{G}ordon equations in
  {$4$}-dimensional {M}inkowski space.
\newblock {\em Comm. Partial Differential Equations}, 24(1-2):273--282, 1999.

\bibitem{RS3}
M.~Reed and B.~Simon.
\newblock {\em Methods of modern mathematical physics. {III}}.
\newblock Academic Press [Harcourt Brace Jovanovich, Publishers], New
  York-London, 1979.
\newblock Scattering theory.

\bibitem{Shu91}
W-T. Shu.
\newblock Asymptotic properties of the solutions of linear and nonlinear spin
  field equations in {M}inkowski space.
\newblock {\em Comm. Math. Phys.}, 140(3):449--480, 1991.

\bibitem{Shu92}
W-T. Shu.
\newblock Global existence of {M}axwell-{H}iggs fields.
\newblock In {\em Nonlinear hyperbolic equations and field theory ({L}ake
  {C}omo, 1990)}, volume 253 of {\em Pitman Res. Notes Math. Ser.}, pages
  214--227. Longman Sci. Tech., Harlow, 1992.

\bibitem{Sogge08}
C.D. Sogge.
\newblock {\em Lectures on non-linear wave equations}.
\newblock International Press, Boston, MA, second edition, 2008.

\bibitem{Tau19}
G.~Taujanskas.
\newblock Conformal scattering for the {M}axwell-scalar field on de {S}itter
  space.
\newblock {\em J. Hyperbolic Differ. Equ.}, 16(3):743--791, 2019.

\bibitem{Wang13}
F.~Wang.
\newblock Radiation field for {E}instein vacuum equations with spacial
  dimension {$n\geq 4$}.
\newblock arXiv:1304.0407.

\bibitem{Wang10}
F.~Wang.
\newblock {\em Radiation field for {E}instein vacuum equations}.
\newblock ProQuest LLC, Ann Arbor, MI, 2010.
\newblock Thesis (Ph.D.)--Massachusetts Institute of Technology.

\bibitem{Yafaev1}
D.R. Yafaev.
\newblock {\em Mathematical scattering theory}, volume 105 of {\em Translations
  of Mathematical Monographs}.
\newblock American Mathematical Society, Providence, RI, 1992.
\newblock General theory, Translated from the Russian by J. R. Schulenberger.

\bibitem{Yafaev2}
D.R. Yafaev.
\newblock {\em Mathematical scattering theory}, volume 158 of {\em Mathematical
  Surveys and Monographs}.
\newblock American Mathematical Society, Providence, RI, 2010.
\newblock Analytic theory.

\bibitem{Yang16}
S.~Yang.
\newblock Decay of solutions of {M}axwell-{K}lein-{G}ordon equations with
  arbitrary {M}axwell field.
\newblock {\em Anal. PDE}, 9(8):1829--1902, 2016.

\bibitem{Yang18}
S.~Yang.
\newblock On the global behavior of solutions of the {M}axwell-{K}lein-{G}ordon
  equations.
\newblock {\em Adv. Math.}, 326:490--520, 2018.

\bibitem{YangYu19}
S.~Yang and P.~Yu.
\newblock On global dynamics of the {M}axwell-{K}lein-{G}ordon equations.
\newblock {\em Camb. J. Math.}, 7(4):365--467, 2019.

\bibitem{Yu20}
D.~Yu.
\newblock Modified wave operator for a scalar quasilinear wave equation
  satisfying the weak null condition.
\newblock arXiv:2002.05355.

\end{thebibliography}

\end{document}